\documentclass[11pt, fullpage]{article}
\usepackage[a4paper, margin={1.2in}]{geometry}

\usepackage[utf8]{inputenc}
\usepackage[T1]{fontenc}
\usepackage{tgpagella}
\usepackage{amssymb}
\usepackage{amsfonts}
\usepackage{amsmath}
\usepackage{amsthm}
\usepackage{amscd}
\usepackage{bussproofs}
\usepackage[all]{xy}
\usepackage{comment}
\usepackage{tikz-cd}
\usepackage{amssymb}
\usepackage{amsfonts}
\usepackage{amsmath}
\usepackage{enumitem}

\newtheorem{tw}{Theorem} 
\newtheorem{wniosek}[tw]{Corollary} 
\newtheorem{lemat}[tw]{Lemma} 
\newtheorem{stwierdzenie}[tw]{Proposition} 

\theoremstyle{definition}
\newtheorem{definicja}[tw]{Definition} 
\newtheorem{konwencja}[tw]{Convention} 

\newtheorem{przyklad}[tw]{Example} 
\newtheorem{uwaga}[tw]{Remark} 

\newtheorem{obserwacja}[tw]{Observation}

\newcommand{\Term}[1]{\textnormal{\textsf{Term}}_{#1}}
\newcommand{\Sent}{\textnormal{\textsf{Sent}}}

\newcommand{\val}[1]{{#1}^{\circ}}
\newcommand{\len}{\textnormal{\textsf{len}}}
\newcommand{\vrbl}{\textnormal{\textsf{Var}}}
\newcommand{\TermSeq}{\textnormal{\textsf{TermSeq}}}
\newcommand{\ClTermSeq}{\textnormal{\textsf{ClTermSeq}}}
\newcommand{\term}{\textnormal{\textsf{Term}}}
\newcommand{\form}{\textnormal{\textsf{Form}}}
\newcommand{\Form}[1]{\textnormal{\textsf{Form}}_{#1}}

\newcommand{\Prov}{\textnormal{\textsf{Pr}}}
\newcommand{\valt}[2]{#1^{#2}}

\newcommand{\ClTerm}{\textnormal{\textsf{ClTerm}}}
\newcommand{\Ass}{\textnormal{\textsf{Asn}}}
\newcommand{\FV}{\textnormal{\textsf{FV}}}
\newcommand{\FVSeq}{\textnormal{\textsf{FVSeq}}}
\newcommand{\Proof}{\textnormal{\textsf{Proof}}}
\newcommand{\Lang}{\mathcal{L}}

\newcommand{\LPA}{\Lang_{\PA}}

\newcommand{\VarSeq}{\textnormal{\textsf{VarSeq}}}

\newcommand{\num}[1]{\underline{#1}}

	\newcommand{\imsubf}{\triangleleft}
\newcommand{\qcr}[1]{\ulcorner #1 \urcorner}
\newcommand{\EV}{\textnormal{\textsf{EV}}}

\newcommand{\RT}{\textnormal{\textsf{RT}}}
\newcommand{\CT}{\textnormal{\textsf{CT}}}
\newcommand{\PA}{\textnormal{\textsf{PA}}}

\newcommand{\CS}{\textnormal{\textsf{CS}}}
\newcommand{\WKL}{\textnormal{\textsf{WKL}}}

\newcommand{\ACT}{\textnormal{\textsf{ACT}}}
\newcommand{\KF}{\textnormal{\textsf{KF}}}
\newcommand{\FS}{\textnormal{\textsf{FS}}}
\newcommand{\ISigma}{\textnormal{\textsf{I}}\Sigma}
\newcommand{\IDelta}{\textnormal{\textsf{I}}\Delta}
\newcommand{\B}{\textnormal{\textsf{B}}}
\newcommand{\Th}{\mathcal{T}}
\newcommand{\Q}{\textnormal{\textsf{Q}}}
\newcommand{\Frag}{\textnormal{\textsf{P}}}
\newcommand{\Exp}{\textnormal{\textsf{Exp}}}
\newcommand{\ACA}{\textnormal{\textsf{ACA}}}
\newcommand{\PT}{\textnormal{\textsf{PT}}}

\newcommand{\md}[1]{\mathcal{#1}}
\newcommand{\df}[1]{\emph{#1}}
\newcommand{\set}[2]{\left\{#1 \ \ | \ \ #2 \right\}}
\newcommand{\pair}[2]{\langle #1, #2\rangle}
\newcommand{\psubf}{\triangleleft}

\newcommand{\rank}{\textsf{ rank}}
\newcommand{\ElDiag}{\textsf{ElDiag}}

\newcommand{\was}[1]{\{#1\}}
\newcommand{\restr}[1]{\upharpoonright_{#1}}

\newcommand{\norm}[2]{\parallel #1\parallel_{#2}}
\newcommand{\Sat}{\textnormal{\textsf{Sat}}}

\newcommand{\Con}{\textnormal{\textsf{Con}}}
\newcommand{\Tr}{\textnormal{\textsf{Tr}}}

\newcommand{\dpt}{\textnormal{\textsf{dp}}}

\newcommand{\Ind}{\textnormal{\textsf{Ind}}}
\newcommand{\enum}{\textsf{enum}}
\newcommand{\M}{\md{M}}

\newcommand{\nat}{\mathbb{N}}
\newcommand{\compl}[1]{\widehat{#1}}

\newcommand{\tuple}[1]{\langle #1 \rangle}

\title{Truth and Feasible Reducibility}
\author{Ali Enayat, Mateusz Łełyk, Bartosz Wcisło}

\begin{document}
	
	\maketitle
		\begin{abstract}
	    Let $\mathcal{T}$ be any of the three canonical truth theories $\CT^-$ (Compositional truth without extra induction), $\FS^-$ (Friedman--Sheard truth without extra induction), and $\KF^-$ (Kripke--Feferman truth without extra induction), where the base theory of $\mathcal{T}$ is $\PA$ (Peano arithmetic). We show that $\mathcal{T}$ is  \textit{feasibly reducible to} $\PA$, i.e., there is a polynomial time computable function $f$ such that for any proof $\pi $ of an arithmetical sentence $\phi $ in $\mathcal{T}$, $f(\pi )$ is a proof of $\phi $ in $\PA$.  In particular, $\mathcal{T}$ has at most polynomial speed-up over $\PA$, in sharp contrast to the situation for $\mathcal{T}[\B]$ for \textit{finitely axiomatizable} base theories $\B$.
	\end{abstract}
	\tableofcontents
\section{Introduction }

One of the celebrated results in the area of axiomatic theories of truth is
the Krajewski-Kotlarski-Lachlan (KKL) theorem \cite{Kotlarski-Krajewski-Lachlan} that asserts that
every countable recursively saturated model of $\PA$  (Peano
arithmetic) is expandable to a model of $\CT^{-}[\PA]$
(compositional truth over $\PA$ with no extra induction\footnote{%
This theory is referred to as $\CT\! \upharpoonright $
in \cite{halbach}, $\CT^{-}$ in \cite{Cieslinski_Book}, and $\PA^{\mathsf{FT}}$ in \cite{enayat-visser}.}). The KKL theorem is an overtly
model-theoretic result, but it is well-known that it is equivalent to the
conservativity of $\CT^{-}[\PA]$ over $\PA$.%
\footnote{%
This equivalence follows from two key facts: (1) every countable consistent theory has
a countable recursively saturated model, and (2) countable recursively
saturated models are resplendent, both of which can be verified in the
subsystem $\ACA_{0}$ of second order arithmetic.} Recent proofs of the KKL theorem given by Enayat and Visser \cite{enayat-visser} (using model-theoretic techniques) and Leigh \cite{leigh} (using proof-theoretic
machinery) show that $\CT^{-}[\B]$ is conservative over $\B$ for every "base theory" $\B$ (i.e., a theory $\B$ that supports a modicum of coding machinery for handling elementary syntax). Leigh's proof makes it clear that if $\B$ is a base theory with a computable set of axioms, then $\CT^{-}[\B]$ is \textit{proof-theoretically reducible} to $\B$, and in particular, there is a primitive recursive function $f$ such that for any proof $\pi$ of a sentence $\phi $ in $\CT^{-}[\B]$, where $\phi $ is a sentence in the language of $\B$, $f(\pi )$ is a proof of $\phi $ in $\B$. Indeed, Leigh's "reducing function" $f$ is readily seen to be a provably total function of the fragment of $\mathsf{PRA}$\ (Primitive Recursive Arithmetic) commonly known as $\IDelta _{0}+\mathsf{Supexp.}$\footnote{$\mathsf{Supexp}$ asserts the totality of the superexponential function $\mathsf{Supexp}(n,x)$%
, with $\mathsf{Supexp}(0,x)=x$ and $\mathsf{Supexp}(n+1,x)=2^{\mathsf{Supexp%
}(n,x)}.$ Leigh \cite{leigh} refers to this function as hyper-exponentiaton.}%

The main result of this paper shows that $\CT^{-}[\PA]$ is
\textit{feasibly} \textit{reducible to }$\PA$, i.e., there
is a polynomial-time computable function $f$ such that for any proof $\pi $ of
an arithmetical sentence $\phi $ in $\CT^{-}[\PA]$, $%
f(\pi )$ is a proof of $\phi $ in $\mathsf{PA}$. The feasible
reducibility of $\CT^{-}[\PA]$ to $\PA$ readily
implies that $\CT^{-}[\PA]$ does not exhibit significant 
speed-up
over $\PA$, i.e., there is a polynomial function $p(x)$ such that
for any arithmetical sentence $\phi $, if $\phi $ is provable in $%
\CT^{-}[\PA]$ by a proof of length $n$ (in some standard
proof system\footnote{
The choice of the "standard proof system" is immaterial since it is
well-known that any two such systems polynomially simulate each other
\cite{Pudlak}.}), then $\phi $ is provable in $\mathrm{PA}$ by a proof of length $p(n)$.  This
solves a problem posed by Enayat in 2012 \cite{enayat-question}.

The absence of significant speed up of $\CT^{-}[\PA]$ over $\PA$ implied by the feasible reducibility of $\CT^{-}[\PA]$ to $\PA$ exhibits a dramatic difference
between $\CT^{-}[\PA]$ and $\CT^{-}[\B]$ for
finitely axiomatized base theories $\B$, since as shown by Fischer \cite{Fischer-speed}, $\CT^{-}[\B]$ has superexponential speed-up over
$\B$ for finitely axiomatized base theories $\B$, and
therefore, $\CT^{-}[\B]$ is not feasibly reducible to $%
\B$ for finitely axiomatized base theories $\B$. It is also
known that $\CT^{-}[\PA]+\mathsf{Int}$ (where $\mathsf{Int}$
is the axiom of internal induction) is conservative over $\PA$
(\cite{Kotlarski-Krajewski-Lachlan}, \cite{leigh}) but not feasibly reducible to $\PA$ since it has
superexponential speed-up over $\PA$ \cite{Fischer-speed}.

Our proof of the feasible reduction of $\CT^{-}[\PA]$ to $\PA$ includes the verification that $\PA$ proves the formal
consistency of every finite subtheory of $\CT^{-}[\PA]$,
thereby establishing that $\CT^{-}[\PA]$ is a reflexive
theory. This result follows from Leigh's work \cite{leigh}; and was also
established by Enayat and Visser (unpublished) with
help of the "low
basis theorem" of computability theory to arithmetize their model-theoretic
proof of conservativity of $\CT^{-}[\PA]$ over $\PA$. The proof presented here, however, is based on a simpler arithmetization of the Enayat--Visser construction and does not appeal to the low basis theorem; the syntactic analysis of this arithmetization forms one of the main ingredients of the proof of our main result. 

We also employ the machinery developed for the proof of our main result to analyse two other prominent theories of truth, namely $\FS^{-}[\PA]$ (Friedman--Sheard theory of truth
over $\PA$, with no extra induction), and $\KF^{-}[\PA]$ (Kripke--Feferman theory of truth over $\PA$ with no extra
induction). More specifically, we show that $\FS^{-}[\PA]$
and $\KF^{-}[\PA]$ are both reflexive and feasibly reducible
to $\mathrm{PA}$. These results, in turn, show that both $\FS^{-}[\PA]$
and $\KF^{-}[\PA]$ are interpretable in $\mathsf{PA}$ and
have at most polynomial speed-up over $\mathsf{PA}$.

A word about the organization of the paper is in order. Section 2 deals with arithmetical preliminaries and technical machinery that will be employed for establishing our principal results. Section 3 presents basic definitions and facts about the truth theories $\CT^{-}[\PA]$, $\KF^{-}[\PA]$, and $\FS^{-}[\PA]$, including their conservativity over $\PA$. The main results of the paper are contained in Section 4, which contains the proofs of feasible reduction of $\CT^{-}[\PA]$, $\KF^{-}[\PA]$, and $\FS^{-}[\PA]$ to $\PA$; these proofs should be viewed as refined arithmetizations of the conservativity proofs presented in Section 3.3. The last subsection of Section 4, on the other hand, spells out the interpretability-theoretic ramifications of our work. Section 5 collects some open questions; and the Appendix (Section 6) consists of routine-but-technical proofs of certain results employed in the body of the paper.

\section{Setting the stage: arithmetical machinery}

This section discusses basic notions and fundamental machinery that can be generally described as \textit{refined arithmetization} of certain parts of proof theory and model theory that play a key role in the statements and proofs of our main results in Section \ref{The main act}. Note, however that the material in Subsection \ref{sub_interpretability_and_speedup} will be only employed in Subsection \ref{sub_speed_up_via_FACT}.

\subsection{Arithmetized syntax}

In this paper, $\PA$ denotes the theory using  $\was{0, S, + , \times}$ as non-logical function symbols, whose axioms consist of the axioms of Robinson's Arithmetic $\Q$ together with the usual induction scheme for the whole language $\LPA$ of $\PA$. Its intended model are natural numbers $\mathbb{N}$ with addition, multiplication and successor functions. We will also denote $\mathbb{N}$ by $\omega$, typically when treating it as set of indices for some construction. Sometimes in this paper, we will be referring to $\mathbb{N}$ or $\omega$ when working in $\PA$. Then these symbols simply refer to the whole universe. This should not lead to any confusion.

Crucially for our purposes, $\PA$ is capable of representing syntax. This means that in $\PA$, one can employ recursion to define notions such as "term," "formula" or "proof in $\PA$" similarly to how these notions are defined in Zermelo--Fr\"ankel set theory. This is a standard topic, covered, e.g., in \cite{Kaye} or \cite{HajekPudlak}. 

 Every sequence $(w_0,\ldots, w_n)\in\was{0,1}^{< \omega}$ is represented by a number 
	\[m = \sum_{i=0}^n2^i(w_i + 1).\]
  Each formula is represented as a $0,1$-string and then coded as a number. If $s$ is a binary string, then $|s|$ denotes its length.

Throughout the paper we will use certain formulae to represent various syntactic and technical notions. For the convenience of our reader, we gather here all the notation which might be possibly confusing.

\begin{definicja}[Arithmetized syntax] \label{konw_technical definitions}
\begin{enumerate}
\item $\len(s)=x$ asserts: "$s$ is a sequence and its length is equal to $x$." 
\item $\term_{\Lang}(x)$ asserts: "$x$ is a term of a language $\Lang$." For instance, $\term_{\LPA}(x)$ asserts that $x$ is a (code of an) arithmetical term.

\item $\ClTerm_{\Lang}(x)$ asserts: "$x$ is a closed term of a language $\Lang$," i.e. a term without free variables.
\item $\TermSeq_{\Lang}(x)$ asserts: "$x$ is a sequence of terms of the language $\Lang$." 
\item $\ClTermSeq_{\Lang}(x)$ asserts: "$x$ is a sequence of closed terms of the language $\Lang$."
\item $\val{x}=y$ asserts: "$\ClTerm_{\Lang}(x)$ and $y$ is the value of the term $x$." For instance, \begin{displaymath}
\val{\qcr{1+((1+1)+0)}} = 3,
\end{displaymath}
holds provably in $\PA$ ($\qcr{t}$ stands for the number coding the term $t$).
\item $\vrbl(x)$ asserts: "$x$ is a variable." For instance, $\vrbl(17)$ means that $17$ is a code of a variable in the coded language. Since without loss of generality we can assume that all first-order languages have the same set of variables, we omit the reference to a specific language in the subscript.
\item $\form_{\Lang}(x)$ asserts: "$x$ is a formula of the language $\Lang$."
\item $\FV(x,y)$ asserts: "$y$ is a free variable of $x$," where $x$ is either a term or a formula.
\item $\form^{\leq 1}_{\Lang}(x)$ asserts: "$x$ is a formula of the language $\Lang$ with at most one free variable", and $\form^{1}_{\Lang}(x)$ asserts: "$x$ is a formula of the language $\Lang$ with exactly one free variable."
\item $\Sent_{\Lang}(x)$ asserts: "$x$ is a sentence of the language $\Lang$."

\item $\FVSeq(x,y)$ asserts: "$y$ is a (coded) sequence whose elements are (some) free variables of $x$," where $x$ is either a term or a formula.
\item $\alpha$ is an assignment for a formula or a term $\phi$ if $\alpha$ is a function whose domain includes the free variables of $\phi$. The formula $\Ass(x,y)$ asserts: "$y$ is an assignment for $x$," where $x$ is a term or a formula. We will often denote it with $y \in \Ass(x)$. If $s$ is a coded set or sequence, we will write $y \in \Ass(s)$ to denote that $y$ is an assignment for all elements of $s$. We will sometimes also write $\Ass(x_1, \ldots, x_n, \alpha)$ or $\alpha \in \Ass(x_1, \ldots, x_n)$ meaning $\alpha \in \Ass(s)$, where $s = \tuple{x_1, \ldots, x_n}$
\item $\beta \sim_v \alpha$ asserts: "$\alpha$ and $\beta$ are assignments, $v$ is a variable, and $\alpha(w) = \beta(w)$ for all variables $w$, possibly except for $v$ which belongs to the domain of $\beta$ (and not necessarily to the domain of $\alpha$)." 
\end{enumerate}
\end{definicja}

The reader could expect to see in this list certain other predicates such as $\Proof$ or $\Con$. Since we will need some more precise information about these formulae and their length, they will be only introduced in Subsection \ref{sub_lengths_of_proofs}.

Next we introduce numerals. For our purposes the numeral $\num{x}$ for a number $x$ will not be a code of $S\ldots S0$ iterated $x$ times, since this is not an efficient notation when it comes to writing short formulae and proofs. Our numerals will be handled via binary expansions.

\begin{definicja}[Numerals] \label{defin:arithm}
 For any natural number $n$, $\num{n}$ denotes the binary expansion of $n$ written as a term of $\Th$. More precisely: let $n = \sum_{i\leq k} \varepsilon_i2^i$, where $\varepsilon_i\in\was{0,1}$. We define:
	\begin{equation*}
	\num{n} = \num{\varepsilon_0} + \num{2}\times (\num{\varepsilon_1} + \num{2}\times (\ldots \varepsilon_{k-1} + \num{2}\times\num{\varepsilon_k})\ldots)
	\end{equation*}
	where $\num{0} = 0$, $\num{1} = S(0)$ and $\num{2} = \num{1} + \num{1}$. Thus it takes $O(\log n)$ symbols to represent $n$.
\end{definicja}

Let us introduce one more definition:

\begin{definicja}[Substitutions] \label{defi_substitutions}
\begin{enumerate}
\item Let $\phi(v_1,\ldots,v_{n})$ be a formula with $n$ free variables shown and let $\alpha$ be an assingment for $\phi$. By $\phi[\alpha]$ we denote the formula in which the numeral (in the sense of Definition \ref{defin:arithm}) denoting $\alpha(v_i)$ is substituted for the variable $v_i$. 
\item Similarly, if $t \in \term_{\Lang}$ and $\alpha \in \Ass(t)$, then by $t[\alpha]$ we mean the closed term obtained by substituting the numeral $\alpha(v)$ for each free variable $v$ in $t$.
\item If $t \in \term_{\Lang }$ and $\alpha \in \Ass(t)$, then by $\valt{t}{\alpha}$ we mean $\val{t[\alpha]}.$ Notice that if $v$ is a variable and $\alpha \in \Ass(v)$, then $\valt{v}{\alpha} = \alpha(v)$ provably in $\PA$.
\item If $x \in \form_{\Lang}$ for some language $\Lang$, $v \in \vrbl$ and $t \in \term_{\Lang}$, then $x[t/v]=y$ is an arithmetical formula which asserts "$y$ is the effect of substituting in the formula $x$ the term $t$ for every occurrence of the variable $v$." 
\end{enumerate}
\end{definicja}

We adopt a few conventions as to how we will be dealing with formalized syntactic notions.

\begin{konwencja} \label{konwencja_slopp_syntax}
\begin{enumerate}
	\item If $\phi$ is a standard formula or a standard term, we will sometimes denote the corresponding code by $\qcr{\phi}$ to prevent ambiguity, but most of the time we skip the corners to lighten the notation.
	\item We can clearly code a formula $\phi$ as a binary string. Then by $|\phi|$ we mean the length of this string. Note that $\qcr{\phi}$ is of size exponential in $|\phi|$ and $\num{\qcr{\phi}}$ is of size linear in $|\phi|$. We deal with proofs in $\Th$ in a similar fashion.
	\item Since we will sometimes skip the corners denoting formulae, $\num{\qcr{\phi}}$ is most of the time denoted by $\num{\phi}$ for a standard formula $\phi$. 
	\item We will sometimes use the formulae defining syntactic notions as if they were denoting sets. For example, we will sometimes write "$x \in \form_{\LPA}$" rather than "$\form_{\LPA}(x)$."
    \item We will use provably functional formulae such as $\num{x}$, $\val{x}$, or $x[t/v]$ as if they were terms. 
    \item For better readability we will sometimes skip formulae denoting syntactic operations and write the effect of the operations instead. Thus, for example, we will write $T(\neg \phi)$ to denote "There exists $\psi$ which is the negation of the sentence $\phi$ and $T(\psi)$." 
\end{enumerate}
\end{konwencja}

 \subsection{Arithmetized model theory} \label{sub_arithmetised_model_theory}

Peano arithmetic is capable of accommodating a substantial part of the model theory of countable structures. We will make constant use of this fact throughout the whole paper. This subsection briefly introduces the reader to this topic. 
The rough convention is as follows: a \df{theory} is a definable set of sentences. If $\phi$ is a formula which defines a set of (codes of) sentences, then we call \emph{that formula} a theory. 

Models come in two kinds. By a \df{full model} $\md{M}$ we mean the elementary diagram of that model (or, actually, a formula defining the elementary diagram). It is given as a complete Henkinized theory. By a \df{model} $\md{M}$, we mean a formula defining its domain and some relations on that domain (this does not mean that we only deal with models of relational languages, but rather, we construe the denotations of function and constant symbols as relations).

\begin{definicja} \label{model_theory}
A formula $\phi$ defines a \df{theory} in a language $\Lang$ if for all $x$, $\phi(x) \rightarrow \bigl(x \in \Sent_{\Lang}\bigr)$
holds. By a \df{full model} of a theory $\Th$, we mean a theory $\Th' \supseteq \Th$ in a language $\Lang'$ expanding $\Lang$ with some constants (possibly trivially) such that:
\begin{itemize}
\item $\Th'$ is complete and consistent, so for any $\phi$, $\phi \in \Th'$ if and only if $\neg \phi \notin \Th'$.
\item $\Th'$ has all existential statements witnessed by constants which means that if $\exists x \phi(x) \in \Th'$, then for some constant $c$, $\phi(c) \in \Th'$.
\end{itemize}

By a \df{model} of a language $\Lang$ (or simply an $\Lang$-structure), we mean a formula $\md{M}$ which defines a set of (coded) sequences such that if $\md{M}(s)$, then:
\begin{itemize}
\item $s(0)$ is either a symbol of the language $\Lang$ or some fixed element $d$ which is not a symbol of $\Lang$. 
\item If $s(0)$ is a relation symbol of $\Lang$, then the length of $s$ is the arity of $s(0)$ plus one.
\item If $s(0)$ is a function symbol of $\Lang$, then the length of $s$ is the arity of $s(0)$ plus two. We treat constants as functions of arity zero.
\item If $s(0)$ is the fixed element $d$, then $s$ has length two.
\item If $a = s(n)$ for some $n > 0$, then $\md{M}( \tuple{d,a})$ holds. ($a$ is in the domain of a model.)
\end{itemize}
\end{definicja}

In the above, a model is essentially defined as a particular kind of tuple: (definition of the domain, definition of the first relation, definition of the second relation $\ldots$). However, since we allow models with infinite signatures, we define it in the above compact way rather than an actual tuple of formulae. However, if a model is defined with a standard number of definable relations, we can easily construct a definition in the format specified above. The second warning is as follows: although officially in this paper a full model is the same as the \emph{elementary diagram} of that model, in practice, we will refer to models in the usual way, since it is clear how to transfer statements about models to statements about their elementary diagrams and vice versa. 

\begin{definicja} \label{defi_aithmetized_models}
\begin{enumerate}
\item If $\md{M}$ is a full model, we write $x \in M$ to say that $x$ is a constant in the language of a full model $\md{M}$. (This means that $x$ is an element of $\md{M}$, since we implicitly assume that all full models are built on Henkin constants.) If $\md{M}$ is a model, the expressions $x \in M$,  means that $\md{M}(d, x)$ holds. 

\item If $\md{M}$ is a (full) model over a language $\Lang$, and $\phi \in \form_{\Lang}$, we say that $\alpha$ is an $\md{M}$-assignment or an $\md{M}$-valuation for a formula $\phi$ if $\alpha$ is a (coded) finite function, whose domain contains $\FV(\phi)$, $\alpha \in \Ass(\phi)$, and for every $x$, $\alpha(x) \in M$. We denote this by $\alpha \in \Ass(\phi, \md{M})$. 
\item If $\md{M}$ is a full model over a language $\Lang$, $\phi \in \form_{\Lang}$, and $\alpha$ is an $\md{M}$-assignment, then the relation $\md{M} \models \phi[\alpha]$ is defined simply as $\phi(\alpha(x_1), \ldots, \alpha(x_c)) \in \md{M}$.
\item If $\md{M}$ is a model over a language $\Lang$, $\phi \in \form_{\Lang}$, and $\alpha$ is an $\md{M}$-assignment, then the relation $\md{M} \models \phi[\alpha]$ is defined only for $\phi$ of standard complexity
via the usual compositional conditions with quantifiers restricted to the domain of $\md{M}$ and satisfaction for base relations $R \in \Lang$ (of arity $c$) defined as follows:
\begin{displaymath}
\md{M} \models R[\alpha] \textnormal{ iff } \md{M}\bigl(\tuple{R,\alpha(v_1), \ldots, \alpha(v_c)} \bigr).
\end{displaymath}
We define satisfaction for equalities of terms in an analogous fashion.
\item If $\md{M}$ is a full model, we will write $\ElDiag(\md{M})$ (elementary diagram of $\md{M}$) instead of $\md{M}$ when we want to stress that we are thinking of a theory rather than of a structure.
\item If $\md{M}$ is a (full) model over a language $\Lang$, $a_1, \ldots, a_c \in M$ and $\phi(v_1, \ldots, v_c)$ is a formula with the displayed free variables, we will write 
\begin{displaymath}
\md{M} \models \phi(a_1, \ldots, a_c)
\end{displaymath}
meaning there there exists an $\md{M}$-valuation $\alpha$ for $\phi$ such that $\alpha(v_i) = a_i$ for all $i < c$ and 
$\md{M} \models \phi[\alpha].$
\item If $\md{M}$ is a (full) model over a language $\Lang$ and $\phi(v_1, \ldots, v_c) \in \form_{\Lang}$ with all free variables displayed, then by $\phi(\md{M})$ we mean the set of (tuples of) elements defined by the formula $\phi$ in $\md{M}$. In other words, it is the set of tuples $(a_1, \ldots, a_c)$ of the elements of $\md{M}$ such that $\md{M} \models \phi[\alpha]$ for some (equivalently, any) $\alpha \in \Ass(\phi)$ such that $\alpha(v_i) = a_i, i \leq c$. 
\end{enumerate}
\end{definicja}

Note that we haven't yet defined what it means that a model satisfies a theory. This is not an omission. Since for general (not full) models, satisfaction is defined only for standard sentences, we only define satisfaction for standard formulae. This is actually a scheme: for each formula we define what it means that a model satisfies this formula. More precisely: for each $n \in \mathbb{N}$, we define what it means that a model satisfies a formula of depth $n$. 

On the other hand, in our paper, non-full models will play a crucial role and in some specific circumstances we are going to say that a model satisfies a theory. This will be defined in some specific cases of our interest later in Definition \ref{def_satisfaction_of_theories_for_models}.

Let us define some more notions which will be particularly important in further parts of our paper. 

\begin{definicja} \label{def_elementary_extension}
Let $\md{M}, \md{N}$ be full models of theories in the same language. We say that $\md{N}$ is an \df{elementary extension} of $\md{M}$ if $\md{M}$ is contained in $\md{N}$. 
\end{definicja}

Recall that officially, a full model is the same as the elementary diagram of that model, so elementary submodels in our sense correspond to elementary submodels in the usual sense. In what follows, we will sometimes conflate an \emph{elementary submodel} with an \emph{image of the elementary embedding}. This should be understood in the obvious way: a formula $\phi(x,y)$ defines an elementary embedding between models $\md{M}$ and $\md{N}$ if it is a injection on elements of the models $\md{M}, \md{N}$ (i.e., it is a relation on constants such that $\phi(a,b_1), \phi(a,b_2)$ together imply $(b_1 = b_2) \in \md{N}$ and $\phi(a_1,b), \phi(a_2,b)$ together imply that $(a_1 = a_2) \in \md{M}$) and the image is an elementary submodel of $\md{M}$ (i.e, the restriction of $\md{N}$ to the language with constants representing the image of $\md{M}$ is a full model). 

We will denote both being an elementary submodel and being an image of an elementary embedding with 
\begin{displaymath}
\md{M} \preceq \md{N}.
\end{displaymath}

Sometimes we only require that the elementarity relation only holds for formulae in a language $\Lang$ such that $\md{M}, \md{N}$ are full models over a language containing $\Lang$. We then write:
\begin{displaymath}
\md{M} \preceq_{\Lang} \md{N}.
\end{displaymath}
Crucially, in this paper we will be looking at the situation where $\md{M}$ is a full model, $\md{N}$ is a (non-full) model over a bigger language and $\md{N}$ is an $\Lang$-elementary extension of $\md{M}$. This is a slightly subtler notion which will be made precise in Definition \ref{def_L_elementary_extension}.

\begin{definicja} \label{def_expansion_of_a_model}
Let $\md{M},\md{N}$ be models or full models in languages $\Lang_{\md{M}}$, $\Lang_{\md{N}}$ respectively. We say that $\md{N}$ is an \df{expansion} of $\md{M}$ if the following conditions are satisfied:
\begin{itemize}
\item $\Lang_{\md{M}} \subseteq \Lang_{\md{N}}$.
\item For every element $a \in N$ there is an element $b \in M$
such that $\md{N} \models x=y[\alpha]$, where $\alpha(x) = a$, $\alpha(y) = b$. (That is, the domain does not change. We write it in this slightly convoluted manner, since we want the definition to work both for models and full models).
\item For every atomic formula $\phi \in \Lang_{\md{M}}$ and $\md{M}$-assignment $\alpha$ for $\phi$,  $\md{M} \models \phi[\alpha]$ if and only if $\md{N} \models \phi[\alpha]$.
\end{itemize}
\end{definicja}

We can say that $\PA$ is capable of handling basic model theory, since it is able to capture the link between models and consistent theories.
In the context of our definitions, this means that in $\PA$ every consistent theory can be extended to a complete consistent theory with Henkin constants.

\begin{tw}[Arithmetized Completeness Theorem] \label{th_act}
For every $n \in \mathbb{N}$, $\PA$ proves that if $\Th$ is a consistent $\Delta_n$-theory, then it has a $\Delta_{n+1}$-full model.
\end{tw}

By $\Delta_n$-theory, we mean a set of sentences defined both with a $\Sigma_n$-formula and a $\Pi_n$-formula. $\Delta_{n+1}$-full model is defined in an analogous way. The above theorem is standard, cf. Section 13.2 of Kaye's book \cite{Kaye}. 
Throughout the paper, we will be using the following handy conventions concerning models:

\begin{konwencja} \label{konw_modele}
\begin{enumerate}
\item When there is no risk of confusion, we will use the same symbol for a predicate symbol and for its denotation in a given (full) model.

\item We will sometimes denote (full) models as tuples, like $(\md{M},T)$. This will simply mean that $(\md{M},T)$ is an expansion of $\md{M}$ with a predicate $T$.

\end{enumerate}
\end{konwencja}

\subsection{Lengths of proofs} \label{sub_lengths_of_proofs}

This section summarizes the basic definitions and tools that we will need in connection with analysing lengths of proofs; much of this material is taken from Pudl\'ak's survey article \cite{Pudlak}. 

\begin{definicja} \label{def_lengths_of_proofs}
	Recall that $|\phi|$ denotes the length of binary code of $\phi$ and $\num{\qcr{\phi}}$ is its G\"odel number (which we denote as $\num{\phi}$ by our Convention \ref{konwencja_slopp_syntax}).
    \begin{enumerate}
	\item Let $\phi$ be an $\Lang_{\Th}$-formula.
		\begin{equation}\label{def::norma}
		\norm{\phi}{\Th} = \left\{\begin{array}{l} \textnormal{ the length of the shortest proof of $\phi$, if $\Th\vdash \phi$};\\
		\infty \textnormal{ otherwise.}\end{array}\right.	
		\end{equation}
      \item If $\norm{\phi}{\Th} \leq n$, we write also
      \[\Th \vdash^{n} \phi.\]
	\end{enumerate}
\end{definicja}

We treat theories simply as sets of sentences, not necessarily closed under deductions. This is because changing the axiomatization of a theory might result in facilitating proofs. 
 
\begin{definicja}[Simulations, Speed-up, Reducibility]
	Let $\mathcal{T}_1$ and $\Th_2$ be two theories and $\mathcal{F}$ a family of functions $f:\mathbb{N}\rightarrow\mathbb{N}$. We shall say that $\Th_1$ $\mathcal{F}$-\emph{simulates} $\Th_2$ iff there exists a function $f\in\mathcal{F}$ such that for every sentence $\phi\in \Lang_{\Th_1}\cap\Lang_{\Th_2}$ that is provable in both $\Th_1$ and $\Th_2$, and for every $n\in \mathbb{N}$, we have:
\[ \Th_2\vdash^n\phi\Rightarrow \Th_1\vdash^{f(n)} \phi.\]
We say that $\Th_2$ is $\mathcal{F}$-\emph{reducible} to $\Th_1$ if there exists a function $f\in\mathcal{F}$ such that for every $n\in\mathbb{N}$ we have:
\[n \textnormal{ is a proof of } \phi \textnormal{ in }\Th_2 \Rightarrow f(n)\textnormal{ is a proof of } \phi \textnormal{ in }\Th_1.\]
We say that $\Th_2$ has a \emph{super-$\mathcal{F}$ speed-up} over $\Th_1$ if $\Th_1$ does not $\mathcal{F}$-simulate $\Th_2$. 

\end{definicja}

Typical examples in the literature of speed-up (simulability) phenomena concern the cases where $\mathcal{F}$ is the family of either polynomial or elementary functions, which respectively correspond to super-\emph{polynomial} speed-up (or \emph{polynomial simulation}) and super-\emph{exponential} (or super-\emph{elementary}) speed-up.  

\begin{uwaga}
The main focus of our paper is the relation of $\mathcal{F}$-reducibility, where $\mathcal{F}$ is the family of all P-time computable functions. Recall that $f$ is a P-time computable function if it is a total function such that for each $n$ (the binary code of) $f(n)$ can be computed by a deterministic Turing machine which takes as input (the binary code of) $n$ (see \cite{HajekPudlak} for the precise formulation) and which runs in \textit{polynomial time}.  We call this relation \emph{feasible reducibility}. 

\end{uwaga}

The proofs of the following observations are routine:

\begin{obserwacja}
Let $\mathcal{F}$ be any family of functions  from $\mathbb{N}$ to  $\mathbb{N}$.
\begin{enumerate}
    \item If $\Th_2$ is $\mathcal{F}$-reducible to $\Th_1$, then $\Th_1$ $\mathcal{F}$-simulates $\Th_2$. Moreover if $\Th_2$ is feasibly reducible to $\Th_1$, then $\Th_1$ polynomially simulates $\Th_2$.
    \item  
    If $\mathcal{F}$ is countable, then 
    $\Th_2$ has a super $\cal{F}$ speed-up over $\Th_1$ if there exists an infinite sequence of formulae $\phi_0,\phi_1,\ldots,$ provable in both $\Th_1$ and $\Th_2$ such that for every function $f\in\mathcal{F}$ there exists $n \in \mathbb{N}$ such that
	\[\norm{\phi_n}{\Th_1}> f(\norm{\phi_n}{\Th_2}).\]
\end{enumerate}
\end{obserwacja}

The most prominent role in the investigations in the lengths of proofs is played by consistency statements. We shall now discuss arithmetized provability. Recall that $|n|$ denotes the length of the binary expansion of $n$.

\begin{definicja}[Pudl\'ak, \cite{pudlak_lc}]
	Let $\Th$ be a theory, $\phi(x_0,\ldots,x_n)$ be a formula and $R\subseteq \mathbb{N}^{k+1}$ be a relation. We say that $\phi$ \df{polynomially numerates} $R$ in $\Th$ if there exists a polynomial $p(x_0,\ldots,x_k)$ such that for all natural numbers $n_0,\ldots,n_k$
	\[R(n_0,\ldots,n_k) \textnormal{ iff } \norm{\phi(\num{n_0},\ldots,\num{n_k})}{\Th}\leq p(|n_0|,\ldots,|n_k|).\]
	
\end{definicja}

\begin{tw}[Pudl\'ak, \cite{pudlak_lc} Theorem 3.2]
For any consistent theory $\Th \supseteq \Q$ with an NP-time set of axioms\footnote{An NP-time set or relation is one whose characteristic function can be computed by a non-deterministic Turing machine that runs in polynomial time.} and any $R\subseteq \nat^k$ the following are equivalent:
\begin{enumerate}
	\item $R$ is an NP-time relation;
	\item $R$ is polynomially numerable in $\Q$;
	\item $R$ is polynomially numerable in $\Th$.
\end{enumerate} 	
\end{tw}

Since on  hand in our paper, we want slightly stronger results concerning \emph{feasible reducibility} between theories rather than mere facts about speed-up, and on the other, we work with relatively strong theories anyway, we will use a modification of Pudl\'ak's result, (which actually is simpler than the original theorem).

\begin{definicja}
	Let $\Th$ be a theory, $\phi(x_0,\ldots,x_n)$ be a formula and $R\subseteq \mathbb{N}^{k+1}$ be a relation. We say that $\phi$ \df{uniformly polynomially numerates} $R$ in $\Th$ if there exists 	a P-time computable function $f$ such that for all natural numbers $n_0,\ldots,n_k$, $R(n_0, \ldots, n_k)$ holds iff $f(n_0, \ldots, n_k)$ is a proof of $\phi(\num{n_0}, \ldots, \num{n_k})$ in $\Th$.
\end{definicja}

In what follows, we need only the following simple fact which may be proved by a natural formalisation of Turing Machines in $\IDelta_0 + \Exp$. It is significantly simpler than Pudl\'ak's result, since we do not need to consider cuts or use cut-shortening techniques.


\begin{tw} \label{tw_pudlak_for_simulations}
For any consistent theory $\Th \supseteq \IDelta_0 + \Exp$ with a set of  axioms computable in polynomial time, and any $R \subseteq \nat^k$ the following are equivalent:
\begin{enumerate}
	\item $R$ is a P-time relation;
	\item $R$ is uniformly polynomially numerable in $\IDelta_0 + \Exp$;
	\item $R$ is uniformly polynomially numerable in $\Th$.
\end{enumerate} 
\end{tw}


\begin{wniosek}\label{thm::PolyProv}
	Let $\Th$ be a theory with a P-time set of axioms. Then there exists a formalization of the relation
	\[\Proof_{\Th}(x,y) := \textnormal{"$x$ is a proof of $y$ from the axioms of $\Th$"}\]
	and a polynomially computable function $f(n)$ such that
	\[\mathbb{N}\models \Proof_{\Th}(m,n),\]  
implies that $f(n,m)$ is a proof of $\Proof_{\Th}(\num{m}, \num{n})$ in $\IDelta_0 + \Exp$.
\end{wniosek}
 Moreover, we define formulae $\Con$ and $\Prov$ in the usual manner as follows:
 \[ \Prov_{\Th}(y):=\exists x \Proof_{\Th}(x,y) \] and  
\[ \Con_{\Th}:=\neg\Prov_{\Th}(\num{\exists x(x\neq x)}). \]

If not stated otherwise, throughout this paper $\Th$, and $\Th_i$ for $i\in\mathbb{N}$ range over P-time axiomatizable theories.

\begin{uwaga}[Relativized provability predicates]\label{rem::RelProv}
The content of this remark will be needed only in Section \ref{sub_speed_up_via_FACT}. The formalization of the provability predicate from Corollary \ref{thm::PolyProv} is of the form: "There exists an accepting computation of the Turing machine which recognizes $\Th$ proofs". Let $\phi(x)$ be any arithmetical formula. By writing 
\[\Proof_{\Th}^{\phi}(x,y)\]
we mean the relativized version of the above predicate, i.e., the one in which the relevant Turing machine is supplied with an oracle given by $\phi$ and recognizes the theorems of $\Th+\phi$ (whatever $\phi$ means). We can treat $\Th + \phi$ as a new arithmetized theory, but then in typical cases it won't be $\Delta_1$.
This is why we decide to distinguish between the roles played by the lower and the upper indices in $\Proof_{\Th}^{\phi}(x,y)$: the former will be reserved for theories satisfying the thesis of Corollary \ref{thm::PolyProv}
and the latter for arbitrary formulae, playing the roles of oracles. Obviously the relativized version of Corollary \ref{thm::PolyProv} need not be true, but we will only demand that the following two conditions hold:
\begin{enumerate}
	\item There exists a P-time computable function $f$ such that for all $n$ and all $\phi(x)$, $f(n, \qcr{\phi}, \qcr{\Th})$ is a $\PA$ proof of
	\begin{equation}\label{equat::ProvRel}\tag{RelProv1}
	(\phi(\num{n})\rightarrow \Proof_{\Th}^{\phi}(\num{n},\num{n})),
	\end{equation}
    where $\Th$ denotes the chosen arithmetical definition of $\Th$.
	\item Likewise, there exists a P-time computable function $g$ such that $g(\qcr{\phi}, \qcr{\psi}, \qcr{\Th})$ is a $\PA$ proof of the sentence
	\begin{equation}\tag{RelProv2}
	\forall x \left(\phi(x)\rightarrow \psi(x)\right) \rightarrow \forall y \forall z \left(\Proof_{\Th}^{\phi}(y,z)\rightarrow \Proof_{\Th}^{\psi}(y,z)\right).
	\end{equation}
	 This requires that $\Proof_{\Th}^{\phi}(y,z)$ be constructed uniformly in $\phi$, which can certainly be assured.
 
\end{enumerate}  
In particular, $\Proof_{\Th}^{\phi}(x,y)$ is of length polynomial in the lengths of $\phi$ and the chosen definition of $\Th$. As usual 
\[ \Prov_{\Th}^{\phi}(y):=\exists x \Proof_{\Th}^{\phi}(x,y) \] and  
\[ \Con_{\Th}^{\phi}:=\neg\Prov_{\Th}^{\phi}(\num{\exists x(x\neq x)}). \] 
\end{uwaga}

The following theorem gives a canonical example of a family of sentences whose proofs grow super-exponentially. Let $\Prov_{\Th}(\num{m},\num{n})$ denote the canonical sentence expressing "there is a proof of sentence $n$ from the axioms of $\Th$ of length at most $m$." Then $\Con_{\Th}(\num{m}):= \neg\Prov_{\Th}(\num{m}, 0=1)$.

\begin{tw}[Pudl\'ak,\cite{Pudlak}, Theorem 7.2.2]
Let $\Th$ be a sufficiently strong theory. Let $f$ be an increasing computable function, provably total in $\Th$, whose graph has a polynomial numeration in $\Th$. Then there exists a $\delta>0$ such that
\[\norm{\Con_{\Th}(f(\num{n}))}{\Th}>f(n)^{\delta}.\]
\end{tw}
In particular, for $f(n) := 2_n$, where $2_0 := 1$, and $2_{n+1} := 2^{2_n}$, there is some some $\delta>0$ such that:
\begin{align*}
\norm{\Con_{\ISigma_1}(2_{\num{n}})}{\ISigma_1} & >  2_n^{\delta}; \\
\norm{\Con_{\PA}(2_{\num{n}})}{\PA} & >  2_n^{\delta}.
\end{align*}

\subsection{Feasible truth predicates} \label{sub_feasible_truth_predicates}
Now we turn to arithmetized partial truth predicates, which we want to apply to arbitrary sentences of fixed complexity, defined here in a way that is different from the usual one (i.e., $\Sigma_n$, $\Pi_n$). The measure of complexity will be given by the \emph{depth} of formulae, as defined below.

\begin{definicja} \label{defi_syntactic_depth}
	The \textit{depth} of a formula is the length of the longest path in its syntactic tree (which is allowed to contain arbitrary terms as leaves). $\dpt(\phi,n)$ denotes an arithmetical formula representing the relation "The depth of a formula $\phi$ is at most $n$." We will also write it as $\dpt(\phi) \leq n$.
\end{definicja}
For example $0=0 \wedge \forall x \neg(SSS(x)= 0)$ has depth $3$. 

To state some results in greater generality we will use Pudl\'ak's notion of sequential theories (since there seems to be more than one good definition of sequentiality, we include Pudl\'ak's definition):

\begin{definicja}[Pudl\'ak, \cite{pudlak_lc}]
	A theory $\Th$ is \emph{sequential}, if 
	\begin{enumerate}
		\item Robinson's arithmetic $\Q$ is interpretable in $\Th$ relativized to some formula $N(x)$ of $\Lang_{\Th}$ and
		\item there exists a formula $(x)_t$ (of two variables $x$, $t$) that defines in $\Th$ a \emph{total} function (in both variables) and such that $\Th$ proves
		\[\forall x,y,t\exists z\ \ \left(N(t)\rightarrow \forall s<t\ \ \ \left( (x)_s=(z)_s \wedge (z)_t = y\right)\right)\]
	\end{enumerate}    
\end{definicja}

Now the promised "feasible" partial truth predicates:

\begin{tw}[Pudl\'ak, \cite{Pudlak}, Theorem 3.3.1]\label{thm::Pudlak}
	Let $\Th$ be a sequential theory. There is a family of formulae $\Sat_n(x,y)\in\Lang_{\Th}$ and a polynomial $q(x)$ such that for every $n$ there exists a $\Th$-proof of compositional Tarski's conditions for $\Sat_n(x,y)$ such that the size of that proof is less than $q(n)$. Moreover for every $\phi(x_1,\ldots,x_k)$ of \emph{length} less  than $n$, the shortest $\Th$ proof of
	\[\forall \alpha \in \Ass(\phi) \ \ \left(\Sat_n(\num{\phi},\alpha)\equiv \phi(\alpha)\right)\]
	is less than $q(n)$.
\end{tw}

Moreover, by inspection of Pudl\'ak's proof one quickly sees that in fact finding the above mentioned proofs of polynomial length is feasible. This means that the above theorem can be slightly strengthened as in the theorem below.

\begin{tw} \label{tw_uniformly_feasible_truth_predicates}
	Let $\Th$ be a sequential theory. There is a family of formulae $\Sat_n(x,y)\in\Lang_{\Th}$ and P-time computable functions $f(x)$, $g(x,y)$ such that for every $n$, $f(n)$ is a $\Th$ proof of compositional Tarski's conditions for $\Sat_n(x,y)$. Moreover for every $\phi(x_1,\ldots,x_k)$ of \emph{length} less
    than $n$, $g(\qcr{\phi},n)$
    is a $\Th$-proof of
	\[\forall \alpha \in \Ass(\phi) \ \ \left(\Sat_n(\num{\phi},\alpha)\equiv \phi(\alpha)\right).\]	
\end{tw}

In what follows, $\Tr_n(x)$ abbreviates $\Sat_n(x,\varnothing)$. 

\begin{obserwacja}\label{Trn are monotone}
There exists a P-time computable function $f$ such that for every $k\leq n$, $f(n,k)$ is a proof in $\PA$ of 
\[\forall \phi\in\Sent_{\LPA} \ \ \dpt(\phi) \leq k \rightarrow \bigl(\Tr_n(\phi)\equiv \Tr_k(\phi)\bigr).\]
\end{obserwacja}
The proof uses induction (in $\PA$) on the complexity of $\phi$ and provable Tarski biconditionals for $\Tr_l$ predicates.

In further parts of the paper, we will also need the relativized version of the Theorem \ref{tw_uniformly_feasible_truth_predicates} (we state it only for $\PA$). If $\M$ is a $\Delta_k$-model (not necessarily a full one) for a language $\Lang'$ with finitely many fresh non-arithmetical relational symbols, then by $\M$-relativized Tarski's conditions for a formula
$\Phi(x,y)$ we mean the usual statement that $\Phi(x,y)$ satisfies Tarski's compositional truth conditions in which the condition for atomic formulae is:
\[\forall \alpha\in \Ass(\phi,\md{M}) \ \ \Phi(\qcr{R(s_0,\ldots,s_{n-1})},\alpha)\equiv \md{M}\models R(s_0,\ldots,s_{n-1})[\alpha]\]
for an arbitrary relation $R$ in $\Lang'$, and the condition for the existential quantifier is given by 
\[\forall v \in \vrbl \forall \phi(v) \in \form_{\Lang'} \forall \alpha\in \Ass(\phi,\md{M}) \ \ \Phi(\exists v \phi,\alpha)\equiv \exists \beta\sim_v \alpha\ \ \bigl(\beta\in \Ass(\phi,\md{M}) \wedge \Phi(\phi,\beta)\bigr).\]

\begin{wniosek}\label{wniosek:relativized_partial_feasible_truth}
Let $\M$ be any $\Delta_k$-model. There is a family of formulae $\Sat^{\M}_n(x,y)\in \form_{\Lang}$ and P-time computable functions $f(x,y)$, $g(x,y,z)$ such that for every $n$, $f(\qcr{\M}, n)$ is a $\PA$-proof of the $\M$-relativized compositional Tarski's conditions for $\Sat^{\M}_n(x,y)$. Moreover for every $\phi(x_1,\ldots,x_k)$ of \emph{length} less
    than $n$, $g(\qcr{\M},\qcr{\phi}, n)$ is a $\PA$-proof of
	\[\forall \alpha\in \Ass(\phi,\md{M}) \ \ \left(\Sat^{\M}_n(\num{\phi},\alpha)\equiv \M\models \phi(\alpha(x_1), \ldots, \alpha(x_n))\right).\]	
\end{wniosek}
The family $\Sat_n^{\M}(x,y)$ can be defined essentially by relativizing $\Sat_n(x,y)$ predicates from Theorem \ref{thm::Pudlak}. Since the definition of $\M$ does not depend on $n$, the length of $\Sat_n^{\M}(x,y)$ will be polynomial in $n$.

As above, we will write $\Tr_n^{\md{M}}$ to denote satisfaction under the empty valuation. Occasionally, we will use the handy notational convention described below.

\begin{konwencja} \label{konw_sat_predicates_on_elements}
If $\phi(v_1, \ldots, v_n) \in \form$ is a formula with standardly many many variables, we will be writing $\Sat_k(\phi,a_1, \ldots, a_n)$ to denote
\begin{displaymath}
\Sat_k(\phi,\alpha),
\end{displaymath}
where $\alpha \in \Ass(\phi)$ is some valuation which assigns $a_1, \ldots a_n$ to the variables $(v_1, \ldots, v_n)$.
\end{konwencja}

Let us end this subsection with a definition of satisfaction of theories for a larger class of models.

\begin{definicja} \label{def_satisfaction_of_theories_for_models}
Let $n \in \mathbb{N}$. Let $\md{M} \models \B$ be a model over a language $\Lang$ and let $(\md{M},T)$ be its expansion to an $\Lang'$-structure. Suppose that $\Th$ is a theory over the language $\Lang'$ such that $\Th \setminus \B$ consists only of sentences of depth $\leq n$. We say that
\begin{displaymath}
(\md{M}, T) \models \Th
\end{displaymath}
if  $(\md{M},T) \models \Tr_n(\phi)$ for all $\phi \in \Th \setminus \Sent_{\Lang}$.
\end{definicja}

The above definition is actually a scheme. We define separately for all standard $n$ what it means for an expanded structure to satisfy a theory whose axioms have depth bounded by $n$. Note that whenever $(\md{M}, T)$ is an expansion of a full model $\md{M} \models \B$ and $\Th$ extends $\B$ with finitely many standard sentences $\phi_1, \ldots, \phi_n$, the condition $(\md{M},T) \models \Th$ means simply that $(\md{M},T) \models \phi_i$ for $i \leq n$. Let us introduce one more definition in the similar spirit.
\begin{definicja} \label{def_L_elementary_extension}
Let $\md{M}, \md{N} \models \B$, where $\B$ is a theory over language $\Lang$. Let $\Th$ be a theory over language $\Lang'$ such that $\Th \setminus \B$ consists only of sentences of depth $\leq n$. Suppose that $(\md{N},T) \models \Th$ is an expansion of $\md{N}$. We say that $(\md{N}, T)$ is an $\Lang$-elementary extension of $\md{M}$ if $\md{N}$ is an elementary extension of $\md{M}$. We denote it with
\begin{displaymath}
\md{M} \preceq_{\Lang} (\md{N},T).
\end{displaymath}
\end{definicja}

\subsection{Polynomial simulations and feasible reductions for theories extending $\PA$.} \label{sec_polynomial_simulations}

Let us fix $\Th_2\supseteq\Th_1\supseteq \PA$ such that $\Th_2$ conservatively extends $\Th_1$. It turns out that in order to verify that $\Th_2$ is feasibly reducible to $\Th_1$ it is sufficient to demonstrate that the formalized conservativity statements for finite fragments of $\Th_2$ over finite fragments of $\Th_1$ are feasibly provable in $\Th_1$. The theorem below makes this precise.\footnote{We are grateful to Fedor Pakhomov who pointed out to us that this is the most direct way of proving our main results. Our previous proofs employed the conceptually more transparent---but technically more demanding---framework of feasible interpretations, as explained in Subsection \ref{sub_speed_up_via_FACT}.} Before stating the theorem, we need a definition.

\begin{definicja} \label{def_restriction_of_a_theory}
	Let $\Th$ be a theory. By $\Th\restr{n}$ we mean the set of axioms of $\Th$ of \emph{length} at most $n$. Abusing notation, we treat $\Th\restr{\num{n}}$ as the canonical formula representing $\Th\restr{n}$ in $\IDelta_0 + \Exp$.
\end{definicja}

\begin{tw}\label{lem::main}
 	Let $\Th$ be a theory extending $\PA$ with an NP-set of axioms. If there is a polynomial $p(n)$ such that for every $n \in \mathbb{N}$, 
 	\[\PA\vdash^{p(n)}\forall \phi \in \Sent_{\LPA} \left(\dpt(\phi)\leq \num{n} \wedge \Prov_{\Th\restr{\num{n}}}(\phi)\rightarrow \Prov_{\PA\restr{\num{p(n)}}}(\phi)\right),\]
 	 then $\PA$ polynomially simulates $\Th$. Moreover, if $\Th$ admits a P-time computable set of axioms and there exists a P-time computable function $f$ and a polynomial $p(n)$ such that for all $n$, $f(n)$ is a $\PA$-proof of
 	 \[\forall \phi \in \Sent_{\LPA} \left(\dpt(\phi)\leq \num{n} \wedge \Prov_{\Th\restr{\num{n}}}(\phi)\rightarrow \Prov_{\PA\restr{\num{p(n)}}}(\phi)\right),\]
 	 then $\Th$ is feasibly reducible to $\PA$.
\end{tw}
 
Recall that $\dpt(\phi)$ is the height of the syntactic tree of $\phi$ and that we use this symbol for the arithmetic formula representing this function. The proof of Theorem \ref{lem::main} will be facilitated by the following lemma which shows that $\PA$ is \emph{feasibly} strongly reflexive.

\begin{lemat}\label{subl::1} 
There exists a P-time computable function $f$ such that for every $n,k\in\mathbb{N}$, $f(n,k)$ is a $\PA$ proof of: 
		\[\forall \phi \in \Sent_{\LPA} \bigl(\dpt(\phi)\leq \num{k} \wedge \Prov_{\PA\restr{\num{n}}}(\phi)\rightarrow \Tr_k(\phi)\bigr).\]
\end{lemat}

\begin{proof}[Proof of Lemma \ref{subl::1}]
	The proof follows the usual pattern, but we have to check that each transformation at work is feasible. Here we provide the general outline; the details are verified carefully in Subsection \ref{Feasible reflexivity} of the Appendix. Assume first that $n\leq k$. Working in $\PA$, we first prove cut-elimination for First Order Logic (this is a single sentence independent of $n$). Then we show that every axiom of $\PA$ of length $\leq n$ is true. For finitely many axioms of Robinson's $\Q$, this is done independently of $n$. For induction axioms of length at most $n$ we use Proposition \ref{stw_prawda_indukcji}, and for logical axioms we use Proposition \ref{stw_pudlak_prawda_logiki}.
	Next we apply cut-elimination over First Order Logic to show that for a sentence $\phi$ of depth $\leq k$ if $\Prov_{\PA\restr{n}}(\phi),$
	  then there is a cut-free proof of a sequent 
      \[\Gamma \longrightarrow \phi,\]
      where $\Gamma$ contains only axioms of $\PA\restr{\num{n}}$. By the subformula property, in such a proof every formula is of depth bounded by $k$.
      Then, using induction on the number of proof lines in a proof using only formulae of depth at most $k$, we show that if 
      \[\Gamma\longrightarrow \Delta\] is provable, then we have:
      \[\forall \alpha \bigl(\alpha \in \Ass(\Gamma\cup \Delta) \wedge \forall x\in \Gamma\ \ \Sat_k(x,\alpha) \rightarrow \exists y \in \Delta \ \ \Sat_k(y,\alpha)\bigr),\]
     where $\alpha \in \Ass(\Gamma\cup \Delta)$ abbreviates $\forall x \in\Gamma\cup \Delta \ \ \alpha \in \Ass(x)$, as in Definition \ref{konw_technical definitions}. Since we already know that all $\PA\restr{n}$ axioms are true, we conclude that $\phi$ is true.
     
     If $k\leq n$, then it is sufficient to carry out the above proof substituting $n$ for $k$ everywhere and use Observation \ref{Trn are monotone}. Note that all the transformations above are uniform in $n,k$, hence in particular they give rise to a function $f$ as in the thesis of the lemma.
\end{proof}

\begin{proof}[Proof of Theorem \ref{lem::main}]
Assume $\Th\vdash^n \phi$ and $\phi \in\Lang_{\PA}$. Then in fact $\Th\restr{n}\vdash \phi$ and $\phi$ is of depth at most $n$. Let $k$ code this proof. Then, by the properties of the provability predicate we have:
 	\[\PA\vdash \Proof_{\Th\restr{\num{n}}}(\num{k},\phi).\]
 	Hence also $\PA\vdash \Prov_{\Th\restr{\num{n}}}(\phi)$. By the formalized conservativity we have that $\PA\vdash \Prov_{\PA\restr{\num{p(n)}}}(\phi)$; by Lemma \ref{subl::1}, we have $\PA\vdash\Tr_{p(n)}(\phi)$; and by feasibly provable $T$-biconditionals (Theorem \ref{thm::Pudlak}) we obtain $\PA\vdash \phi$. All the intermediate steps are polynomial in $n$.
    
    Also, if $\Th$ has a P-time computable axiomatization, by Corollary \ref{thm::PolyProv} there exists a P-time computable function $h$ such that given a proof $k\in\Th$ of a formula $\phi$, $h(k)$ is the proof of the sentence
    \[\Proof_{\Th}(\num{k}, \phi).\]
    Given a function $f$ as in our assumptions and the function $f'$ as in Lemma \ref{subl::1} one easily defines the appropriate feasible reduction by concatenating the proofs given by $f$, $f'$ and $h$ as in the proof for polynomial simulation.
 \end{proof}

\begin{wniosek}\label{cor::key-corollary}
Suppose that there exists a polynomial $p(n)$ and there is a $k\in\mathbb{N}$ such that for every $n\in\mathbb{N}$, $\PA\vdash^{p(n)}\theta_n$, where $\theta_n$ expresses "Every $\Delta_2$-full model of $\PA\restr{\num{p(n)}}$ has an $\LPA$-elementary extension to a full $\Delta_k$-model of $\Th\restr{\num{n}}$." Then $\PA$ polynomially simulates $\Th$. Moreover, if $\Th$ is P-time, and there exists a P-time function $f$ such that for all $n\in\mathbb{N}$, $f(n)$ is a $\PA$-proof of $\theta_n$, then $\Th$ is feasibly reducible to $\PA$.
\end{wniosek}
Let us make one remark before we proceed to the proof of Corollary \ref{cor::key-corollary}. Recall from Subsection \ref{sub_arithmetised_model_theory} that in the current paper we treat full models
as specific arithmetically definable sets of sentences. \textit{Note that although we cannot quantify over models in general, we can do this for models of fixed quantifier complexity using arithmetical satisfaction predicates.}
\begin{proof} [Proof of Corollary \ref{cor::key-corollary}]
Fix $k$ and polynomial $p$. We shall prove the assumptions of Theorem \ref{lem::main}. Fix $n$.
Start the proof by showing $\theta_n$ using a subproof of length $p(n)$.
Then fix $\phi$ of depth 
$\leq n$ and assume:
\[\neg\Prov_{\PA\restr{\num{p(n)}}}(\phi).\]
It immediately follows that $\PA\restr{\num{p(n)}}+\neg\phi$ is consistent. We check that this is a $\Delta_1$-theory (this verification is polynomial in the definition of $\PA\restr{\num{p(n)}}+\neg\phi$, hence polynomial in $n$). We prove the Arithmetized Completeness Theorem \ref{th_act} 
for $\Delta_1$-theories; this is independent of $n$ and gives us a $\Delta_2$-full model of $\PA\restr{\num{p(n)}}+\neg\phi$. Now, by $\theta_n$, this model has an elementary extension to a full model of $\Th\restr{\num{n}}$. By elementarity $\Th\restr{\num{n}}+\neg\phi$ is consistent, which ends the feasible proof of the conservativity claim.

The proof of moreover part of the theorem is fully analogous: we apply it to the respective part of Theorem \ref{lem::main}.
\end{proof}
\begin{wniosek}\label{wniosek:penultimate_corollary}
Suppose there is a polynomial $p(n)$ and there is a $k\in\mathbb{N}$ such that for every $n\in\mathbb{N}$, $\PA\vdash^{p(n)}\theta_n$, where $\theta_n$ expresses "Every $\Delta_2$-full model of $\PA\restr{p(n)}$ has an $\LPA$-elementary extension to a $\Delta_k$-model of $\Th\restr{n}$." Then $\PA$ polynomially simulates $\Th$. Moreover, if $\Th$ is P-time and there exists a P-time computable function $f$ such that for all $n$, $f(n)$ is a $\PA$-proof of $\theta_n$, then $\Th$ is feasibly reducible to $\PA$.
\end{wniosek}
Note that above $\LPA$-elementary extensions are understood in the sense of Definition \ref{def_L_elementary_extension}.
\begin{proof}
 We show that the above assumption implies the assumption of the previous corollary. Fix $\Th$.
 Take $p(n)$ and $k$ as in the assumptions. Work in $\PA$. Fix an arbitrary $\Delta_2$-full model $\M$ of $\PA\restr{\num{p(n)}}$. Then there exists a $\Delta_k$-model $\mathcal{N}\models \Th\restr{n}$ which is an elementary extension of $\md{M}$. We argue that the $\Delta_2$-theory
\[\Phi:=\ElDiag(\M)\cup \Th\restr{\num{n}}\]
is consistent. This will finish the proof,
since by Arithmetized Completeness Theorem we will get a $\Delta_3$-full model of this theory (the length of this subproof is polynomial in $n$ as the proof of $\ACT$ is independent of $n$). 

Take an arbitrary proof $\pi$ of a sentence $\phi$ in $\Phi$ and prove cut-elimination for first order logic (the length of this proof is independent of $n$) and conclude that there exists a proof $\pi'$ of $\phi$ in $\Phi$ with the subformula property. It follows that every formula in this proof is either an arithmetical formula or is a subformula of an additional axiom of $\Th \restr{\num{n}}$. In particular non-arithmetical formulae which occur in this proof are of depth bounded by $n$. Define:
\[\Sat(x,y) := x\in \form_{\LPA} \wedge \M\models x[y] \vee \bigl(x \notin \form_{\LPA} \wedge \Sat^{\mathcal{N}}_n(x,y)\bigr),\]
where $\Sat^{\mathcal{N}}_n(x,y)$ is a feasible relativized truth predicate from Corollary \ref{wniosek:relativized_partial_feasible_truth}. By induction on the length of $\pi'$ show that if $\Gamma\longrightarrow \Delta$ occurs in $\pi'$, then for every $\alpha$
\[\left(\forall x\in\Gamma\ \ \Sat(x,\alpha)\right)\rightarrow \bigl(\exists x\in\Delta\ \ \Sat(x,\alpha)\bigr).\]
It follows that $\pi'$ cannot be a proof of the empty sequent, hence $\Phi$ is consistent.
\end{proof}

The following is the ultimate corollary that best fits the proofs of our main results:

\begin{wniosek}\label{wniosek:ultimate_corollary}
Suppose that $\Th$ is a finite extension of $\PA$ of the form $\PA + \phi$, $k\in\mathbb{N}$ and the following sentence is provable in $\PA$:
\begin{center}
"If $\B$ is any finite fragment of $\PA$, then every $\Delta_2$-full model of $\B$ has an $\LPA$ elementary extension to a $\Delta_k$-model of $\B+\phi$."
\end{center}
Then $\Th$ is feasibly reducible to $\PA$.
\end{wniosek}
\begin{proof}
The assumptions of Corollary \ref{wniosek:ultimate_corollary} clearly implies that the assumptions of Corollary \ref{wniosek:penultimate_corollary}, as the verification that $\B$ is a $\Delta_1$-finite fragment of $\PA$ can be done quickly and uniformly in $n$.
\end{proof}

We will end this subsection with two simple observations which may be obtained by inspection of the proof of Theorem \ref{lem::main} and the proof of subsequent corollaries. They provide some slightly different sufficient condition for feasible reducibility. Observation \ref{obs_expandability_implies_reflexivity} will be useful in Subsection \ref{subsection_FSminus}.
\begin{obserwacja} \label{obs_reflection_implies_reducibility}
Suppose that $\Th$ is a theory with a P-time set of axioms which is \df{$\PA$-provably feasibly strongly reflexive}, i.e., there exists a P-time computable function $h$ such that for each $n,k \in \mathbb{N}$, $h(n,k)$ is a $\PA$ proof of the sentence
 \begin{equation*} \label{equat_feasible_reflexivity} \tag{*}
     \forall \phi \in \Sent_{\LPA} \bigl(\dpt(\phi)\leq \num{k} \wedge \Prov_{\Th\restr{\num{n}}}(\phi)\rightarrow \Tr_k(\phi)\bigr).
 \end{equation*}
 Then $\Th$ is feasibly reducible to $\PA$.
\end{obserwacja}

\begin{obserwacja} \label{obs_expandability_implies_reflexivity}
Suppose that $\Th$ satisfies assumptions of Corollary \ref{wniosek:penultimate_corollary} (with the "moreover" part) or Corollary \ref{wniosek:ultimate_corollary}. Then $\Th$ is $\PA$-provably feasibly strongly reflexive.
\end{obserwacja}
\begin{proof}
By (very direct) inspection of proofs of Corollaries \ref{wniosek:penultimate_corollary} and \ref{wniosek:ultimate_corollary}, we see that if $\Th$ satisfies assumptions of any of these statements, then there exists a P-time computable function $f$ and a polynomial $p(n)$ such that for each $n \in \mathbb{N}$, $f(n)$ is a $\PA$-proof of 
 \begin{equation}
     \forall \phi \in \Sent_{\LPA} \left(\dpt(\phi)\leq \num{n} \wedge \Prov_{\Th\restr{\num{n}}}(\phi)\rightarrow \Prov_{\PA\restr{\num{p(n)}}}(\phi)\right).
 \end{equation}
  It follows that  $\Th$ is $\PA$-provably feasibly strongly reflexive. Indeed , fix the above mentioned function $f$, polynomial $p$ and $g$ witnessing the feasible reflexivity of $\PA$. We define $h(n,k)$. Compute first $f(n)$, then compute $g(p(n),k)$, i.e. the proof of
 \[\forall \phi \in \Sent_{\LPA} \bigl(\dpt(\phi)\leq \num{k} \wedge \Prov_{\PA\restr{\num{p(n)}}}(\phi)\rightarrow \Tr_k(\phi)\bigr).\]
 Then after performing some fixed number of logical transformations obtain the proof of \ref{equat_feasible_reflexivity}.
\end{proof}

\subsection{Feasible interpretability and speed-up} \label{sub_interpretability_and_speedup}

This section presents sufficient conditions for feasible interpretability, a notion that will be only used in Subsection \ref{sub_speed_up_via_FACT}. All the basic notions relied to relative interpretability can be found in \cite{HajekPudlak}, Chapter III. We begin with an observation of Albert Visser, presented in \cite{Fischer-speed}.\footnote{The formulation of the cited theorem is, however, is not quite right since it is claimed that it holds without any restrictions on $\Th_2$, which is incorrect, e.g., let $s\PA:= \PA + \set{\Con_{\PA}(2_{\num{n}})}{n\in\mathbb{N}}$. Then the identity interpretation witnesses relative interpretability of $s\PA$ in $\PA$, but the former theory has super-exponential speed-up over the latter. However, with the proviso that $\Th_2$ is fintiely axiomatizable the proof goes through. Also the proof of our Proposition \ref{prop::inter_implies_speed} can be easily adapted to this case.}\label{Fischer-gap}

\begin{tw} \label{Visser's Pi_1-observation}
	Suppose that $\Th_2\supseteq \Th_1$, where both $\Th_1$ and $\Th_2$ are formulated in the language of $\PA$ and $\Th_2$ is finitely axiomatizable and relatively interpretable in $\Th_1$. Then $\Th_1$ polynomially simulates $\Th_2$ with respect to $\Pi_1$-sentences.
\end{tw}

  The crucial insight in the proof of the above theorem is that relative interpretations are always $\Pi_1$-correct if the interpreting theory is an extension of $\PA$ in the same language, i.e., for every $\Pi_1$-sentence $\phi$ of $\Lang_{\Th_1}$ we have:
\[\Th_1\vdash \phi^I \rightarrow \phi,\]
where $I$ is the chosen relative interpretation of $\Th_2$ in $\Th_1$.

The next definition is formulated so as to allow us to lift the $\Pi_1$-correctness in Theorem \ref{Visser's Pi_1-observation} to wider-and-wider class of sentences so as to establish polynomial simulations of non-finitely axiomatizable theories.


\begin{definicja}\label{def::interpretations_polynomial_correct}
Let  $\Th_1$, $\Th_2$ be two theories each of which has an NP-set of axioms.
\begin{enumerate}
    \item An interpretation $I:\Th_2\rightarrow \Th_1$ is $n$-\emph{correct} if for an arbitrary sentence $\phi$ of depth $n$ we have:
\[\Th_1\vdash \phi^I\rightarrow \phi.\]
    \item An interpretation $I:\Th_2\rightarrow \Th_1$ is \textit{feasible}\footnote{Feasible interpretations are thoroughly investigated in  Verbrugge's doctoral thesis \cite{Verbrugge}. In particular, as shown in Theorem 6.4.2 of \cite{Verbrugge}, there is a sentence $\theta$ such that $\PA+\theta$ is interpretable in $\PA$, and yet there is no \textit{feasible} interpretation of  $\PA+\theta$ in $\PA$. } if  there is a polynomial $p(n)$ such that for all arbitrary sentences $\phi$ and all $n\in \mathbb{N}$ we have:
\[\Th_2\vdash^n \phi\Rightarrow\Th_1\vdash^{p(n)}\phi^{I}.\]
    \item A family of interpretations $\was{I_n}_{n\in\mathbb{N}}: \Th_2 \rightarrow \Th_1$ is \emph{polynomially correct} if for each $n$, $I_n$ is $n$-correct and there exists a polynomial $p(n,k)$ such that the following two conditions hold:
\begin{enumerate}
\item $p(n,k)$ witnesses that $I_k$ is a feasible interpretation, i.e., for every $k,n\in \mathbb{N}$, and for every sentence $\phi$ of $\Lang_{\Th_2}$,
\[\Th_2\vdash^n \phi\Rightarrow\Th_1\vdash^{p(n,k)}\phi^{I_k}.\]
.
\item For every $k\in \mathbb{N}$ and every sentence $\phi$ of length at most $k$,
\[\Th_1 \vdash^{p(k,k)} \phi^{I_k}\rightarrow \phi.\]
\end{enumerate} 
\item A family of interpretations $\was{I_n}_{n\in\mathbb{N}}: \Th_2 \rightarrow \Th_1$ is \emph{uniformly polynomially correct} if for each $n$, $I_n$ is $n$-correct and there exist P-time computable functions $f$, $g$ such that the following two conditions hold:
\begin{enumerate}
\item For every $k\in \mathbb{N}$, and every $\Th_2$-proof $\pi$ of $\phi\in\Lang_{\Th_2}$ $f(k,\pi)$ is a $\Th_1$-proof of $\phi^{I_k}$.
\item For every $k\in \mathbb{N} $, and every sentence $\phi$ of length at most $k$, $g(k,\qcr{\phi})$ is a $\Th_1$-proof of
\[\phi^{I_k}\rightarrow \phi.\]
\end{enumerate}
\end{enumerate}
\end{definicja}

\begin{uwaga}
In the context of theories with no additional rules of reasoning, condition (a) in the definition of polynomial correctness (point 3.) can equivalently be replaced with the following one
\begin{enumerate}
\item[(a)'.] For every $k$, 
\[\Th_1\vdash^{p(n,k)}\phi^{I_k},\]
for every axiom $\phi$ of $\Th_2$ (including the logical axioms for $\Lang_{\Th_2}$) of length at most $n$.
\end{enumerate}
(Analogously for the uniform version.) However, we prefer (a) over (a)' as it can be used in the context of theories such as $\FS^-$ which is closed under two additional rules of reasoning: NEC and CONEC.
\end{uwaga}

The proposition below follows simply by unravelling the relevant definitions:

\begin{stwierdzenie}\label{prop::inter_implies_speed}
If there exists a polynomially correct family of interpretations $\was{I_n}_{n\in\mathbb{N}}: \Th_2\rightarrow \Th_1$, then $\Th_1$ polynomially simulates $\Th_2$. Moreover, if $\was{I_n}_{n\in\mathbb{N}}$ is a uniformly polynomially correct family of interpretations, then $\Th_2$ is feasibly reducible to $\Th_1$.
\end{stwierdzenie}
\begin{proof}
Fix a polynomially correct family of interpretations $\was{I_n}_{n\in\mathbb{N}}$ and let $p(n,k)$ be a polynomial witnessing this. Suppose that $\Th_2\vdash^n \phi$ for some $\phi \in \Lang_{\Th_2}\cap \Lang_{\Th_1}$. Then
\begin{itemize}
\item $\phi$ is of length at most $n$, 
\item every axiom of $\Th_2$ which occurs in the proof is of length at most $n$, and
\end{itemize}
Hence, by Definition \ref{def::interpretations_polynomial_correct}, $\Th_1\vdash^{p(n,n)} \phi^{I_n}$. Since $I_n$ is $n$-correct, we have that 
\[\Th_1\vdash^{O(p(n,n))} \phi.\]
The moreover part is fully analogous.
\end{proof}

\begin{uwaga}\label{rem::polynom_inter_correct}
By the inspection of the proof one quickly realizes that in fact the requirements for the family $\was{I_n}_{n\in\mathbb{N}}$ can be relaxed even further. We do not have to demand that each $I_n$ interprets the whole theory $\Th_2$, instead we can make the weaker demand that there is a polynomial $p(n)$ such that for every $n$ we have:
\[I_n: \Th_2\restr{n}\rightarrow \Th_1\restr{p(n)},\]
where $\Th\restr{n}$ denotes the set of axioms of $\Th$ of length at most $n$.
\end{uwaga}

\section{Dramatis personæ: typed and untyped theories of truth}

 In this section $\B$ denotes a "base theory" for a theory of truth, i.e., a theory with a modicum of arithmetic capable of handling syntax. For example any theory extending $\IDelta_0 + \Exp$ will do. $T$ denotes a fresh unary predicate that is not in the language of $\B$. $\Lang_{\B}$ denotes the language of $\B$ and $\Lang_T$ denotes the language of $\B$ enriched with the predicate $T$. For simplicity assume that the signature of $\Lang_{\B}$ extends the arithmetical signature with finitely many relational symbols. 
 
 
In this paper, we will be dealing with theories of truth conservative over their base theories. We say that a theory $\Th$ in the language $\Lang_{\Th}$ is conservative over $\B \subseteq \Th$ if for every sentence $\phi \in \Lang_{\B}$ we have:
\begin{displaymath}
\B \vdash \phi \textnormal{ iff } \Th \vdash \phi.
\end{displaymath}
In our case, this means that adding the truth predicate and some axioms governing its behaviour does not allow us to prove new arithmetical sentences. 

Below, we discuss some 
prominent examples of truth theories. A standard reference to the subject is Halbach's book \cite{halbach}.

\subsection{$\CT^-$}

\begin{definicja} \label{defi_ctminus}
	$\CT^-[\B]$ is the theory extending a theory $\B$ with the following sentences:
	\begin{enumerate}
		\item[$\CT 1$] $\forall s,t \in \ClTerm_{\Lang_{\B}} \ \ T(s=t)\equiv \val{s} = \val{t}$.
        \item[$\CT2$] $\forall s_1,\ldots, s_n \in \ClTerm_{\Lang_{\B}} \ TR(s_1,\ldots,s_n) \equiv R(\val{s_1},\ldots,\val{s_n})$, for every relational symbol of $\Lang_{\Th}$.
		\item[$\CT 3$] $\forall \phi,\psi \in \Sent_{\Lang_{\B}} \ \  T(\phi\vee\psi)\equiv T\phi \vee T\psi$.
		\item[$\CT 4$] $\forall \phi \in \Sent_{\Lang_{\B}} \ \ T(\neg\phi) \equiv \neg T\phi$.
		\item[$\CT 5$] $\forall \phi \in {\form}^{\leq 1}_{\Lang_{\B}} \forall v \in \vrbl \ \ T(\exists v\phi)\equiv \exists xT\phi(\num{x})$.
		\item[$\CT 6$] $\forall \phi(\bar{x}) \in \form_{\Lang_{\B}} \forall \bar{s},\bar{t} \in \ClTermSeq_{\Lang_{\B}} \ \ \bigl(\bar{\val{s}} = \bar{\val{t}} \rightarrow T\phi[\bar{s}/\bar{x}]\equiv T\phi[\bar{t}/\bar{x}]\bigr)$
	\end{enumerate}
\end{definicja}
The last condition is sometimes called \textit{ generalized regularity}, or \textit{generalized term-extensionality}. It should resemble the well known extensionality rule from deductive calculi for first-order logic, i.e. 
\[
\frac{s_0 = t_0, \ldots, s_k = t_k, \phi(s_0,\ldots, s_k)}{\phi(t_0,\ldots,t_k)}.
\]
We include it since without it the quantifier axiom for $\CT^-$ behaves in an unnatural way.\footnote{It behaves decently already after adding the ungeneralized version of \CT6 for single terms.} For example, for $\B = \PA$ we have:
\[\PA+\CT1\wedge\CT2\wedge\CT3\wedge \CT4\nvdash T\forall x \phi(x) \rightarrow \forall t  T\phi(t)\]
Obviously one can simply interchange the quantifier axiom with the following one: 
\[T \exists y\phi(y)\equiv \exists tT \phi(t).\]
But then, without regularity, the following implication:
\[\forall x \ \ T\phi(\num{x})\rightarrow T \forall y\phi(y)\]
becomes unprovable. With the Regularity both quantifier axioms are easily seen to be equivalent. 

The above version of $\CT^-[\PA]$ was claimed to be conservative over $\PA$ in \cite{WcisloLelykNotes}. However, no proof of this fact was provided and only a hint that it requires a slight modification of the Enayat and Visser construction (see \cite{enayat-visser}). This modification, however, adds a layer of technical difficulty, so in the current version we prove feasible conservativity of $\CT^-[\PA]$ in full detail. A detailed proof of conservativity of this theory is provided also in \cite{WcisloKossak}.

\subsection{$\KF^-$ and $\FS^-$} 

The idea behind the untyped notion of truth is that the truth predicate can be meaningfully applied also to sentences containing it, to the effect that we could e.g., judge 
\[T(\qcr{0=0})\]
to be true. In this setting the following additional axiom seems desirable:

\begin{equation}\label{Idem}\tag{\textnormal{\textsf{TRP}}}
\forall s \in \ClTerm_{\Lang_{\B}}\forall \phi\in \Sent_{\Lang_{T}}\ \ \bigl(\val{s} = \phi\rightarrow TT(s)\equiv T\phi\bigr),
\end{equation}

\noindent where "TRP" abbreviates "TRansParency". Obviously if one wants to have a compositional theory of self-applicable truth, one cannot simply take \eqref{Idem} the axioms $\CT1$ through $\CT6$ and let the quantifiers range over all formulae of $\Lang_T$, since the resulting theory would be inconsistent by Tarski's Theorem. The next two theories which we shall investigate exhibit two different directions in which one can look for a natural theory of untyped truth. In the first one the axiom $\CT3$ is rejected and somewhat compensated. In the second one the transparency axiom is missing.

\begin{definicja}\label{defi::KF}
$\KF^-[\B]$ is the $\Lang_T$-theory extending $\B$ with the following axioms:
\begin{itemize}
	\item[$\KF1$] $\forall s,t \in \ClTerm_{\Lang_{\B}} \ \ T(s=t) \equiv \val{s} = \val{t}$.
    \item[$\KF2$] $\forall s,t \in \ClTerm_{\Lang_{\B}} \ \ T(s \neq t) \equiv \val{s} \neq \val{t}$.
        \item[$\KF3$] $\forall s_1,\ldots, s_n \in \ClTerm_{\Lang_{\B}} \ TR(s_1,\ldots,s_n) \equiv R(\val{s_1},\ldots,\val{s_n})$, for every relational symbol of $\Lang_{\B}$.
        \item[$\KF4$] $\forall s_1,\ldots, s_n \in \ClTerm_{\Lang_{\B}} \ T \neg R(s_1,\ldots,s_n) \equiv \neg R(\val{s_1},\ldots,\val{s_n})$, for every relational symbol of $\Lang_{\B}$.
	\item[$\KF5$] $\forall \phi \in \Sent_{\Lang_T} \ \ T(\neg \neg \phi ) \equiv T \phi.$
	\item[$\KF6$] $\forall \phi, \psi \in \Sent_{\Lang_T} \ \ T(\phi \vee \psi) \equiv T \phi \vee T \psi.$
	\item[$\KF7$] $\forall \phi, \psi \in \Sent_{\Lang_T} \ \ T(\neg(\phi \vee \psi)) \equiv T \neg \phi \wedge T \neg \psi.$
	\item[$\KF8$] $\forall y \in \vrbl \forall \phi \in \form^{\leq 1}_{\Lang_T} \ \ T(\exists y \phi(y)) \equiv \exists x T \phi(\num{x}).$
	\item[$\KF9$] $\forall y \in \vrbl \forall \phi \in \form^{\leq 1}_{\Lang_T} \ \ T(\neg \exists y \phi(y)) \equiv \forall x T \neg \phi(\num{x}).$
	\item[$\KF10$] $\forall \bar{s},\bar{t} \in \ClTermSeq_{\Lang_{\B}} \forall \phi(\bar{x}) \in \form_{\Lang_T} \ \ \Big(\val{\bar{s}}=\val{\bar{t}} \rightarrow T\phi(\bar{s}) \equiv T\phi(\bar{t})\Big).$
    \item[$\KF11$] $\forall \phi \in \Sent_{\Lang_T} \forall t \in \term_{\Lang_{\B}} \ \ \val{t} = \phi \rightarrow TT (t) \equiv T\phi$.
    \item[$\KF12$] $\forall \phi \in \Sent_{\Lang_T} \forall t \in \term_{\Lang_{\B}} \ \ \val{t} = \phi \rightarrow T \neg T(t) \equiv T \neg \phi$.

\end{itemize}
\end{definicja}

$\KF$, a theory obtained by augmenting $\KF^-[\PA]$ with full induction scheme for formulae with the truth predicate, was introduced by Feferman in \cite{feferman} as an axiomatisation of a theory of truth proposed by Kripke in \cite{kripke}. $\KF^-$ represents an attempt to define a reasonably behaved self-applicable truth predicate guided by the following intuition: we try to mark the sentences which are definitely true. We start with the set of true equations on arithmetical sets. Then we proceed in stages, e.g. whenever $\phi$ and $\psi$ are definitely true, we mark $\phi \wedge \psi$ as definitely true. Whenever $\phi$ is definitely true, we mark $T(\phi)$ as definitely true. Whenever $\neg \phi(\num{x})$ is definitely true for all $x$, we mark $\neg \exists \phi(\num{x})$ as definitely true. 
Thus in the process we only enlarge the set of true sentences until it reaches a fixed point. $\KF^-$ axiomatises properties of fixed points obtained in such a way. 

The desirable feature of $\KF^-[\B]$ is that it satisfies the \ref{Idem} axiom. However, the idempotence of the truth predicate fails rather spectacularily in a different place. It turns out that adding both derivation rules 
\begin{align*}
\frac{\phi}{T(\phi)} \textnormal{\ \ (NEC)} & \textnormal{ } & \frac{T(\phi)}{\phi} \textnormal{\ \ (CONEC)}
\end{align*}
to $\KF^-[\B]$ at the same time yields this theory inconsistent (see \cite{halbach}, Lemma 15.20. The Lemma is stated for the full $\KF$, but the induction axioms are not used in the proof). Moreover, the rule (NEC) is inconsistent with the following axiom of consistency which says that no sentence is both true and false:
\begin{displaymath}
\forall \phi \in \Sent_{\Lang_{T}} \ \ \neg \left( T \phi \wedge T \neg \phi \right).
\end{displaymath}
Dually, the rule (CONEC) is inconsistent with the axiom of completeness which states that every sentence is either true or false.

The other standard candidate for a well-behaved theory of self-referential truth is Friedman--Sheard's theory $\FS$.

\begin{definicja}\label{defi::FS}
$\FS^-[\B]$ is the extension of $\B$ in the language extending $\Th$ with the following axioms:

\begin{enumerate}
\item[$\FS1$] $\forall s,t\in \ClTerm_{\Lang_{\B}} \ \ T(s = t) \equiv \bigl(\val{s}= \val{t}\bigr)$.
\item[$\FS2$] $\forall s_1,\ldots, s_n \in \ClTerm_{\Lang_{\B}} \ TR(s_1,\ldots,s_n) \equiv R(\val{s_1},\ldots,\val{s_n})$, for every relational symbol of $\Lang_{\B}$.
\item[$\FS3$] $\forall \phi\in \Sent_{\Lang_{T}} \ \ T(\neg\phi)\equiv \neg T(\phi)$.
\item[$\FS4$] $\forall \phi,\psi\in \Sent_{\Lang_T} \ \ T(\phi\vee \psi) \equiv T(\phi)\vee T(\psi)$.
\item[$\FS5$] $\forall v \in \vrbl \forall \phi \in \form^{\leq 1}_{\Lang_T} \ \ T(\exists v\phi)\equiv \exists xT\phi(\num{x})$.
\item[$\FS6$] $\forall \bar{s},\bar{t} \in \ClTermSeq_{\Lang_{\B}} \forall \phi(\bar{x}) \in \form_{\Lang_T} \ \ \Big(\val{\bar{s}}=\val{\bar{t}} \rightarrow T\phi(\bar{s}) \equiv T\phi(\bar{t})\Big).$
\end{enumerate}
which additionally is closed under the rules (NEC) and (CONEC).
\end{definicja}

Note that in none of the above theories we extend the induction scheme to the full $\Lang_T$. As usual we write simply $\FS^-$ to abbreviate $\FS^-[\PA]$.

A set of axioms, which is deductively equivalent to the above was first introduced in \cite{FriedSheard87}. The above list of axioms is taken from \cite{halbach} with a minor variation: we supplemented the normal axiomatization with $\FS6$ for reasons analogous to the ones for $\CT^-$.

At  first sight, $\FS^-[\B]$ seems to be much more natural than $\KF^-[\B]$. The presence of $\textnormal{NEC}$ and $\textnormal{CONEC}$ rules compensates in a way the lack of the transparency axiom making the theory \emph{symmetric}: for every $\phi\in \Lang_{T}$ it holds that
\[\FS^-[\B]\vdash \phi\iff \FS^-[\B] \vdash T\phi.\]
This heavily contrasts with the case of $\KF^-[\B]$. However this symmetric feature turns out to be very pricey, as the well-known McGee's theorem shows:
\begin{tw}[McGee, \cite{McGee85}]
$\FS^-[\B]$ is $\omega$-inconsistent.
\end{tw}

Moreover, the fully inductive versions of both theories differ dramatically in strength, when evaluated over $\PA$: $\KF[\PA]$ can define $\varepsilon_0$ levels of the ramified truth hierarchy (i.e. $\RT_{<\alpha}$ for every $\alpha<\varepsilon_0$. See \cite{halbach} for details), while the strength of $\FS[\PA]$ is exhausted by $\omega$-many such levels. We sketch the proof of the latter fact in Subsection \ref{Subsection_Conservativity_FS} and give a strengthening of it in Subsection \ref{subsection_FSminus}.

Both $\KF^-[\B]$ and $\FS^-[\B]$ are conservative extensions of $\B$.

\begin{tw}[Cantini, \cite{Cantini}]
$\KF^-[\B]$ is a conservative extension of $\B$.
\end{tw}

The above theorem has been proved by Cantini for $\PA$, but his proof works essentially in the same way for all base theories $\B$ with a modicum of arithmetic. Conservativity of $\FS^-$ follows from the work of Halbach. He showed that $\FS$ with full induction is reducible to the system $\RT_{<\omega}$ with full induction and a stratified family of compositional truth predicates. His proof, however, does not rely on induction in the considered theories or on the specific choice of the base theory. Therefore, essentially the same argument shows that $\FS^-[\B]$ is reducible to $\RT^-_{<\omega}[\B]$ for a wide choice of base theories $\B$. Conservativity of $\RT^-_{<\omega}[\B]$ can in turn be shown by using known proofs of conservativity  for $\CT^-$, so, in a sense, it was "in the air".\footnote{However, we know of no published proof of this result.} We will provide more details (including the definition of $\RT^-_{<\omega}$) in Subsection \ref{Subsection_Conservativity_FS}.
\begin{tw}[Essentially due to Halbach]
$\FS^-[\B]$ is a conservative extension of $\B$.
\end{tw}
We shall sketch both proofs in the next Subsection \ref{Section::conservativity_and_interpretability_of_truth}.

\subsection{Conservativity of truth theories}\label{Section::conservativity_and_interpretability_of_truth}

The main goal of this paper is to establish that certain truth theories over $\PA$ are feasibly reducible to $\PA$. This involves certain elaborate technical arguments in each case. However, what these proofs have in common is that they all rely on the  results from Subsection \ref{sec_polynomial_simulations} since they follow the same general pattern: Suppose that $\Th$ is a theory of truth over $\PA$ that is conservative over $\PA$. Moreover, assume that the conservativity proof in fact can be formalized in $\PA$ and that it is uniform in the sense  that the proof works equally well for $\PA$ and its large enough finitely axiomatized fragments $\B$ that containing $\IDelta_0 + \Exp$. Then $\Th$ can be shown to be feasibly reducible to $\PA$. Let us recall the precise formulation of this fact (it was formulated as Corollary  \ref{wniosek:ultimate_corollary}):

\noindent \textit{Suppose that $\Th$ is a finite extension of $\PA$ of the form $\PA + \phi$, $k\in\mathbb{N}$, and the following sentence is provable in $\PA$:
\begin{center}
"If $\B$ is any finite fragment of $\PA$, then every $\Delta_2$-full model of $\B$ has an elementary extension to a $\Delta_k$-model of $\B+\phi$."
\end{center}
Then $\Th$ is feasibly reducible to $\PA$.}

The proofs of our feasible reducibility results will in each case consist in an appropriate arithmetization in $\PA$ of a known conservativity proof of $\Th$ over fragments of $\PA$. Therefore, we are forced to pay close attention to the specific features of the arithmetical implementation of the conservativity proofs, which is bound to obscure the main idea of the proof of feasible reduction. Therefore, to provide some help to the reader, we present outlines of the relevant conservativity proofs in this section.

\subsubsection{Conservativity of $\CT^-$}\label{sub_conservativity_ct}
In this section, we sketch the proof of the following conservativity result:

\begin{tw} \label{th_ctminus_is_conservative}
Fix any fragment $\B$ of $\PA$ extending $ \IDelta_0 + \Exp$. Then $\CT^-[\B]$ is conservative over $\B$.
\end{tw}

\begin{proof}[Sketch of a proof]

We will base our proof on the argument given by Enayat and Visser in \cite{enayat-visser}. Fix any model $\md{M}$ of $\B$. We will construct $(\md{M}',T) \models \CT^-[\B]$ where $\md{M} \preceq \md{M}'$ by first constructing a chain of models  
\begin{displaymath}
(\md{M}_0,\varnothing) \subseteq (\md{M}_1,S_1) \subseteq (\md{M}_2,S_2) \subseteq \ldots
\end{displaymath}
such that $\md{M}_0 = \md{M}$, $\md{M}_i \preceq \md{M}_{i+1}$ and the subsets $S_i$ are partially defined satisfaction predicates. Each $S_{i+1}$ satisfies compositional conditions for all valuations from $\md{M}_{i+1}$ but only for formulae from $\md{M}_i$. This condition is axiomatized as a scheme.
For example, we require for any arithmetical formulae $\phi, \psi \in \md{M}_i$ \emph{separately} that:
\begin{displaymath}
(\md{M}_{i+1}, S_{i+1}) \models \forall \alpha \in \Ass(\phi,\psi) \ \ S_{i+1}(\phi \vee \psi ,\alpha) \equiv S_{i+1}(\phi,\alpha) \vee S_{i+1}(\psi,\alpha) 
\end{displaymath}
and for any formula $\phi$
\begin{displaymath}
(\md{M}_{i+1}, S_{i+1}) \models \forall \alpha \in \Ass(\exists v \phi) \ \ S_{i+1}(\exists v \phi, \alpha) \equiv \exists \alpha' \sim_v \alpha \ \ S_{i+1}(\phi,\alpha').
\end{displaymath}
Thus we require that $S_{i+1}$ behaves compositionally for formulae which belong to $\md{M}_i$, including nonstandard ones.   Additionally, we require that $S_{i+1}$ agrees with $S_i$ on formulae from $\md{M}_{i-1}$ and arbitrary valuations. Note that if $\phi \in M_i$, then a direct subformula of $\phi$ also belongs to $M_i$. In other words: the predicate gets fixed on the formulae on which it is guaranteed to behave compositionally.

We write the compositional conditions for $S_{n+1}$ in a pointwise manner (formula by formula), so in order to check that such an extension exists, by compactness, we only have to check that for each finite subset of formulae from $\md{M}_i$, we can find a predicate $S_{i+1}$ which satisfies compositional conditions for these formulae. This turns out to be possible with a fairly straightforward recursion.

Finally, we take the sum of models $(\md{M}_i,S_i)$.
In order to check that the resulting sum satisfies compositional axioms (for formulae and assignments), we take an arbitrary formula $\phi$, its direct subformulae, and some fixed valuation $\alpha$ for $\phi$. We check that it satisfied compositional conditions in the model $(\md{M}_{i+1},S_{i+1})$, such that $\phi$ and was present already in $\md{M}_i$, and that the compositional conditions were preserved along the construction.

Finally, we turn the model $(\md{M}',S)$ with a satisfaction class (i.e., a set of pairs $(\phi, \alpha)$ such that $\phi$ is an arithmetical formula and $\alpha \in \Ass(\phi)$) obtained as a sum of a chain into a model with a truth class. We define $T \subsetneq M'$ as follows
\begin{displaymath}
\phi \in T \equiv \phi \in \Sent_{\LPA}(\md{M}') \wedge (\phi, \varnothing) \in S.
\end{displaymath}
This concludes the sketch of the proof. A detailed argument will be presented in Subsection \ref{Section_core_construction}.
\end{proof}

The proof as written above does not overtly formalise in $\PA$. The problem is as follows: when we speak in $\PA$ of full models $(\md{M}_i,S_i)$, we really speak of formulae defining elementary diagrams of $(\md{M}_i,S_i)$. The defining formulae for the full models can in general be more and more complex as we iterate the construction, and there might be no formally correct way of defining the sum of the obtained chain of models. Actually, we cannot even define the whole chain, but only its standard initial fragments.

There are a couple of ways to circumvent this issue. The route undertaken in  this paper is the simplest we know of: we do not speak directly of models, but rather, through appropriate first order theories. More specifically, we will show that for any natural number $x$, the theory $\Th_x$ (formulated in an extension of the language of $\B$ with finitely many new predicate symbols), saying: 
\begin{center}
    "There exists a chain $(\md{M}_0,S_0) \subseteq (\md{M}_1,S_1) \subseteq \ldots \subseteq (\md{M}_x,S_x)$ of models satisfying the conditions from the Enayat--Visser construction."
\end{center}
is consistent. This will be done by formalising the inductive step in the construction by Enayat and Visser, i.e., by showing that for all numbers $x$, if $\Th_x$ is consistent, then $\Th_{x+1}$ is consistent as well. The consistency of $\Th_x$ is a $\Pi_1$-statement, so $\PA$ will be able to verify that for any $x$ the theory $\Th_x$ is consistent. This in turn will be enough to show that the theory saying: \begin{center}
"There is a chain of models: $(\md{M}_0,S_0) \subseteq (\md{M}_1, S_1) \subseteq \ldots$
of infnite length which satisfy conditions from Enayat--Visser construction."
\end{center}
is consistent, hence has a model. From this model we will be able to define the whole chain in a uniform way and, consequently, its sum, which will give us a model of $\CT^-[\B]$ (not a full model though). The details involve a number of intricate and technical considerations; they are presented in the next section.

\subsubsection{Conservativity of $\KF^-$} \label{sub_cons_of_kf}

In this subsection we will outline the proof of conservativity of $\KF^-[\B]$, where $\B$ is a fragment of $\PA$ extending $\IDelta_0 + \Exp$.  The standard proof of conservativity is motivated by the original construction of Kripke.\footnote{What we present here more resembles a modified constructions which seems to be first formulated by Cantini in \cite{Cantini}.}
We can define truth predicate semantically as a fixed point of an operator that takes a subset $T_{\alpha}$ of $\mathbb{N}$, thought of as the set of sentences (possibly containing the truth predicate) which can be already identified as true at a given stage of the construction, and replaces it with $T_{\alpha +1} \supseteq T_{\alpha}$ in the following way:
\begin{itemize}
\item If $\phi$ is a true atomic or negated atomic formula, then $\phi \in T_{\alpha +1}$.
\item If $\phi \in T_{\alpha}$, then $\phi \in T_{\alpha +1}$.
\item If $\phi \in T_{\alpha}$, and $\phi = \val{t}$ for a term $t$, then $T(t) \in T_{\alpha +1}$.
\item If $\neg \phi \in T_{\alpha}$, and $(\neg \phi) = \val{t}$, then $\neg T (t) \in T_{\alpha +1}$.
\item If $\phi \in T_{\alpha}$, then $\neg \neg \phi \in T_{\alpha +1}$.
\item If $\phi \in T_{\alpha}$ or $\psi \in T_{\alpha}$, then $\phi \vee \psi \in T_{\alpha +1}$.
\item If $\neg \phi \in T_{\alpha}$ and $\neg \psi \in T_{\alpha}$, then $\neg (\phi \vee \psi) \in T_{\alpha +1}$.
\item If $\phi(\num{x}) \in T_{\alpha}$, then $\exists v \phi(v) \in T_{\alpha +1}$.
\item If $\neg \phi(\num{x}) \in T_{\alpha}$ for all $x$, then $\neg \exists v \phi(v) \in T_{\alpha +1}$.
\end{itemize}

If $\lambda$ is a limit ordinal, we set $T_{\lambda} = \bigcup_{\alpha < \lambda} T_{\alpha}$. In the above construction, we enlarge the set $T_{\alpha}$ of sentences which are definitely true with a set of sentences which are definitely true if we interpret the truth predicate as the set $T_{\alpha}$. Since at each stage, we only keep enlarging our set, the construction will reach its fixed point.
Such fixed points can be easily shown to satisfy the axioms of $\KF^-[\B]$. The outlined argument carries over to an arbitrary model $\md{M}$ of $\B$ thus establishing conservativity of $\KF^-[\B]$ over its base theory.

The main problem with the outlined argument is that it does not directly formalize in $\PA$ since it relies on the principle: "Every positive operator on subsets of $\mathbb{N}$ reaches a fixpoint." which is clearly not available in $\PA$. However, there is a rather simple fix to this problem.

Start with a \emph{recursively saturated }model $\md{M} \models \B$. Notice that for $n \in \omega$, the $n$-th set obtained in the inductive procedure described above, $T_n$, is arithmetically definable in $\md{M}$ (let us call the defining formula $\Theta_n$). By definability of $T_n$ and recursive saturation of $\md{M}$, we can deduce that already $T_{\omega}$ is a truth predicate satisfying axioms of $\KF^-[\B]$. Essentially, this relies on the fact that in recursively saturated models $\phi(\num{x}) \in T_{\omega}$ holds for all $x \in M$ if and only if $\phi(\num{x}) \in T_k$ holds for some $k \in \omega$ and all $x \in M$.\footnote{A very similar argument has been presented in \cite{CLW} in the proof that any recursively saturated model of $\PA$ can be expanded to a model of $\PT^-$ with internal induction for total formulae. It seems that this reasoning appears originally in \cite{Cantini}, where Cantini proved conservativity  of $\KF^-$ with internal induction for total formulae over $\PA$.}

It turns out that this argument can be repeated in $\PA$ for a finitely axiomatized fragment $\B$ of $\PA$ extending $\IDelta_0 + \Exp$. Namely, we first take a full model $\md{M}$ of $\B$, then we take its recursively saturated elementary extension, a full model $\md{M}'$ and we define the predicate $T$ in $\md{M}'$ as the sum of all sets defined with certain formulae $\Theta_c$ defining the analogues of the sets $T_c$ from the above construction. The details will be given in Subsection \ref{subsection_KFminus}

\subsubsection{Conservativity of $\FS^-$} \label{Subsection_Conservativity_FS}

The proof of conservativity of $\FS^-$ over $\PA$ is analogous to the one showing the upper bounds on the proof-theoretical strength of its fully inductive version, $\FS[\PA]$. As an intermediate step we pass through a theory of iterated compositional truth predicate of length $\omega$, $\RT_{<\omega}^-$.

\begin{definicja}

$\RT_{<n+1}^-[\B]$ is the extension of $\B$ in the language $\Lang_{<n+1}$ extending $\LPA$ with $n+1$ new predicate symbols $\was{T_0,\ldots,T_{n}}$ (we stipulate that $\Lang_{<0} = \Lang_{\PA}$ and $\RT^-_{<0} = \PA$) satisfying the following axioms for all $k<n+1$:

\begin{enumerate}
\item[$\RT1$] $\forall s,t \in \ClTerm_{\LPA} \ \ T_{k}(s=t) \equiv \val{s} = \val{t}$.
\item[$\RT2$] $\forall \phi \in \Sent_{\Lang_{<k}} \ \ T_{k}(\neg \phi)\equiv \neg T_{k}(\phi)$.
\item[$\RT3$] $\forall \phi,\psi\in \Sent_{\Lang_{<k}}\ \ T_{k}(\phi \vee \psi) \equiv T_{k}(\phi)\vee T_{k}(\psi)$.
\item[$\RT4$] $\forall \phi(x) \in \form^{\leq 1}_{\Lang_{<k}}\forall v \in \vrbl \ \ T_k(\exists v \phi)\equiv \exists x\ \ T_{k}(\phi(\num{x}))$.
\item[$\RT5$] $\forall \bar{s},\bar{t} \in \ClTermSeq_{\LPA} \forall \phi(x)\in \form^{\leq 1}_{\Lang_{<k}}\ \ \bigl(\val{\bar{s}} = \val{\bar{t}}\rightarrow T_k(\phi(\bar{s}))\equiv T_k(\phi(\bar{t}))\bigr)$.
\item[$\RT6$] $\bigwedge_{i<k}\forall s\in \ClTerm_{\LPA} \ \bigl(\val{s} \in \Sent_{\Lang_{<i}} \rightarrow T_k(T_i(s))\equiv T_i(\val{s})\bigr)$.
\item[$\RT7$] $\forall i< k\forall s \in \ClTerm_{\LPA} \ \ \bigl(\val{s} \in \Sent_{\Lang_{<i}}\rightarrow T_k(T_i(s))\equiv T_k(\val{s})\bigr)$
\end{enumerate}
Define $\RT_{<\omega}^-[\PA] := \bigcup_{n\in\omega}\RT_{<n}^-[\PA]$
\end{definicja}

\begin{uwaga}
We assume that the initially chosen coding is extended in such a way that the length of $T_n$ is logarithmic in $n$ (in fact, polynomial will do, so this logarythmic bound is not that important).
\end{uwaga}

As in the case of $\FS^-$, $\RT^-_{<n}$ and $\RT^-_{<\omega}$ abbreviate $\RT^-_{<n}[\PA]$ and $\RT^-_{<\omega}[\PA]$ respectively. Note that similarly to all the rest of theories studied in this paper, in $\RT_{<\omega}^-$ we do not extend the scheme of induction to formulae with the truth predicate.

Now let $\B$ be our base theory. We shall now reduce the problem of conservativity of $\FS^-[\B]$ over $\B$ to the analogous problem for $\RT^-_{<\omega}$. Let us recall that an interpretation $^*$ is an $\omega$-\emph{interpretation}, if for every arithmetical sentence $\phi$ we have
\[\phi^* = \phi.\]
In order to perform the above mentioned reduction it suffices to show that every "finite piece" of $\FS^-$ can be $\omega$-interpreted in $\RT^-_{<\omega}$. In this context "an $n$-piece" means "a sentence which can be deduced from $\B$ and axioms $\FS1$--$\FS6$ (note that in this context $\FS2$ is missing) using at most $n$ applications of NEC and CONEC rules." We shall denote it with $\FS^-_n[\B]$. Thus $\phi$ is in $\FS^-_1[\B]$ if it can be deduced using one application of the NEC rule or one application of the CONEC rule. (But not both. Our definition differs from the original one given by Halbach.) Now the following holds:

\begin{lemat}[Essentially Halbach, \cite{halbach}, Theorem 14.31]\label{lem_Halbach_reduction}
For each $n$, $\FS^-_n[\B]$ is $\omega$-interpretable in $\RT_{<2n+1}^-[\B]$.
\end{lemat}
\begin{proof}
Define a family $\was{g_n}_{n\in\mathbb{N}}$ of primitive recursive functions as follows
\begin{align*}
g_{n}(k) =& \left\{ \begin{array}{ll} 
k &\textnormal{ if $k=(s=t)$,}\\
\qcr{0=1} &\textnormal{ if $k\notin\Sent_{\Lang{T}}$ or $k = T(t)$ and $n=0$,}\\
T_{n-1}(g_{n-1}(t)) &\textnormal{ if $k = Tt$ and $n>0$,}\\
\neg g_{n}(\phi) &\textnormal{ if $k = \neg\phi$,}\\
g_{n}(\phi)\vee g_{n}(\psi) &\textnormal{ if $k=\phi\vee \psi$,}\\
\exists x g_{n}(\phi) &\textnormal{ if $k = \exists x\phi$.}\end{array}\right.
\end{align*}
 Where $T_n(g_n(t))$ abbreviates
 \[ \forall x \bigl(g_n(t) = x \rightarrow T(x)\bigr),\]
 and $g_n(x)= y$ is a natural $\Delta_0$-formula which represents $g_n$ in $\IDelta_0 + \Exp$. We shall check that for every $n$, $g_{n+1}$ is an $\omega$-interpretation of $\FS^-_n[\B]$ in $\RT^-_{<2n+1}[\B]$. It is evident that each $g_n$ acts as identity on arithmetical sentences. Moreover, for every $\phi\in\Lang_T$ and each $n$, $g_n(\phi)$ is a sentence of $\Lang_{<n}$ (that is, it contains truth predicates with indices at most $n-1$) and this fact is provable in $\B$. Hence if $\phi$ is any axiom from $\FS1$ through $\FS6$ and $0<k\leq n$, then
 \begin{equation}\label{equat::interaxiom}\tag{$*$}
     \RT^-_{<n}[\B]\vdash g_{k}(\phi).
 \end{equation}
 Now, following the lines of Halbach's argument, we fix $n$ and, by induction on $i$ up to $n$, we show that for every $i\leq n$ and every $j\in \was{i+1,\ldots, 2n+1-i}$\footnote{That this range of $j$ shrinks in the induction process is needed to deal with CONEC.} we have:
 \[\forall \psi\ \ \FS^-_i[\B]\vdash \psi \Longrightarrow \RT^-_{<2n+1}[\B]\vdash g_{j}(\psi).\]
 
 Note that \eqref{equat::interaxiom} witnesses that the above holds for $i=0$. Now inductively assume that the above holds for an $0<i<n$ and fix $j\in \was{i+2,\ldots,2n+1-(i+1)}$.

 Fix a proof $\pi$ of $\psi$ in $\FS^-_{i+1}[\B]$. Arguing by induction assume that the last rule used in $\pi$ is either NEC or CONEC. In both cases we will use the fact that for all $k\leq l< m$, and every $\phi\in\Lang_{<k}$, $\RT^-_{<m}[\B]$ proves
 \begin{equation*}\label{equat:disquotationRT}\tag{$**$}
     T_l(\phi)\equiv \phi.
 \end{equation*}
If $\psi$ is obtained by NEC, then $\psi = T(\theta)$ and by our induction assumption we know that $\RT^-_{<2n+1}[\B]\vdash g_{j-1}(\theta)$. Since $g_{j-1}(\theta)\in\Sent_{<j-1}$, by \eqref{equat:disquotationRT} we obtain $\RT^-_{<2n+1}[\B]\vdash T_{j-1}(g_{j-1}(\theta))$. The last sentence is by definition equal to $g_{j}(T(\theta))$, hence this case is done. 

 If $\psi$ is obtained by CONEC, then we argue dually using $g_{j+1}$ applied to $T(\psi)$.
\end{proof}

In the rest of this section we sketch the proof of conservativity of $\RT^-_{<\omega}[\B]$ over $\B$ based on Enayat-Visser construction. For starters, let us note that it suffices to construct, for an arbitrary model $\M\models \B$ a chain of models $(\M_i)_{i\in \omega}$ such that 
\begin{enumerate}
    \item $\M_0 = \M$;
    \item $\M_i\models \RT^-_{<i}$;
    \item $\M_i\preceq_{\Lang_{<i}}\M_{i+1}$.
\end{enumerate}
Then $\bigcup_{i\in\mathbb{N}} \M_i$ will be an  elementary extension of $\M$ satisfying $\RT^-_{<\omega}[\B]$. To get $\M_{i+1}$ we basically start the Enayat-Visser construction (as sketched in Subsection \ref{sub_conservativity_ct}) on $\M_i$ for the base language $\Lang_{<i}$. More precisely, we build an $\omega$-chain of models $(\M_i^j, S_j)_{j\in\mathbb{N}}$ such that 
\begin{enumerate}
    \item $\M^0_i = \M_i$ and $S_0 = \varnothing$;
    \item $\M^j_i\preceq_{\Lang_{<i}}\M^{j+1}_i$;
    \item $S_j\subseteq S_{j+1}$
    \item $S_{j+1}$ is a satisfaction class for $\form_{\Lang_{<i}}(\M^j_i)$ with respect to all valuations from $\M^{j+1}_i$
\end{enumerate}
Satisfying the above requirements would suffice to guarantee that in the limit model axioms $\RT1$ through $\RT6$ will hold. However, to account for $\RT7$ we have to  improve our satisfaction classes $S_j$ slightly. This can be done by requiring that $S_{j+1}$ makes true all the statements $\phi$ such that
\[\M_i^j\models T_{l}(\phi) ,\]
for $l\leq i$ and $\phi \in \Sent_{\Lang_{<l}}^{\M^j_i}$ (i.e.,  $\pair{\phi}{\alpha}\in S_{j+1}$ for such any $\phi$ and for every assignment $\alpha \in \M^{j+1}_i$). This, in turn, requires only a tiny modification of the original Enayat--Visser proof. Details will be presented in Section \ref{subsection_FSminus}.

\section{The main act: feasible reductions of truth theories}
\label{The main act}

This section contains the principal results of this paper. The first three subsections are devoted, respectively, to feasible reductions of $\CT^-[\PA]$, $\KF^-[\PA]$, and $\FS^-[\PA]$ to $\PA$. The last section, on the other hand, presents an interpretability-theoretic perspective of our work.

\subsection{Feasible reduction of $\CT^-[\PA]$ to $\PA$}
\label{Section_core_construction}

This section is devoted to the proof of the following result:

\begin{tw}\label{thm::main}  {$\CT^-[\PA]$ is feasibly reducible to $\PA$.}
\end{tw}

An immediate corollary of Theorem \ref{thm::main} is that $\CT^-[\PA]$ does not have super-polynomial speed-up over $\PA$. The proof of a special case of this corollary for $\Pi _{1}$-sentences of arithmetic was presented by Fischer \cite{Fischer-speed}, based on an outline suggested by Visser, but as pointed out in a footnote in Subsection \ref{sub_interpretability_and_speedup} the presented proof lacks an important detail.

Our proof of Theorem \ref{thm::main} will be based on the verification of the veracity of the assumption of Corollary \ref{wniosek:ultimate_corollary} for $k=4$ and $\Th=\CT^-[\PA]$.  In fact, we shall do slightly better: let us say that a theory $\B$ is \emph{good} if it is formulated in a language $\Lang_{\B}$ that extends $\Lang_{\PA}$ with new finitely many relation symbols (so all terms are arithmetical) and $\B$ extends $\ISigma_1$. We shall show that for every $l\in\mathbb{N}$ the following \emph{single} sentence is provable in $\PA$:

\begin{center}
"If $\B$ is any $\Delta_1$-good theory, then every $\Delta_l$-full model of $\B$ has an elementary extension to a $\Delta_{l+2}$-model of $\CT^-[\B]$."
\end{center}
In the above, $l$ is to be thought as independent of the size of the proof that our reduction takes as an argument. In the case of $\CT^-[\PA]$ we will need the above theorem only for $l=2$. The more uniform version will be needed to handle the $\FS^-[\PA]$ case.

Our proof consists in formalizing the $\omega$-chain Enayat--Visser construction inside $\PA$, according to the sketch given in Subsection \ref{Section::conservativity_and_interpretability_of_truth}. As in the conservativity proof of Enayat and Visser, we shall make a detour through partial satisfaction classes. Let us introduce one more definition that will play an intermediate role in the proof below.

 \begin{konwencja}
 If $P$ is an arbitrary unary predicate and $\phi(x)$ an arbitrary formula with one free variable, then we write $\phi\restr{P}$ for the formula $\phi(x)\wedge P(x)$.
 \end{konwencja}
 
 \begin{definicja}[$\CS^-\restr{P}$]\label{defi:CS-}
 	Let $\B$ be a theory in a finite language $\Lang_{\B}$ extending $\ISigma_1$ and $P$ be a fresh unary predicate. $\CS^-\restr{P}[\B]$ is the theory of \textit{$P$-restricted, extensional satisfaction class for $\Lang_{\B}$} formulated in the language $\Lang_S = \Lang_{\B} \cup \was{S}\cup \was{P}$ and extending $\B$ with the following axioms:

 	\begin{enumerate}
 		\item $\forall x,y \bigl(S(x,y)\rightarrow x\in\Form{\mathcal{L}_{\B}}\restr{P}\wedge y \in\Ass(x)\bigr)$.
 		\item $\forall s_0\ldots\forall s_n \in \Term{\mathcal{L}_{\B}}\restr{P} \forall \alpha\in \Ass(s_0,\ldots,s_n) \bigl( S(R(s_0,\ldots,s_n), \alpha)\equiv R(\valt{s_0}{\alpha},\ldots, \valt{s_n}{\alpha})\bigr)$.
 		
 		\item $\forall \phi,\psi \in \form_{\Lang_{\B}}\restr{P} \forall \alpha\in \Ass(\phi,\psi)\ \ \bigl(S(\phi\vee\psi, \alpha)\equiv S(\phi,\alpha)\vee S(\psi,\alpha)\bigr)$.
 		\item $\forall \phi \in \form_{\Lang_{\B}}\restr{P}\forall \alpha\in \Ass(\phi)\ \ \bigl(S(\neg\phi,\alpha)\equiv \neg S(\phi,\alpha)\bigr).$
 		
 		\item $\forall \phi \in \form_{\Lang_{\B}}\restr{P} \forall v \in \vrbl\restr{P} \forall\alpha\in\Ass(\exists v\phi)\ \ \bigl(S(\exists v \phi,\alpha)\equiv \exists \ \ \beta \sim_v \alpha, \beta\in\Ass(\phi)\ \ S(\phi,\beta)\bigr)$.

 	\end{enumerate}
    and the axiom of generalized regularity:
   \begin{multline*}
        \forall \phi \in \form_{\B}\restr{P} \forall \bar{v}\in \FVSeq(\phi)\restr{P}\forall \bar{s},\bar{t}\in\ClTermSeq_{\Lang_{\B}}\restr{P}\forall \alpha\in\Ass(\phi[\bar{s}/\bar{v}],\phi[\bar{t}/\bar{v}]) \\ \left(\valt{\bar{s}}{\alpha} = \valt{\bar{t}}{\alpha} \rightarrow \left(S(\phi[\bar{s}/\bar{v}],\alpha)\equiv S(\phi[\bar{t}/\bar{v}],\alpha)\right)\right)
 \end{multline*}

 	If $\M\models \B$ and $P\subseteq M$, $S\subseteq M^2$ is such that $(\M,S, P)\models \CS^-\restr{P}[\B]$, then $S$ is called a \df{$P$-restricted extensional satisfaction class for $\Lang_{\B}$ on $\M$}. If $S$ is "$x=x$"-restricted, it is called \emph{full}.  Note that the above notion makes sense even if $(\M, S, P)$ is not a full model since $\CS^-\restr{P}[\B]$ is a finite extension of $\B$ (recall Definition \ref{def_satisfaction_of_theories_for_models}).
 \end{definicja}
 
Note that in the definition above we do not restrict the range of assignments (denoted by variable $\alpha$ in the above definition). In effect, we do not assume that the assignments come from the restricted set. This is crucial to our purposes. 

\begin{konwencja}
Below we always assume that $P$ is either empty or defines in $\M$ a universe of an elementary submodel of $\M$. This certainly can be sustained along the inductive condition from the proof below. Under this assumption, $P$ is closed under the direct subformula relation, which we denote with $\imsubf$. More precisely
\[(\M,P)\models \forall \phi,\psi \in \form_{\Lang_{\B}} \biggl(\bigl(P(\phi) \wedge \psi \imsubf \phi\bigr)\rightarrow P(\psi)\biggr).\]
\end{konwencja}
 
 The distinctive feature of Enayat-Visser technique of building truth classes is that one creates a well-behaved \emph{satisfaction} class via a union of chain argument. Let us now state the proposition which will provide us with a proof of the induction step in this construction.

\begin{lemat}[Arithmetized Enayat-Visser construction]\label{lem::AEV}
	Let $\Lang_{\B}$ be a finite language extending $\Lang_{\PA}$. The sentence expressing the following implication is provable in $\PA$ for every $l\in\mathbb{N}$:
	
	
	If $(\md{M},S, P)$ is a $\Delta_l$-full model for $\Lang_{\B}$  such that:
	\begin{enumerate}
		\item $\M\models \ISigma_1$;
		\item $S$ is a $P$-restricted satisfaction class for $\Lang_{\B}$;
	\end{enumerate}  
	then there exists a $\Delta_{l+1}$-full model $\md{N}$ for $\Lang_{\B}$ and an $\Delta_{l+1}$-set $S'\subseteq N^2$ such that:
	\begin{enumerate}
		\item $\M\preceq \mathcal{N}$;
		\item $S'$ is an $M$-restricted satisfaction class for $\Lang_{\B}$ (we add a predicate for the universe of $\md{M}$ to the language);
		\item $S\subseteq S'$.
	\end{enumerate}
\end{lemat}

\begin{uwaga}
    We call the reader's attention to the asymetry in the above lemma: we start with a full model $(\M, S, P)$ but finish with a full model $\md{N}$ and two its subsets $S'$, $M$. This will be compensated for in our inductive construction.
\end{uwaga}

\begin{proof}
	We work in $\PA$. Let $(\mathcal{M}, S, P)$ be as in the antecedent of the implication. We follow the lines of the standard Enayat-Visser proof from \cite{enayat-visser}, but we perform it inside $\PA$. Moreover we  have the additional technical complication caused by adding the regularity axiom to $\CT^-[\B]$. Let us define the language $\Lang_{\EV}$:
	\[\Lang_{\EV} = \Lang_{\B} \cup \was{P}\cup\was{S}\cup\set{c}{c\in M}\cup \set{U_{\phi}(x)}{\M\models \phi\in \Form{\Lang_{\B}}}.\]
	Now let us define the Enayat-Visser theory for $(\M, P, S)$ as the sum of the following sets:
	\begin{align*}
		&\set{\phi(a_1,\ldots,a_n)}{\phi\in\Lang_{\B}, a_i\in M, \M\models \phi(a_1,\ldots,a_n)} (=\ElDiag(\M)),\\
		&\set{U_{\phi}(\alpha)}{\pair{\phi}{\alpha}\in S, \phi\in P},\\
		&\set{\forall \alpha \in \Ass(s_0,\ldots,s_n),\ \ \left(U_{R(s_0,\ldots,s_n)}(\alpha)\equiv R(\valt{s_0}{\alpha},\ldots,\valt{s_n}{\alpha})\right)}{R\in\Lang_{\B}, s_0,\ldots,s_n\in \term_{\Lang_{\B}}(\md{M})},\\
		&\set{\forall \alpha\in\Ass(\psi) \ \ \bigl(U_{\psi}(\alpha)\equiv U_{\phi}(\alpha )\vee U_{\theta}(\alpha )\bigr)}{ \M\models \psi \in \Sent_{\Lang_{\B}} \wedge (\psi = (\phi\vee\theta))},\\
		&\set{\forall \alpha\in\Ass(\psi)\ \ \bigl(U_{\psi}(\alpha)\equiv \neg U_{\phi}(\alpha)\bigr)}{\M\models \psi \in \Sent_{\Lang_{\B}} \wedge \psi = (\neg\phi)},\\
		&\set{\forall \alpha\in\Ass(\psi)\ \ \bigl(U_{\psi}(\alpha)\equiv \exists \beta\sim_v\alpha \ \ U_{\phi}(\beta)\bigr)}{ \M\models \psi \in \Sent_{\Lang_{\B}} \wedge \psi = (\exists v\phi)},\\
		&\left\{ \forall \alpha,\beta\in\Ass(\psi,s,t),\ \ \left(\valt{\bar{s}}{\alpha}=\valt{\bar{t}}{\beta}\rightarrow \left(U_{\psi}(\alpha)\equiv U_{\phi}(\beta)\right)\right) \mid \right. \\
        & \left. \M\models \bar{s}, \bar{t} \in \TermSeq_{\Lang_{\B}} \wedge \exists\theta \in \form_{\Lang_{\B}} \  \exists \bar{v}\in \FVSeq (\theta) \ \ \left(\psi = \theta[\bar{s}/\bar{v}]\wedge\phi = \theta[\bar{t}/\bar{v}]\right) \right\}.\\
	\end{align*}
	\newcommand{\C}{\mathbf{C}}
	We argue that this theory is consistent. Then, by ACT (Theorem \ref{th_act}) there will be a $\Delta_{l+1}$-model $\mathcal{N}$ of this theory. Then putting:
    \[S' = \set{\pair{x}{y}}{x\in \form_{\Lang_{\B}}(\M) \wedge \md{N} \models U_x(y)},\]
we easily check that the triple $(\mathcal{N}, S', M)$ satisfies the claim.

To prove the consistency, we will argue by the compactness theorem. Let $F$ be a finite (in the sense of $\PA$) fragment of this theory. For each predicate $U_{\phi}$ which occurs in $F$ we will find a formula $\theta_{\phi}(x)\in\Lang_{\B}$ such that 
\[(\M, P ,S)\models F[\theta_{\phi}/U_{\phi}]_{U_{\phi}\in F},\]
where $F[\theta_{\phi}/U_{\phi}]_{U_{\phi}\in F}$ denotes the theory resulting from $F$ by replacing each occurrence $U_{\phi}$ with the corresponding formula $\theta_{\phi}$. Note that the above makes perfect sense, since $(\M,S, P)$ is a full model. This clearly guarantees that $F$ is consistent. Moreover from now on we do not need to bother with the sentences from $\ElDiag(\M)$, since they obviously hold in $\M$.

As in the original Enayat-Visser proof we construe $\theta_{\phi}$'s by induction on the appropriately defined rank. Note that we have more work to do here than in the proofs given by Enayat and Visser \cite{enayat-visser} (since in their set-up, the  language of arithmetic is  purely relational), and by Cie\'{s}li\'{n}ski \cite{Cieslinski_Book} (since in his set-up $\CT^-$ does not include our generalized regularity axiom $\CT6$). Let $c$ be the set of formulae $\phi$ such that the predicate $U_{\phi}$ occurs in a formula in $F$.  Let $b$ be an arbitrary
coded set of formulae of $\Lang_{\B}$. We put $\rank^b(\phi)\geq x$ iff there exists a sequence $y$ such that the following three conditions hold (in the last condition $\imsubf$ denotes the relation of being an immediate subformula):
\begin{enumerate}
	\item $\len(y)=x+1$ and $(y)_{x} = \was{\phi}$.
	\item For all $i<x+1$ $(y)_i\subseteq b$.
	\item For all $i<x$ for all $\theta$, $\theta\in (y)_{i+1}$ iff for all $\psi$ such that $\md{M} \models \psi \imsubf \theta$, $\psi\in (y)_i$.
\end{enumerate} 
We say that $\rank^b(\phi) = x$ if $x$ is the greatest $x$ such that $\rank^b(\phi)\geq x$. This definition makes sense, since if $\rank^b(\phi)\geq x$, then $x\leq |b|$ where $|c|$ denotes the cardinality of $c$.

\begin{przyklad}
If $b = \was{0=0, 0=0 \vee 1=1}$, then the $\rank^b(0=0\vee 1=1) = 0$, since $1=1 \notin b$.
\end{przyklad}

The intuition behind the above definition is that $\rank^c(\phi)$ is the complexity of $\phi$ where formulae whose some immediate subformula does not belong to $c$ are treated as atoms. The idea is that for any such formula the satisfaction set $U_{\phi}$ can be defined almost arbitrarily and then $U_{\psi}$'s for formulae of higher rank can be defined in terms of previously defined satisfaction sets. 

Observe that if we follow the above described recursive procedure, then all the compositional axioms (i.e., counterparts of axioms for atomic formulae, disjunction, negation and quantifier) from $F$ will be satisfied. However, we have one immediate problem: it can happen that (an instance of) the axiom of regularity for $\phi$ and $\psi$ is in $F$, but $\phi$ and $\psi$ get different ranks. In such a situation the standard procedure does not seem to guarantee that $\theta_{\phi}$ and $\theta_{\psi}$ (i.e. formulae which interpret $U_{\phi}$ and $U_{\psi}$ in $(\M, S, P)$) will satisfy the regularity axiom. To simplify the notation let us define $\phi\approx_{F}\psi$ if the following is in $F$:
\begin{align*}
&\forall \alpha\in\Ass(\psi,\bar{s})\forall\beta\in\Ass(\phi,\bar{t})\ \ \left(\valt{\bar{s}}{\alpha}=\valt{\bar{t}}{\beta}\rightarrow \left(U_{\psi}(\alpha)\equiv U_{\phi}(\beta)\right)\right).
\end{align*}
Note that it can happen only if the following holds in $\M$ (this follows by definition of the Enayat--Visser theory):	
\begin{align*}
&\bar{s}, \bar{t} \in \TermSeq_{\Lang_{\B}} \wedge \exists\theta \in \form_{\Lang_{\B}} \  \exists \bar{v}\in \FVSeq (\theta) \ \ \left(\psi = \theta[\bar{s}/\bar{v}]\wedge\phi = \theta[\bar{t}/\bar{v}]\right).
\end{align*}
A solution to our puzzle is to complete $c$, obtaining $\compl{c}$, to assure that we have for all $\phi$, $\psi$
 \[\phi\approx_F \psi \Rightarrow \rank^{\compl{c}}(\phi) = \rank^{\compl{c}}(\psi).\]
 It is convenient to extend $\approx_F$ a little bit to make it an equivalence relation. We say that $\xi$ is the \emph{term trivialization} of $\phi$, and write $\xi = \compl{\phi}$ if the following four conditions hold:
\begin{enumerate}
	\item For every occurrence $t$ of a term in $\xi$, if all occurrences of variables in $t$ are free, then $t$ is a free occurrence of a variable.
	\item No variable occurs in $\xi$ as both bounded and free, and no variable occurs as free more than once.
	\item For some $\rho$, a  function with domain $\FV(\theta)$ and  values in $\term_{\Lang_{\B}}$, the equality $\xi[\rho] = \phi$ holds. (Recall that $\xi[\rho]$
	denotes the result of a formal substitution of terms for free variables of $\xi$ and that $\term_{\Lang_{\B}}$ contains also terms with free variables.)
	\item The indices of free variables of $\xi$ are chosen in a canonical way (for example according to the tree-ordering of the syntactical tree of $\xi$. This is only needed to guarantee uniqueness).
\end{enumerate}
The idea behind $\compl{\phi}$ is that if for some term substitution $\rho$ and some formula $\psi$ we have
\[\phi[\rho] = \psi,\]  
then, $\compl{\phi} = \compl{\psi}$ and there are unique term substitutions $\gamma_1$, $\gamma_2$ such that:
\begin{equation*}
\compl{\phi}[\gamma_1] = \phi\textnormal{ and } \compl{\phi}[\gamma_2] = \psi.
\end{equation*}

We write $\phi\approx^{\M}\psi$ if $\M\models\compl{\phi} = \compl{\psi}$.\footnote{The idea of using such term trivializations was directly inspired by Graham Leigh's \cite{leigh}.} Obviously $\approx^{\M}$ is an equivalence relation. Moreover, $\approx^{\M}$ is a congruence with respect to the direct subformula relation $\imsubf$, i.e. the following lemma holds. For its proof consult the appendix.
\begin{lemat}[Congruence lemma] \label{lem_congruence}
	For all $\phi$, $\phi'$, $\psi'$ it holds that
	\begin{equation}\label{equat::congr_triv}\tag{\textnormal{C}}
	\bigl(\phi\imsubf \phi'\wedge \phi'\approx^{\M} \psi' \bigr)\Rightarrow \exists \psi\ \ \bigl(\psi\imsubf \psi' \wedge\psi\approx^{\M}\phi \bigr).
	\end{equation} 
\end{lemat}

By induction it follows that the congruence lemma holds for $\imsubf_a$ in place of $\imsubf$, where $\imsubf_a$ denotes the $a$-step transitive closure of $\imsubf$ (by stipulation $\imsubf_0$ is the relation of equality).

Finally, observe that for every $\phi$, $U_{\compl{\phi}}$ and $U_{\phi}$ are mutually interdefinable. Indeed, fix $\phi$ and $\gamma:\FV(\compl{\phi})\rightarrow \term_{\Lang_{\B}}$ such that $\compl{\phi}[\gamma] = \phi$. Then, having $U_{\compl{\phi}}$, we define $U_{\phi}$ with the condition:
\begin{equation}\label{equat::inter_U1}\tag{$U_{\compl{\phi}}\rightarrow U_{\phi}$}
\alpha \in U_{\phi} \iff \exists \beta \in U_{\compl{\phi}} \forall v\in \FV(\compl{\phi})\ \ \beta(v) = \gamma(v)^{\alpha}. 
\end{equation}
Similarly, having $U_{\phi}$ we define $U_{\compl{\phi}}$ with the condition:
\begin{equation}\label{equat::inter_U2}\tag{$U_{\phi}\rightarrow U_{\compl{\phi}}$}
\beta \in U_{\compl{\phi}} \iff \exists \alpha \in U_{\phi} \forall v\in \FV(\compl{\phi})\ \ \beta(v) = \gamma(v)^{\alpha}.
\end{equation}
Now, define $\compl{c}$ to be the \emph{completion} of $c$ if for all $\psi$, $\psi \in \compl{c}$
if and only if there exists $i,j\leq m$ and $\psi',\phi,\phi' \in c$ such that the following pair of conditions hold:
\begin{enumerate}
	\item $\M\models \psi\imsubf_i \psi' \wedge \phi\imsubf_j\phi'$
	\item $\phi\approx^{\M}\psi$.
\end{enumerate}
\[\xymatrix{	\psi'  & &\phi' \\
	\psi\ar[u]^{\imsubf_i}\ar[rr] & & \phi\ar[ll]^{\approx^{\M}}\ar[u]_{\imsubf_j}}		\]
Let us observe that with the current definition of $\compl{c}$ it holds that for all $\phi,\psi\in c$
\[\phi\approx^{\M}\psi \Rightarrow \rank^{\compl{c}}(\phi) = \rank^{\compl{c}}(\psi).\]
Indeed, suppose this is not the case. Then, assumming without loss of generality that 
\[\rank^{\compl{c}}(\phi) > i = \rank^{\compl{c}}(\psi)\] 
there exists $\phi'\in \compl{c}$ such that $\phi' \imsubf_{i+1} \phi$ but no formula from $\compl{c}$ is the $i+1$-st direct subformula of $\psi$. By the congruence lemma (and induction) there exists $\psi'$ such that $\psi'\imsubf_{i+1}\psi$ and $\psi'\approx^{\M}\phi'$. Since $\phi'\in \compl{c}$, then there are $\theta,\theta', \phi'' \in c$ such that for some  $j,k\leq m$ $\theta\imsubf_j \theta'$ and $\phi'\imsubf_k \phi''$ and $\phi' \approx^{\M} \theta$ (possibly $\phi'' = \phi = \theta'$ and $\theta = \phi'$---when $\phi'\in c$).
Since $\approx^{\M}$ is an equivalence relation, then $\psi' \approx^{\M}\theta$. Now, by the definition of $\compl{c}$ we obtain that $\psi'\in \compl{c}$, a contradiction.

For every $x$, let $F\restr{x}$ denote the fragment of $F$ consisting of axioms for $U_{\phi}$ predicates for $\phi$ of $\rank^{\compl{c}}$ at most $x$ and recall that if $\was{\theta_{\phi}}$ is a family of formulae with one free variable indexed with $\phi$ such that $U_{\phi}\in F\restr{x}$, then 
$$F\restr{x}[\theta_{\phi}/U_{\phi}]_{U_{\phi}\in F\restr{x}}$$
denotes the theory resulting from $F\restr{x}$ by replacing every occurrence of $U_{\phi}$ with the formula $\theta_{\phi}$. Let $\zeta(x)$ be the formula asserting the following implication:

"There exists the unique family of $\Lang_{\B}\cup \was{S}$-formulae $\was{\theta_{\phi}}_{\rank^{\compl{c}}(\phi)\leq x}$ indexed with formulae of $\rank^{\compl{c}}\leq x$ such that:
\begin{enumerate}
	\item For every $\phi$, if $\rank^{\compl{c}}(\phi) = 0$, then:
	\begin{enumerate}
		\item if $\M\models  \exists t_1, \ldots, t_a \in \term_{\Lang_{\B}} \phi = R(t_0,\ldots,t_a)$, then $\theta_{\phi}(x) = R(\valt{t_0}{x},\ldots,\valt{t_a}{x})$, and
		\item if $\phi$ is from $P$, then $\theta_{\phi}(x) = S(\phi,x)$, and
		\item if for some $\psi\in P$, $\phi \approx^{\M} \psi$, then $U_{\phi}$ is defined from $U_{\psi}$ using \eqref{equat::inter_U1} and \eqref{equat::inter_U2};
		\item otherwise put $\theta_{\phi}(x) = (x\neq x)$.
	\end{enumerate}
	\item $(\M,S,P)\models F\restr{x}[\theta_{\phi}/U_{\phi}]_{\rank^{\compl{c}}(\phi)\leq x}$."
\end{enumerate}

Now, we prove $\forall x \zeta(x)$ by induction. 
This concludes the proof of Lemma \ref{lem::AEV}.\end{proof}
Let us now complete the proof of Theorem \ref{thm::main}. We shall show how, working inside $\PA$, given an arbitrary good $\Delta_1$-theory $\B$, we can  elementarily extend an arbitrary $\Delta_l$-full model of $\B$ to a $\Delta_{l+2}$-model
of $\CT^-[\B]$. To this end, working in $\PA$, fix a good theory $\B$, $l\in\mathbb{N}$ and a $\Delta_l$-full model $\M$ of $\B$.
Next, still working in $\PA$ we shall construct an unbounded $\Delta_{l+1}$-chain
of $\Delta_{l+1}$-full models
\[(\M_0, S_0),(\M_1, S_1, M_0),\ldots,(\M_x, S_x, M_{x-1}),\ldots\]
such that:
\begin{itemize}
\item[R1] $\M\preceq\M_0$,
\end{itemize}
and for each $y$ we have:
\begin{enumerate}
        \item[R2] $\M_y\preceq \M_{y+1}$,
		\item[R3] $S_0 = \varnothing$ and $S_{y+1}$ is an $M_y$-restricted satisfaction class for $\Lang_{\B}$ and
		\item[R4] $S_y\subseteq S_{y+1}$.
\end{enumerate}

In particular each triple $(\M_x, S_x, M_{x-1})$ will have a fixed $\Delta_{l+1}$-complexity. 
Let us assume that such a chain has been constructed and $\md{M}_x(y)$ and $S_x(y)$ are formulae defining the sequences of respective $\Lang_{\B}$-full models and restricted satisfaction classes. For example it holds that $\M_x(y)$ iff $y$ is the definition of the $x$-th full model (recall that officially full models are identified with their elementary diagrams). Then (in $\PA$) we define the limit model with the formulae:
\begin{align*}
& \M_{\infty}(z) := \exists x \exists y \in \form^{1}_{\LPA}  \ \ \left(\M_x(y)\wedge \Sat_{l+1}(y,z) \right) ,\\
& S_{\infty}(z) := \exists x\exists y \in \form^{1}_{\LPA} \ \left(S_x(y)\wedge \Sat_{l+1}(y,z)\right),
\end{align*}
where $\Sat_{l+1}(x,y)$ denotes the canonical satisfaction predicate for $\Sigma_{l+1}$-formulae.\footnote{Recall that by Convention \ref{konw_sat_predicates_on_elements}, $\Sat_{l+1}(y,z)$ means $\Sat_{l+1}(y, \zeta)$, where $\zeta$ is a valuation which assigns $z$ to the only variable of $y$.}
Note that $\M_{\infty}$ is really a full $\Lang_{\B}$-model, since the chain is elementary with respect to $\Lang_{\B}$-formulae and each $\M_x$ is a full model for $\Lang_{\B}$.

The rest of the argument follows along the lines of Enayat--Visser proof: we check that $S_{\infty}$ is a full satisfaction class on $\M_{\infty}$, hence $(\M_{\infty}, S_{\infty})$ is a $\Delta_{l+2}$-model of $\CT^-$.

Let us now construct the promised chain of models: reasoning in $\PA$, we first define a sequence of increasing
theories $\left\langle \Th_{m}:m\in \mathbb{N}\right\rangle .$ Intuitively speaking,
for each $m$, $\Th_{m}$ describes a structure $\mathcal{K}_{m}=\left\langle (%
\mathcal{M}_{i},S_{i}):i\leq m\right\rangle $ and the family $\was{(\md{M}_i, S_i, M_{i-1})}_{i\leq m}$ satisfies conditions R1 --R4 for boundedly many numbers. In other words, $\was{(\md{M}_i, S_i, M_{i-1})}_{i\leq m}$ is the initial segment of our desired chain consisting of first $m+1$ models.

We now give a precise description of $\Th_{m}$. The non-logical symbols of $%
\Th_{m}$ consist of $\mathcal{L}_{\B}$, together with constant symbols for every element of $M$, unary predicate
symbols $\{\textnormal{M}_{i}:i\leq m\}$, and binary predicate symbols $\{%
\textnormal{S}_{i}:i\leq m\}.$ Let $\Lang_m$ denote this language. 
\begin{konwencja}
If $\phi$ is any formula (in the sense of $\PA$), and $\textnormal{M}(x)$ is any of $\textnormal{M}_i$'s then we write $\phi^{\textnormal{M}}$ to denote the relativisation of $\phi$ to $\textnormal{M}$. This means that we syntactically replace all quantifiers $\exists x \alpha(x)$ with $\exists x  \left(\textnormal{M}(x) \wedge  \alpha(x)\right)$, all quantifiers $\forall x \alpha(x)$ with $\forall x\left(\textnormal{M}(x) \rightarrow \alpha(x)\right)$ and adding to $\phi$ a conjunct $\bigwedge_{x_i\in \FV(\phi)} \textnormal{M}(x_i)$.
\end{konwencja}

The official translations of R1 through R4 above are as follows: 
\begin{itemize}
\item Condition R1 is translated as $\set{\phi^{\textnormal{M}_0}}{ \phi \in \ElDiag(\md{M})}$. 
\item Condition R2 is translated as
\[\set{\forall x_0\ldots \forall x_a \bigl(\phi(x_0,\ldots,x_a)^{\textnormal{M}_i}\rightarrow \phi(x_0,\ldots,x_a)^{\textnormal{M}_{i+1}}\bigr)}{i<m, \phi(x_0\ldots, x_a)\in\form_{\Lang_{\mathrm{B}}}}.\]
\item Condition R4 is expressed by the following finite set of sentences:
\begin{center}
$\left\{ \forall x\forall \alpha  (\left( \textnormal{S}_{i}(x,\alpha
)\rightarrow \textnormal{S}_{i+1}(x,\alpha )\right) :i<m\right\} .$
\end{center}
\item Condition R3 is expressed by the conjunction of the universal closures of the following finitely many axioms $1$i-$6$i, $0\leq i\leq m$, which directly correspond to the ones from Definition \ref{defi:CS-} (we stipulate that $\phi^{\M_{-1}}(x)$ is always the formula $x\neq x$) :
\begin{itemize}
\item[1i] $S_i(x,y)\rightarrow \bigl(\form_{\mathcal{L}_{\B}}^{\textnormal{M}_{i-1}}(x)\wedge \Ass^{\textnormal{M}_i}(x,y) \bigr). $
\item[2i] $\left( \TermSeq_{\Lang_{\B}}^{\textnormal{M}_{i-1}}(\bar{s}) \wedge \left(x = R(\bar{s})\right)^{\textnormal{M}_{i-1}} \wedge \Ass^{\textnormal{M}_i}(x,\alpha)\right) \rightarrow
\bigl(S_i(x,\alpha) \equiv (R(\bar{s}^{\alpha}))^{\textnormal{M}_i}\bigr).$
\item[3i] $\left( \form_{\Lang_{\B}}^{\textnormal{M}_{i-1}}(x) \wedge (x=\neg y)^{\textnormal{M}_{i-1}} \wedge \Ass^{\textnormal{M}_{i}}(x,\alpha) \right) \rightarrow \left(S_{i}(x,\alpha )\equiv \lnot S_{i}(y,\alpha )\right).$
\item[4i]  $\left( \form_{\Lang_{\B}}^{\textnormal{M}_{i-1}}(x)  \wedge  \left( x=y_{1}\vee y_{2}\right)^{\textnormal{M}_{i-1}} \wedge \Ass^{\textnormal{M}_i}(x,\alpha) \right) \rightarrow$
\begin{center}
$\textnormal{S}_{i}(x,\alpha )\equiv \bigl( \textnormal{S}_{i}\left(
y_{1},\alpha\right) \vee \textnormal{S}_{i}\left( y_{2},\alpha\right) \bigr) .$
\end{center}
\item[5i] $\left( \form_{\Lang_{\B}}^{\textnormal{M}_{i-1}}(x)\wedge  \bigl(\exists v \ \ (\vrbl (v)  \wedge x =\exists v\ y)\bigr)^{\textnormal{M}_i}\wedge  \Ass^{\textnormal{M}_i}(x,\alpha)\right)
\rightarrow $
\begin{center}
$\textnormal{S}_{i}(x,\alpha )\equiv \exists \alpha ^{\prime }\left(  (\alpha ^{\prime }\sim_v \alpha)^{\textnormal{M}_i} \wedge \textnormal{S}_{i}(y,\alpha ^{\prime})\right). $
\end{center}
\item[6i] $\left( \form_{\Lang_{\B}}^{\textnormal{M}_{i-1}}(x) \wedge \VarSeq^{\textnormal{M}_{i-1}}(\bar{v}) \wedge \TermSeq^{\textnormal{M}_{i-1}}_{\Lang_{\B}}(\bar{s}) \wedge \TermSeq^{\textnormal{M}_{i-1}}_{\Lang_{\B}}(\bar{t}) \wedge  \Ass^{\textnormal{M}_i}(x,\bar{s},\bar{t},\alpha) \right) \rightarrow $
\begin{center}
$\left( \left((y_1 = x[\bar{s}/\bar{v}])^{\textnormal{M}_i} \wedge (y_2 = x[\bar{t}/\bar{v}])^{\textnormal{M}_i}  \wedge \left(\bar{s}^{\alpha} = \bar{t}^{\alpha} \right)^{\textnormal{M}_i}\right)\rightarrow \right. $
\end{center}
\begin{center}
$\left. \bigl(S_i(y_1,\alpha) \equiv S_i(y_2,\alpha)\bigr)\right).$
\end{center}
\end{itemize}
\end{itemize}






We can now use induction on $m$\ to show that $\forall m \ \Con(\Th_{m})$:

\paragraph{Base case} Recall that $\M$ is a fixed $\Delta_l$-full
model of $\B$. Let $S_{0}=\varnothing.$ Then since $S_{0}$ is definable in $\md{M}$, the elementary diagram of $\mathcal{K}_{0}:=(\md{M},S_{0})$ is also definable. This makes it clear that $\Con(\Th_{0})$ holds.

\paragraph{Inductive step} Fix $m$ and suppose that $\Con(\Th_{m})$ holds. Then by Theorem \ref{th_act}, there is a full model $\mathcal{K}_{m}$ of $\Th_{m}$ satisfying R1 through R4 above whose elementary diagram is $\Delta_{1+l}$-definable.

Let $\Lang_{\Th_{m}}$ be the language of ${\Th_{m}}$, and let $\md{K}_{m}^{-}$ be the \emph{reduct} of the structure $\md{K}_{m}$ to the language $\{\textnormal{M}_{m},\textnormal{S}_{m}, \textnormal{M}_{m-1}, +,\cdot \}$ in which the universe of discourse is the $\md{K}_{m}$-interpretation of $\textnormal{M}_{m}.$ For example, since $\Lang_{\Th_{1}}=\{\textnormal{M}_{1},\textnormal{M}_{0},+,\cdot ,\textnormal{S}_{0},\textnormal{S}_{1}\}$, a model $\md{K}_{1}$ of $\Th_{1}$ will be a structure of the form $(K_{1},M_{1},M_{0},\oplus ,\odot ,S_{0},S_{1})$, where $M_{i}=\mathrm{M}_{i}^{\mathcal{K}}$, $S_{i}=\mathrm{S}_{i}^{\mathcal{K}}$, and $K_{1}$ is the domain of discourse of $\md{K}_{1}.$ In this case, $\mathcal{K}_{1}^{-}=(M_{1},\oplus ,\odot ,S_{1}).$ So in general $\mathcal{K}_{m}^{-}$ is of the form $(M_{m},\oplus ,\odot ,S_{m}).$\footnote{Recall the conventions from Subsection \ref{sub_arithmetised_model_theory}. Although officially full models are elementary diagrams, we refer to them as though they were usual structures, as it is routine to translate statements about complete Henkinized theories into statements about structures.} Observe that $\mathcal{K}_{m}^{-}$ is a full model. Typically, its domain is smaller than the domain of $\md{K}_m$.

Also let $\md{M}_m$ be the reduct of $\md{K}^-_m$ to $\Lang_{\B}$. Let us observe that taking reducts does not raise the complexity of (the definition of) models, so $\md{K}_m^{-}$ and $\md{M}_m$ are still $\Delta_{l+1}$-full models.
To this model apply Lemma \ref{lem::AEV} for $\M = \M_m$, $S = S_m$ and $P = M_{m-1}$. We are given $\md{N}$, a $\Delta_{l+2}$-full model for $\Lang_{\B}$,  and a $\Delta_{l+2}$-set
 $S'$ such that $S'$ is an $M_m$-restricted satisfaction class and $\md{M}_m\preceq \md{N}$. Now we "glue" this model to the end of the chain given by $\md{K}_m$. More precisely, we define a model $\md{K}_{m+1}$ for $\Lang_{m+1}$ in the following way. The universe of $\md{K}_{m+1}$ is the sum of the universes of $\md{K}_m$ and $\md{N}$ (without loss of generality, renaming the elements of $N\setminus M_m$ if necessary, we assume that $K_m \cap N = M_m$). $\textnormal{M}_{m+1}$ is interpreted as $N$, $\textnormal{S}_{m+1}$ as $S'$ and $+$ and $\cdot$ are interpreted on elements from $\textnormal{N}$ as they were in $\md{N}$. For $0\leq i \leq m$ $\textnormal{M}_i$ and $\textnormal{S}_i$ are interpreted as in $\md{K}_m$. Thus we have obtained a structure which contains an elementary chain of models of $\B$, with $\md{N}$ being the top one and possibly some extra elements in the domain of $K_m \setminus N$.  

Also note that for a structure defined in this manner we do not have an elementary diagram at our disposal, hence an argument is needed to show that $\Con(\Th_{m+1})$ holds. We argue as in the proof of Corollary \ref{wniosek:penultimate_corollary}.
Note that if $\pi$ is an alleged proof of contradiction from the axioms of $\Th_{m+1}$ which has a subformula property, then only the following types of sentences can occur in it:
\begin{itemize}
\item[A] formulae of the form $\phi^{\textnormal{M}_0}$ for $\phi\in\Lang_{\B}$;
\item[B] subformulae of sentences of the form 
\[\forall x_0\ldots \forall x_a \bigl(\phi(x_0,\ldots,x_a)^{\textnormal{M}_i}\rightarrow \phi(x_0,\ldots,x_a)^{\textnormal{M}_{i+1}}\bigr),\]
for $\phi(x_0,\ldots,x_1)\in\form_{\Lang_{\mathrm{B}}}$, $i<m+1$.
 \item[C] subformulae of sentences from (formalization) of condition R4;
\item[D] subformulae of sentences from 1i - 6i, i$\leq m+1$.
\end{itemize} 
The complexity of formulae from C and D
is bounded by a standard number.
This is not the case of formulae from A or B. However, to decide every such sentence we can use $\ElDiag(\md{K}_m)$ and $\ElDiag(\md{N})$ and this is clearly sufficient (all formulae from B are in the universal closure of boolean closure of formulae of type $\phi^{\textnormal{M}_i}$ for $i\leq m+1$). All in all, we can define a $\Sigma_n$-truth predicate for $\md{K}_{m+1}$, for sufficiently large $n$, which would work for all formulae from the proof $\pi$. It follows that $\pi$ cannot be a proof of contradiction. This ends the inductive step and we can conclude that $\forall m \  \Con(\Th_m)$ holds.

We shall now define the promised chain of models as a full model of the limit of $\Th_m$'s. Define: 
\[\Th_{\infty} := \bigcup \set{\Th_c}{c\in\mathbb{N}}.\]
Here $\mathbb{N}$ is treated internally, it simply denotes the universe. $T_{\infty}$ is a consistent theory of complexity $\Delta_{l}$ (it is computable in $\ElDiag(\md{M}))$.
It follows that it has a $\Delta_{l+1}$-full model $\mathcal{K}_{\infty}$. This model gives rise to the $\Delta_{l+1}$-chain of $\Delta_{l+1}$-full models $(\M_x, M_{x-1}, S_x)_{x\in\mathbb{N}}$, which can be defined as follows:
\begin{align*}
M_x(y) &:=  \mathcal{K}_{\infty}\models M_{\num{x}}(y),\\
\mathcal{M}_x\models \phi &:=  \mathcal{K}_{\infty}\models \phi^{\M_{x}},\\
S_x(y,z) &:= \mathcal{K}_{\infty}\models S_{x}(y,z).
\end{align*}
The construction guarantees that under such a definition, the chain $(\M_x, S_x, M_{x-1})$ satisfies the requirements R1 through R4. This finally concludes the proof of feasible reduction of $\CT^-[\PA]$ to $\PA$.$\qed$

\subsection{Feasible reduction of $\KF^-[\PA]$ to $\PA$} \label{subsection_KFminus}

In this subsection we will establish:

\begin{tw} \label{tw_KFminus_has_no_speed_up}
$\KF^-[\PA]$ is feasibly reducible to $\PA$. 
\end{tw}

Our proof of the above theorem will demonstrate  that the assumption of  Corollary \ref{wniosek:ultimate_corollary} holds with the choice of $\Th=\KF^-[\PA]$ and $k=4$, i.e., we will prove:

\begin{lemat} \label{lem_szybkie_modele_kf}
$\PA$ proves that  for any finite fragment $\Frag$ of $\PA$, every full $\Delta_2$-model of $\B$ has an elementary extension to a $\Delta_4$-model of $\KF^-[\B]$.
\end{lemat}

Before proving Lemma \ref{lem_szybkie_modele_kf}, we will first show that $\PA$ can formalize the proof of the existence of recursively saturated models.

\begin{lemat} \label{lem_rec_sat_modele_istnieja}
	For any $k \in \mathbb{N}$, $\PA$ proves that any $\Delta_k$-full model $\md{M}$ of a finite fragment $\B$ of $\PA$, there exists a $\Delta_{k+1}$-full model $\md{M}'$ such that $\PA$ proves: 
	\begin{equation*} \tag{*} \label{eqn_rec_sat_modele_istnieja}
		\textnormal{"$\md{M}'$ is a recursively saturated model of $\B$."}
	\end{equation*}
\end{lemat}

Let us first make sense of the above claim. Recall that by a \df{full model} $\md{M}$ over a language $\Lang$, we mean an elementary diagram of that model, that is, a complete consistent Henkinized theory.

We say that a model is \df{recursively saturated} if for every Turing machine with code $e$, and every finite sequence of elements of $\md{M}$, $a_1, \ldots, a_b \in  M$,
if for every finite sequence $\phi_1(x,\bar{y}), \ldots, \phi_c(x,\bar{y})$ of formulae whose G\"odel numbers are accepted by the machine with index $e$, there exists an $a \in M$ such that
\begin{displaymath}
\md{M} \models \bigwedge_{i \leq c} \phi_i(a,a_1, \ldots,a_b),
\end{displaymath}
then there exists $d \in M$ such that for every $\phi \in \Lang$ which is accepted by the machine with the code $e$,
\begin{displaymath}
\md{M} \models \phi(d,a_1, \ldots, a_b).
\end{displaymath}
 
The above definition is well-known. We cite it here to ensure the reader that it really can be spelled out in $\PA$ and that the claim of recursive saturation of $\md{M}$ can be effectively produced in polynomial time given the definition of $\md{M}$.

Let us note that the lemma itself is also well known. Its formulation and proof can be found in \cite{simpson}, Lemma IX.4.2. We demonstrate it here for the convenience of the reader.

\begin{proof}[Proof of Lemma \ref{lem_rec_sat_modele_istnieja}]
	We reason in $\PA$. Let $\LPA^*$
	be the arithmetical language with constants $c_{i,j}, i,j \in \mathbb{N}$ added. Let $(\phi_i^*)$ be any polynomial time enumeration of all sentences of the language $\LPA^*$. Let $\md{M} \models \B$ be a full model and let $\ElDiag(\md{M})^*$ be the theory whose axioms are elementary diagram of $\md{M}$ (which, according to our official definition from Subsection \ref{sub_arithmetised_model_theory} is the model $\md{M}$ itself), all Henkin sentences (in the language with the new constants), and all sentences of the following shape:
	\begin{displaymath}
	\bigwedge_{i \leq N}  \Big(\big( \exists x \phi_1^*(x,\bar{y}) \wedge \ldots \wedge \phi_k^*(x,\bar{y}) \big) \rightarrow \phi_1^*(c_{i,e},\bar{y}) \wedge \ldots \wedge \phi_k^*(c_{i,e},\bar{y}) \Big),
	\end{displaymath}
	where $N \in \mathbb{N}$, and all the constants of $\LPA^* \setminus \LPA$ occurring in the formulae $\phi_1^*, \ldots, \phi_k^*$ are of the form $c_{j,l}$ for $j<i$, and the machine with the code $e$ accepts sentences $\phi_1^*, \ldots, \phi_k^*$ in less than $N$ steps. By Theorem \ref{th_act} ($\ACT$) the theory $\ElDiag(\md{M})^*$ has a $\Delta_{k+1}$-full model $\md{M}'$. This ends the proof of the claim \ref{eqn_rec_sat_modele_istnieja} in $\PA$. 
\end{proof}

Now, we proceed to the proof of Lemma \ref{lem_szybkie_modele_kf} which will end the proof of Theorem \ref{tw_KFminus_has_no_speed_up}.

\begin{proof}
	We work in $\PA$. Let $\md{N}$ be any $\Delta_2$-model of $\B$. By Lemma \ref{lem_rec_sat_modele_istnieja}, there exists a $\Delta_3$-full recursively saturated model $\md{M}$ of $\B$. 
	
	In our proof, we use a construction resembling the one given originally by Kripke in \cite{kripke}. As we have already noted in Subsection \ref{sub_cons_of_kf}, a very similar argument appeared before in \cite{Cantini} and \cite{CLW}.
	By induction, we define a sequence of arithmetical formulae $\Gamma_c, c \in M$. That is, a sequence of elements $\Gamma_c \in M$ such that $\md{M} \models \Gamma_c \in \form^{\leq 1}_{\LPA}$. Let $\Gamma_0(x)$ be a definition of the atomic diagram of $\md{M}$. More precisely, let 
    \begin{displaymath}
	\Gamma_0(x):= \exists s,t \in \ClTerm_{\LPA} \ \ x = (s=t) \wedge \val{s} = \val{t}.
    \end{displaymath}
	Having defined the formula $\Gamma_n$, we set $\Gamma_{n+1}(\phi)$ (which we also denote by $\phi \in \Gamma_{n+1}$) if and only if one of the following conditions is satisfied:
	\begin{itemize}
		\item $\bigvee_{ j \leq n} \ \phi \in \Gamma_j.$
		\item $\exists t \in \ClTerm_{\Lang_{\B}}  \ \phi = Tt \wedge \val{t} \in \Gamma_n$.
		\item $\exists t \in \ClTerm_{\Lang_{\B}} \ \phi = \neg Tt \wedge (\neg \val{t}) \in \Gamma_n$.
		\item $\exists \psi \in \Sent_{\Lang_{T}} \ \phi = (\neg \neg \psi) \wedge \psi \in \Gamma_n$.
		\item $\exists \psi, \eta \in \Sent_{\Lang_{T}} \ \phi = (\psi \vee \eta ) \wedge (\psi \in \Gamma_n \vee \eta \in \Gamma_n)$.
		\item $\exists \psi, \eta \in \Sent_{\Lang_{T}} \ \phi = \neg (\psi \vee \eta ) \wedge (\neg \psi \in \Gamma_n \wedge \neg \eta \in \Gamma_n).$
		\item $\exists v \in \vrbl \ \psi \in \form^{\leq 1}_{\Lang_{T}}  \ \phi = \exists v \psi(v) \wedge \exists x \ \ \psi(\num{x}) \in \Gamma_n$.
		\item $\exists v \in \vrbl \  \psi \in \form^{\leq 1}_{\Lang_{T}} \ \phi = \neg \exists v \psi \wedge \forall x \ \ (\neg \phi(\num{x})) \in \Gamma_n$.  
	\end{itemize}
	Now, let $T$ be the subset of the domain of $M$ defined as the sum $\bigcup_{i \in \mathbb{N}} \Gamma_i (\md{M})$. In other words,
	\begin{displaymath}
	T(x) := \exists y \ \md{M} \models \Gamma_y(x).
	\end{displaymath}
Consider the expanded model $(\md{M},T).$ Since the definition of $(\md{M},T)$ is $\Sigma_1$ in the complexity of $\md{M}$, the complexity of the resulting model is $\Delta_4$. We would like to ensure that $(\md{M},T)$ is a model $\KF^-[\B]$.  The model $(\md{M},T)$ satisfies $\B$, since $\md{M}$ does, so it is enough to check that $(\md{M},T)$ satisfies truth-theoretic axioms  $\KF1$--$\KF10$.

	This is obvious for $\KF1, \KF2$. Let us check the claim for $\KF4$. Suppose that $(\md{M},T) \models T(\phi\vee \psi)$. Since $(\md{M},T) \models T(\phi \vee \psi)$, there exists $i$ such that
	\begin{displaymath}
	\md{M} \models \Gamma_i(\phi \vee \psi).
	\end{displaymath} 
	Then by definition of $\Gamma_i$, either $\md{M} \models \phi \in \Gamma_{i-1}$ (and, consequently, $T(\phi)$ holds) or $\md{M} \models \psi \in \Gamma_{i-1}$ (and then $T(\psi)$ holds). Conversely, if $(\md{M},T) \models T \phi$ or $(\md{M},T) \models T \psi$, then for some $i$, $\md{M} \models \phi \in \Gamma_i$ or $\md{M} \models \psi \in \Gamma_i$. But then $\phi \vee \psi \in \Gamma_{i+1}$ and, consequently, $(\md{M},T) \models T(\phi \vee \psi)$. This guarantees that $(\md{M},T) \models \KF4$. The proofs for axioms $\KF3,\KF5$ are similar, as are the proofs for axioms $\KF9, \KF10$. Let us focus on axiom $\KF7$.
	
	Suppose that $(\md{M},T) \models T \neg  \exists v \ \psi$. Then there exists $i$ such that $\md{M} \models \neg \exists v \ \psi \in \Gamma_i$. This implies that for all $x$, $\md{M} \models \neg \psi(\num{x}) \in \Gamma_{i-1}$. Therefore, for all $x \in M$, $(\md{M},T) \models \neg \psi(\num{x})$ holds. 
	
	Conversely, suppose that for all $x \in M$, $(\md{M},T) \models T \neg \psi(\num{x})$. In other words, for every $x \in  M$, there exists $i$ such that $\md{M} \models \neg \psi(\num{x}) \in \Gamma_i$. We claim that there exists $k$ such that for all $x$, $\md{M} \models \neg \psi(\num{x}) \in \Gamma_k$. Suppose otherwise. Then for every $k$, the following set of arithmetical formulae is realised in $\md{M}$ by some $x$:
	\begin{displaymath}
	\neg \psi(\num{x}) \notin \Gamma_0 \wedge \ldots \wedge \neg \psi(\num{x}) \notin \Gamma_k.
	\end{displaymath} 
	 Therefore, by recursive saturation, there exists an $a \in  M'$ such that for every $k$, 
	 \begin{displaymath}
	 \neg \psi(\num{a}) \notin \Gamma_k,
	 \end{displaymath}
	 contrary to the assumption. This implies that there exists $k \in M$ such that $\md{M} \models \neg \psi(\num{x}) \in \Gamma_k$ for every $x \in  M'$, and therefore $\md{M} \models \neg \exists x \psi(x) \in \Gamma_{k+1}$. We conclude that $(\md{M},T) \models \KF 7$ holds.
	 The case of axiom $\KF 6$ is straightforward. 
 
In order to prove that $\KF8$ holds, we check by induction on $n$ (in $\PA$) that this axiom is satisfied by formulae in $\Gamma_n$. The conclusion follows immediately. This concludes the proof of the lemma.     
\end{proof}

\subsection{Feasible reduction of $\FS^-$ to $\PA$} \label{subsection_FSminus}

In this section we strengthen the conservativity proof from Subsection \ref{Subsection_Conservativity_FS}  by establishing the following result:
\begin{tw}\label{tw::polynom_speed_up_FS}
$\FS^-$ is feasibly reducible to $\PA$.
\end{tw}

The key step in our construction is to feasibly reduce the theory of $\omega$-many truth predicates, $\RT_{< \omega}^-$, defined in Subsection \ref{Subsection_Conservativity_FS}, to $\PA$. This is achieved in the following lemma.

\begin{lemat}
$\RT^-_{<\omega}$ is feasibly reducible to $\PA$.
\end{lemat}
\begin{proof}
We shall prove that the assumptions of Corollary \ref{cor::key-corollary} hold for $\Th = \RT_{<\omega}^-$. Fix $n$ and an arbitrary finite fragment $\B$ of $\PA$, w.l.o.g. $\B \supseteq I\Sigma_1$ and assume that $\M\models \B$ is a $\Delta_2$-full model. We shall build a $\Delta_{4}$-model of $\RT^-_{<\omega}[\B]$, which clearly suffices (note that now we are talking about $\omega$ internally). The aim is to formalize in $\PA$ the conservativity proof from Subsection \ref{Subsection_Conservativity_FS}. In order to do this we shall build a chain of uniformly definable $\Delta_3$-full models $(\M_n)_{n\in\mathbb{N}}$ such that $ \M\preceq_{\Lang_{\PA}}\M_0$ and for each $n$
\begin{enumerate}
\item $\M_n$ is a full $\Delta_3$-model of $\RT^-_{<n+1}[\B]$, and
\item $\M_k\preceq_{\Lang_{<k+1}}\M_n$ for each $k<n$.
\end{enumerate}
Clearly the limit model will be a model of $\RT^-_{\omega}$ (even a full one---this follows by elementarity). To define the respective chain we shall implement the argument from Section \ref{Section_core_construction}: the chain $\M_0,\ldots,\M_{k}$ will be described by a $\Delta_2$-theory $\Th_k$ formulated in the language $\Lang_{\Th_k}$ whose non-logical symbols are:
\begin{enumerate}
\item symbols of $\Lang_{\B}$;
\item unary predicates: $M_0,\ldots, M_k$;
\item unary predicates: $T_0,\ldots, T_{k}$.
\end{enumerate}
Similarly to the proof for $\CT^-[\B]$ in Subsection \ref{Section_core_construction}, the axioms of $\Th_k$ can be divided into three groups:
\begin{enumerate}
\item $\M_0$ is an elementary supmodel of $\M$. Formally this is expressed as an infinite set of axioms: $\set{\phi^{\textnormal{M}_0}}{\phi\in\ElDiag(\M)}$.
\item $(\M_i)_{i\leq k}$ forms an elementary chain of submodels. More precisely: for each $i$, $\M_i$ is an $\Lang_{<i+1}$ elementary submodel of $\M_{i+1}$.
Formally this is expressed analogously to the condition (R2) from the proof for $\CT^-$.
\item For every $i\leq k$, $\M_{i}$ is a model of $\RT^-_{<i+1}$. This is expressed by formally relativizing the axioms of $\RT^-_{<i+1}$ to $M_i$.
\end{enumerate}

Now, by induction on $n$ we show that $\forall n \Con(\Th_n)$. We follow the lines of the sketch of the conservativity proof given in Subsection \ref{Subsection_Conservativity_FS}. For $n=0$ we simply use the proof from Section \ref{Section_core_construction} to build an elementary supmodel of $\M$ satisfying $\CT^-[\B]$. For the induction step, note that, using the same reasoning as we did in Section \ref{Section_core_construction} to verify $\Con(\Th_{k+1})$, it is enough to build a model for $\RT^-_{<k+1}$ which would be a full model for $\Lang_{<k}$ but will possibly leave some sentences with $T_{k}$ undefined. As in the conservativity proof for $\RT^-_{<\omega}$ we use the fact that $\RT^-_{<k+1}$ is deductively equivalent to the theory I$\Th$ below \footnote{"$I$" abbreviates "Induction" as this theory is used in the induction step of our construction.}:
\begin{equation}\label{equat_ind_theory}\tag{I$\Th$}
\CT^-[\RT^-_{<k}] + \forall \phi\in\Sent_{\Lang_{<k-1}}\ \ \bigl(T_{k-1}(\phi)\equiv T(\phi)\bigr).
\end{equation}
From $\Con(\Th_k)$ we obtain a model $\mathcal{K}$ of $\RT^-_{<k}$. We build an extension satisfying \ref{equat_ind_theory} in $\omega$ many steps via the union of chain argument. The following is the analogue of Lemma \ref{lem::AEV} in our situation:

\begin{lemat}
[Arithmetized Enayat--Visser construction+]\label{lem::AEVplus}
	The sentence expressing the following implication is provable in $\PA$ for every $l\in\mathbb{N}$:
	
	
	If $(\mathcal{M}, S, P)$ is a $\Delta_l$-full model for $\Lang_{<k}\cup \was{S} \cup \was{P}$ such that:
	\begin{enumerate}
		\item $\M\models \ISigma_1$;
		\item $S$ is a $P$-restricted satisfaction class for $\Lang_{<k}$;
	\end{enumerate}  
	then there exists a $\Delta_{l+1}$-full model $\mathcal{N}$ and a $\Delta_{l+1}$-set $S'\subseteq N^2$
    such that the following conditions hold:
	\begin{enumerate}
		\item $\M\preceq \mathcal{N}$;
		\item $S'$ is an $M$-restricted satisfaction class for $\Lang_{<k}$ (we add a predicate for the universe of $\md{M}$ to the language);
		\item $S\subseteq S'$;
        \item for every $\phi\in\form_{<k-1}(\md{M})$,  $(\md{N}, S', M) \models T_{k-1}(\phi)\rightarrow \forall \alpha S'\left(\phi, \alpha\right)$.
	\end{enumerate}
\end{lemat}
\begin{proof}[Sketch of the proof]
We indicate how to modify the proof Lemma \ref{lem::AEV}. Firstly, we add the following sentences to the definition of Enayat-Visser theory of $(\M,S, P)$:
\begin{equation}\label{equat_new_axioms}\tag{$*$}
\set{\forall \alpha \ \ U_{\phi}(\alpha) }{\M\models T_{k-1}(\phi)}.
\end{equation}
Now we work with a finite fragment $F$ of the Enayat and Visser theory. 
The next step which requires a modification, is the definition of $\rank^b$ for a coded set of sentences $b$. According to the previous definition, formula $\phi$ was of $\rank^b$ zero if and only if either $\phi$ was atomic or some immediate subformula of $\phi$ was outside $b$. Now we will treat as formulae of $\rank^b$ zero all formulae from $\form_{\Lang_{<k-1}}(\M)$ as well. For such formulae $\phi$ we have an obvious candidate for the definition of $\theta_{\phi}(x)$ (i.e. the formula defining the extension for $U_{\phi}(x)$ in $\M$). We define:
\[\theta_{\phi}(\alpha) := T_{k-1}(\phi[\alpha]).\]
Note that $T_{k-1}$ satisfies generalized regularity, so it is sufficient to verify the truth of $\phi$ on numerals naming values of $\alpha$. The definition of $\rank^b\geq x$ is now as follows: there exists a sequence $y$ such that
\begin{enumerate}
	\item $\len(y)=x+1$ and $(y)_{x} = \was{\phi}$.
	\item For all $i<x+1$ $(y)_i\subseteq b$.
	\item For all $i<x$ for all $\theta$, $\theta\in (y)_{i+1}$ iff $\theta\in\form_{\Lang_{<k}}\setminus \form_{\Lang_{<k-1}}$ and for all $\psi$ such that $\md{M} \models \psi \imsubf \theta$, $\psi\in (y)_i$.\footnote{Note that if $\phi \in \form_{\Lang_{<k-1}}$, this condition implies that $y$ has length $1$.}
\end{enumerate} 
The definitions of $\rank^b = x$ and $\compl{b}$ (for an arbitrary $b$) are analogous to the ones from the original lemma. The last step which requires a modification is the definition of the  formula $\zeta(x)$. Below, as in the proof for $\CT^-$, $c$ is the set of formulae $\phi$ such that $U_{\phi}$ occurs in $F$. We define $\zeta(x)$ to be the formula expressing:

"There exists the unique family of $\Lang_{<k}\cup \was{S}$-formulae $\was{\theta_{\phi}}_{\rank^{\compl{c}}(\phi)\leq x}$ indexed with formulae of $\rank^{\compl{c}}\leq x$ such that:
\begin{enumerate}
	\item For every $\phi$, if $\rank^{\compl{c}}(\phi) = 0$, then:
	\begin{enumerate}
	\item if $\M\models \exists t_1, \ldots, t_a \in \term_{\Lang_{\B}} \phi = R(t_0,\ldots,t_a)$ for a relation symbol $R \in \Lang_{\B}$, then $\theta_{\phi}(\alpha) = R(\valt{t_0}{\alpha},\ldots,\valt{t_a}{\alpha})$, and
		\item if $\M\models \exists t \in \term_{\Lang_{\B}} (\phi = T_{k-1}(t))$, then $\theta_{\phi}(\alpha) = T_{k-1}(t^{\alpha})$, and
		\item if $\phi\in \form_{\Lang_{<k-1}}(\M)$, then $\theta_{\phi}(x) := T_{k-1}(\phi[x])$, and
		\item if $\phi$ is from $P$, then $\theta_{\phi}(x) = S(\phi,x)$, and
		\item if for some $\psi\in P$, $\phi \approx^{\M} \psi$, then $U_{\phi}$ is defined from $U_{\psi}$ using \eqref{equat::inter_U1} and \eqref{equat::inter_U2};
		\item otherwise $\theta_{\phi}(x) = (x\neq x)$.
	\end{enumerate}
	\item $(\M,S,P)\models F\restr{x}[\theta_{\phi}/U_{\phi}]_{\rank^{\compl{c}}(\phi)\leq x}$."
\end{enumerate}
Note that conditions (c) - (e) are the same as in the original definition. THe rest of the proof is as previously.
\end{proof}
Once we can prove $\forall n \Con(\Th_n)$, the construction of the chain $(\M_n)_{n\in \omega}$ and its sum is precisely the same as in Section \ref{Section_core_construction}.
\end{proof}

Now we want to finish the proof of Theorem \ref{tw::polynom_speed_up_FS}. We have just shown that $\RT^-_{<\omega}[\PA]$ satisfies the assumptions of Corollary \ref{wniosek:penultimate_corollary} (with the "moreover" part). By Observation \ref{obs_expandability_implies_reflexivity}, it follows that $\RT^-_{<\omega}[\PA]$ is $\PA$-provably feasibly strongly reflexive, i.e., there exists a P-time computable function $f$ such that for all $n,k \in \mathbb{N}$, $f(n,k)$ is a $\PA$ proof of the sentence
\begin{equation}\label{equat_ref_RT}\tag{$\textnormal{REF}_n$}
   \forall \phi \in \Sent_{\LPA} \bigl(\dpt(\phi)\leq \num{k} \wedge \Prov_{\RT^-_{< \omega}\restr{\num{n}}}(\phi)\rightarrow \Tr_k(\phi)\bigr). 
\end{equation}

Note that there exists a P-time computable function $g$ such that for any $n$, $g(n)$ is a $\PA$ proof of the sentence
\begin{equation}\label{equat_HR}\tag{$\textnormal{HR}_n$}
    \forall \phi \in \Sent_{\LPA} \left(\dpt(\phi)\leq \num{n} \wedge \Prov_{\FS^-_{\num{n}}}(\phi)\rightarrow \Prov_{\RT^-_{<\num{2n+1}}}(\phi)\right).
\end{equation}
The above is in fact an easy consequence of the proof of Halbach's reduction of $\FS^-$ to $\RT^-_{< \omega}$ from Lemma \ref{lem_Halbach_reduction}.

Finally let us observe that the relation $R(k,n,m)$ defined:
\begin{center}
    "$k$ is a $\FS_n^-$ proof of $m$"
\end{center}
is P-time, so, it is uniformly polynomially binumerable in $\IDelta_0 + \Exp$. This gives us a function $h$ such that for every proof $\pi$ of a sentence $\phi$, $h(\qcr{\pi}, n)$ is a $\PA$ proof of $\Prov_{\FS^-_n}(\num{\phi})$.
Our desired reduction can now be defined as follows: given an $\FS^-$ proof $\pi$ of a sentence $\phi$ compute $n$ and $k$ such that there are exactly $n$ applications of NEC and CONEC in $\pi$ and $\phi$ is of depth $k$. Using $h$ find the proof of $\Prov_{\FS^-_n}(\num{\phi})$. Compute $g(n)$, i.e. the proof of \eqref{equat_HR}. Compute $f(n,k)$, i.e. the proof of \eqref{equat_ref_RT}. Apply finitely many logical operations, to conclude $\Tr_k(\num{\phi})$. Finally apply Theorem \ref{tw_uniformly_feasible_truth_predicates} to compute the proof of
\[\Tr_k(\num{\phi})\equiv \phi.\]
Concatenation of the above proofs yields a $\PA$ proof of $\phi$.
\qed

\subsection{Feasible interpretability of truth theories} \label{sub_speed_up_via_FACT}

In Section \ref{sec_polynomial_simulations} we gave a terse proof of Theorem \ref{lem::main}; that proof did not directly link the notions of feasible reducibility with feasible interpretability, which is how we originally conceived of---and arrived at---our main results.  Since interpretations, especially of the feasible variety, are of foundational and philosophical interest in connection with axiomatic theories of truth, we now explain the interpretability-theoretic perspective of our work by establishing the following result:

\begin{tw}[Feasible interpretability of truth theories]\label{feasible intepretability of truth}
Let $\Th$ be any of the truth theories $\CT^-[\PA]$, $\KF^-[\PA]$, and $\FS^-[\PA]$. Then there exists a uniformly polynomially correct family
\begin{center}$\was{I_n}_{n\in\mathbb{N}}: \Th \rightarrow \PA$ 
\end{center}
of interpretations (in the sense of Definition \ref{def::interpretations_polynomial_correct}).
\end{tw}
Note that, by Proposition \ref{prop::inter_implies_speed}, the existence of a uniformly polynomially correct family of interpretations guarantees feasible reducibility. The proof of Theorem \ref{feasible intepretability of truth} can be readily read-off the second proof of Theorem \ref{lem::main} which we give in this section. We shall demonstrate that the assumptions of Theorem \ref{lem::main} imply the existence of a uniformly polynomially correct family of interpretations. This will make it clear that Theorem \ref{feasible intepretability of truth} holds since we have already already verified in Subsections 4.1, 4.2, and 4.3 that the assumptions of Theorem \ref{lem::main} are met when $\Th$ is any of the truth theories $\CT^-[\PA]$, $\KF^-[\PA]$, and $\FS^-[\PA]$. The second proof of Theorem \ref{lem::main} is based on a feasible version of the Arithmetized Completeness Theorem, which we now turn to. In Lemma \ref{lem::ACT} below an \emph{$n$-set} should be understood as a set that is definable by a formula of depth $n$ ("definable" in the sense of the feasible $\Sat_n$ predicates).

\begin{lemat}[Feasible Arithmetized Completeness Theorem, FACT]\label{lem::ACT}
	There exists a polynomial $p(n)$  and a P-time computable function $f$ such that for every $n$, $f(n)$ is a $\PA$ proof of the sentence expressing:
	\begin{center}
		"If an $n$-theory $\Th$ in a language $\Lang$ is consistent then it has a $p(n)$-model."
	\end{center}
	 Moreover there exists a P-time computable function $f$ such that for any $l,k$, $f(l, k)$ is a $\PA$ proof of the sentence expressing:
	\begin{center}
		"If $\M$ is an $l$-model of a $k$-theory $\Th$, then $\Th$ is consistent."
	\end{center}
\end{lemat}
\begin{uwaga}[Uniformity of FACT]
	As a corollary to the proof of the above lemma we will obtain the following proposition:
\begin{stwierdzenie}
	Suppose that $\phi(x,\bar{y})$ is an $\Lang_{\PA}$-formula such that 
	\[\PA\vdash "\num{\phi(x,\bar{y})} \textnormal{ defines a theory }."\]
	($\bar{y}$ are the parameters). Then there exists a formula $\phi'(x,\bar{y})$ such that $\PA\vdash$
	\begin{equation}\tag{$\ACT_{\phi}$}\label{star}
	\textnormal{ "If $\num{\phi}$ is consistent, then $\num{\phi'}$ defines a model for $\num{\phi}$."}
	\end{equation}
	Moreover, there exists a P-time computable function $f$, such that $f(\qcr{\phi})$ is a proof of $\ACT_{\phi}$. 
\end{stwierdzenie}
\end{uwaga}

\begin{proof}[Proof of Lemma \ref{lem::ACT}]
See the Appendix.	 	
\end{proof}

Let us also mention that in addition to FACT we have also the Feasible Compactness Theorem (the proof of which is rather obvious, as we deal here with consistency in the syntactical sense)

\begin{lemat}[Arithmetized Compactness Theorem]
	There exists a P-time computable function $f$ such that for every $n \in \mathbb{N}$, $f(n)$ is a $\PA$ proof of the sentence
	\begin{center}
		"An $n$-theory $\Th$ is consistent if and only if each bounded fragment of $\Th$ is consistent."
	\end{center}
\end{lemat}

	We are now ready to present the second proof of Theorem \ref{lem::main}. We include the statement here for the benefit of the reader.
	
	\begin{tw}[Theorem \ref{lem::main} redux]

Let $\Th$ be a theory extending $\PA$ with an NP-set of axioms. If there is a polynomial $p(n)$ such that for every $n\in\mathbb{N}$, 
 	\begin{equation}\label{equat::Fedor_condition}\tag{$*$}
 	    \PA\vdash^{p(n)}\forall \phi \in \Sent_{\LPA} \left(\dpt(\phi)\leq \num{n} \wedge \Prov_{\Th\restr{\num{n}}}(\phi)\rightarrow \Prov_{\PA\restr{\num{p(n)}}}(\phi)\right).
 	\end{equation}
 	 Then $\PA$ polynomially simulates $\Th$. Moreover, if $\Th$ admits a P-time computable set of axioms and there exists a P-time computable function $f$ such that for all $n\in\mathbb{N}$, $f(n)$ is a $\PA$ proof of
 	 \[\forall \phi \in \Sent_{\LPA} \left(\dpt(\phi)\leq \num{n} \wedge \Prov_{\Th\restr{\num{n}}}(\phi)\rightarrow \Prov_{\PA\restr{\num{p(n)}}}(\phi)\right),\]
 	 then $\Th$ is feasibly reducible to $\PA$.

\end{tw}

 \begin{proof}[Second proof of Theorem \ref{lem::main}, Sketch]
 We will construct a uniformly polynomially correct family of interpretations $\was{I_n}_{n\in\mathbb{N}}: \Th \rightarrow \PA$. Let us define the theory $\Phi_n$:
\begin{equation}
    x \in \Phi_n := (\Tr_{n+1}(x) \vee x\in \Th) \wedge \exists y \bigl( \len(x)\leq y \wedge \Con_{\Th_{\restr{y}}}^{\Tr_{n+1}}\bigr),
\end{equation}
where $\Con_{\Th_{\restr{y}}}^{\Tr_n}$ says that there is no proof of contradiction using as axioms sentences in $\Th\restr{y}$ or true sentences of depth $n$ (see Remark \ref{rem::RelProv} for an explanation). Observe that the length of (the formula defining)
$\Phi_n$ is polynomial in $n$ and its shape depends uniformly on $n$. Then for every $n$, $\PA\vdash \Con_{\Phi_n}$
and the proof is uniform in $n$, so in fact there exists a polynomial $p_1(n)$ such that for every $n$,
\[\PA\vdash^{p_1(n)}\Con_{\Phi_n}.\]
(for the precise argument see \cite{HajekPudlak}, Theorem 2.37). By FACT we know that there exists a formula $\M_{\Phi_n}$ and a polynomial $p_2(n)$ such that
\begin{equation}\label{equat::model}\tag{$**$}
\PA+ \Con_{\Phi_n}\vdash^{p_2(n)} \M_{\Phi_n} \textnormal{ is a full model for }\Phi_n.
\end{equation}
Now $I_n$ is defined as a relativization to $\M_{\Phi_n}$ i.e. a function defined on formulae of the language of $\Th$ which preserves boolean operations such that  for every relational symbol $R$ and all terms $s_1,\ldots,s_n$
\[(R(s_1,\ldots, s_n))^{I_n} = \qcr{R(s_1,\ldots, s_n)} \in \M_{\Phi_n}\]
(recall that full models are coded as elementary diagrams) and for every existential formula $\exists x \phi$,
\[(\exists x \phi)^{I_n} = \exists x \in M_{\Phi_n} \ \ \phi^{I_n}.\]
We check by contraposition that $I_n$ is $n$-correct. Work in $\PA$.  Assume $\neg\phi$, where $\phi$ is of length at most $n$. We will derive $\phi^{I_n}$.
Surely, $\neg \phi$ is of depth at most $n+1$. Let $q(n)$ be as in Theorem \ref{thm::Pudlak}. Then, by provable Tarski biconditionals, with a proof of length $q(n)$ we conclude
\[\Tr_{n+1}(\neg \phi).\]
We show that this implies $\Phi_n(\neg\phi)$. By our main assumption \eqref{equat::Fedor_condition} and Lemma \ref{subl::1} we obtain a polynomial $p_3(n)$ such that 
\begin{displaymath}
\PA\vdash^{O(p_3(n,n))}\Con_{\Th_{\restr{n+1}}}^{\Tr_{n+1}},
\end{displaymath}
which directly implies that $\neg \phi$
belongs to $\Phi_n$.  Consequently
\[\PA\vdash\M_{\Phi_n}\models \neg \phi,\]
and the proof of it has length $O(p_1(n)+p_2(n)+q(n) + p_3(n,n))$. Now using at most $|\phi|$ many steps involving formulae of length polynomial in $n$ we obtain
\[\neg(\phi)^{I_k}.\]

What is left to show is that for some polynomial $p(n,k)$ and each $k\in\mathbb{N}$,
\[\PA\vdash^{p(n,k)}\psi^{I_k},\]
for every sentence $\psi$ in $\Th\restr{n}$.
  We use polynomial binumerability of $\Th\restr{n}$ to guarantee that there is a polynomial $p_4$ such that for all  $n$ and $\psi \in  \Th\restr{n}$, $$\PA\vdash^{p_4(n)} \psi \in \Th\restr{\num{n}}.$$
Now, as previously, we have
$$\PA\vdash^{O(p_3(n,k))} \Con_{\Th\restr{n}}^{\Tr_{k}}.$$
Hence, adding a few more steps, we also have
\[\PA\vdash^{O(p_4(n))}\psi \in \Phi_k.\]
Then, as previously we check that $\psi^{I_k}$ is satisfied.

The "moreover" part holds, since if $\Th$ is P-time computable, then we can feasibly find a $\PA$ proof witnessing that
\[\PA\vdash \phi \in \Th\restr{n}.\]
The rest of steps are fully analogous.
 \end{proof}

\section{Open Questions}

The proofs of our main results in Section 4 suggest that the answers to the following questions are both in the positive; we pose them here as questions since definitive positive answers to them requires a number of technical verifications that are yet to be carried out.


\noindent \textbf{Question A.}  Is the conservativity of $\CT^-[\PA]$, $\KF^-[\PA]$, and $\FS^-[\PA]$ over $\PA$ provable in Buss's system $\mathsf{S}_{2}^{1}?$  


\noindent \textbf{Question B.} Suppose $\B$ is a sequential theory that is inductive; i.e., the scheme of induction over the natural numbers of $\B$ is provable in $\B$. 
Are $\CT^-[\B]$, $\KF^-[\B]$, and $\FS^-[\B]$ 
feasibly reducible to $\B$?

\section{Appendix}
The bulk of the work in the following paper is inherently technical, especially since we are dealing with nuanced arithmetizations and sizes of proofs, and therefore the arguments need to be checked in a very careful manner. In order to minimally distract the reader from the main flow of the argument, we decided to relegate some of the checking to this Appendix.

\subsection{Feasible reflexivity}\label{Feasible reflexivity}

In Section \ref{sec_polynomial_simulations}, we proved Theorem \ref{lem::main} which is the technical core of our paper and which provides us with a uniform way of obtaining polynomial simulations and feasible reductions. We presented two proofs of that theorem. The first of them used the following Lemma (originally Lemma \ref{subl::1}):
\begin{lemat}
There exists a P-time computable function $f$ such that for all $n,k\in\mathbb{N}$ $f(n,k)$ is a $\PA$ proof of the:
		\[\forall \phi\bigl(\dpt(\phi)\leq \num{k} \wedge \Prov_{\PA\restr{\num{n}}}(\phi)\rightarrow \Tr_k(\phi)\bigr).\]
\end{lemat}

It states that we can uniformly find $\PA$ proofs of the uniform reflection for bounded fragments of $\PA$. In the proof, we assumed that the statement holds for the axioms in these fragments, and that arithmetical satisfaction predicates enjoy certain regularity properties. Below, we formulate these results in a precise manner and prove them. 

\begin{definicja}
\begin{enumerate}

\item If $\alpha$ is a valuation and $v$ is in the domain of $\alpha$, then by $\alpha[v \mapsto x]$ we mean a valuation $\alpha'$ which is the same as $\alpha$ except for the variable $v$ whose value is $x$.
\item If $y$ is a formula, then $\Ind(y,v)$ denotes the instantiation of the induction scheme (with parameters) with formula $y$ w.r.t. $v$, i.e. the following formula
	\[\left(y[0/v]\wedge \forall v \left(y\rightarrow y[S(v)/v]\right)\rightarrow \forall v y\right).\]
\end{enumerate}
\end{definicja}

For the sake of simplicity we assume that $\PA$ is axiomatized by induction scheme with free variables treated as parameters. This is in order to avoid taking universal closures of axioms. Let us observe that, living inside $\PA$, we know that every object can be named by a closed term. The Proposition below says that for every formula $\phi$ being satisfied by a sequence $y$ is equivalent to the truth of the sentence $\phi[y]$.\footnote{For the notation $\phi[y]$, recall Definition \ref{defi_substitutions}.} 
\begin{konwencja}
For the sake of simplicity let us agree that saying that for every $n$ the family $\was{\phi_n}_{n\in\mathbb{N}}$ is \emph{uniformly feasible in $n$} means that there exists a P-time computable function $f$ such that for each $n$, $f(n)$ is $\PA$ proof of $\phi_n$.
\end{konwencja}

\begin{stwierdzenie}\label{stw: regularity}
	The following regularity properties for $\Sat_n$ predicates are uniformly feasible in $n$:
	\begin{enumerate}
		\item $\left[\left(\dpt(y) =n\wedge \alpha'=\alpha[v\mapsto S(\alpha(v))]\right)\rightarrow \Sat_n(y,\alpha')\equiv \Sat_n(y[S(v)/v],\alpha)\right]$.
		\item $\left[\left(\dpt(y) = n\wedge \alpha'=\alpha[v\mapsto z]\right)\rightarrow \Sat_n(y,\alpha')\equiv \Sat_n(y[\num{z}/v],\alpha)\right]$.
	\end{enumerate}
\end{stwierdzenie}

\begin{stwierdzenie}[Essentially Pudl\'ak, \cite{pudlak_lc}]\label{stw_pudlak_prawda_logiki}
	The following formulae are uniformly feasible in $n$:
	\begin{enumerate}
		\item $\dpt(x) =n\wedge "x \textnormal{ is a logical axiom }"\wedge \alpha \in \Ass(x)\rightarrow \Sat_n(x,\alpha)$
		\item $\dpt(y) = n \wedge "y\textnormal{ is of the form } x\rightarrow z" \wedge \alpha \in \Ass(y) \wedge \Sat_n(y,\alpha)\wedge \Sat_n(x,\alpha)\rightarrow \Sat_n(z,\alpha)$
	\end{enumerate}
\end{stwierdzenie}

Now we prove that for every $n$ the truth of all $\PA$ axioms of induction of depth $n$ can be feasibly established in $\PA$.

\begin{stwierdzenie}\label{stw_prawda_indukcji}
	The following sentences are uniformly feasible in $n$:
	\[\forall y \in \form^{1}_{\LPA} \forall v \in \vrbl \left(v \in \FV(y) \wedge \dpt(\Ind(y,v)) =n \wedge \alpha \in \Ass(y)\rightarrow \Sat_{n}(\Ind(y),\alpha)\right).\]
\end{stwierdzenie}
\begin{proof}[Proof of Proposition \ref{stw_prawda_indukcji}]
	For the purposes of this proof, we say that $y$ is \emph{small} if $\dpt(\Ind(y,v)) \leq n$. Let $\phi_1(y,v,\alpha)$ abbreviate the following formula: 
	
		\begin{center}
			"$y$ is a small formula such that $v$ is a free variable of $y$ and $\alpha$ is an assignment for $y$"
		\end{center}

	Moreover let $\phi_2(y,v,\alpha,x)$ abbreviate
		\begin{center}
		$ \exists \alpha' \ \ \left(\alpha' = \alpha[v\mapsto x]\wedge \Sat_n(y,\alpha')\right)$,
	\end{center} 
	Let $\phi(x,v,y,\alpha) = \phi_1(y,v,\alpha)\wedge \phi_2(y,v,\alpha,x)$. The idea is that $\alpha$ encodes a sequence of parameters used in the induction and $x$ is the varying value assigned to the variable $v$ while proving $\forall v y$ via induction. We work in $\PA$. We start with $\Ind(\phi(x,v,y,\alpha),x)$ which is an axiom  of length polynomial in $n$
	(since $\phi(x,v,y,\alpha)$ is). Using a few transformations (their number is independent of $n$) we obtain
	\[\forall x,v,y,\alpha\ \ \left(\phi_1(y,v,\alpha) \rightarrow \Ind(\phi_2(y,v,\alpha,x),x)\right)\] 
	Let us look at $\Ind(\phi_2(y,v,\alpha,x),x)$. Observe that by Proposition \ref{stw: regularity}
	\begin{enumerate}
		\item $\phi_2(y,v,\alpha,0)$ is equivalent to $\Sat_n(y[\num{0}/v],\alpha)$ and
		\item $\phi_2(y,v,\alpha,S(x))$ is equivalent to $\Sat_n(y[S(v)/v], \alpha)$.
	\end{enumerate}
	Hence $\Ind(\phi_2(y,v,\alpha,x),x)$ implies
	\begin{multline}\nonumber
	\Sat_n(y[\num{0}/v],\alpha)\wedge \forall x \left(\Sat_n(y,\alpha[v\mapsto x])\rightarrow \Sat_n(y[S(v)/v],\alpha[v\mapsto x])\right)\longrightarrow\\ \longrightarrow \forall x \ \ \Sat_n(y,\alpha[v\mapsto x]).
	\end{multline}
	Now by compositional axioms for $\Sat_n$ the above is equivalent to 
	\[\Sat_n(\Ind(y,v)).\]
\end{proof}

\subsection{Congruence lemma}

We sketch the proof of the following lemma from Section \ref{Section_core_construction}.
\begin{lemat}[Congruence lemma] 
	For all $\phi$, $\phi'$, $\psi'$ it holds that
	\begin{equation} \tag{\textnormal{C}}
	\bigl(\phi\imsubf \phi'\wedge \phi'\approx^{\M} \psi' \bigr)\Rightarrow \exists \psi\ \ \bigl(\psi\imsubf \psi' \wedge\psi\approx^{\M}\phi \bigr).
	\end{equation} 
\end{lemat}
\begin{proof}[Sketch of the proof]
	We prove the lemma by induction on the complexity of $\phi$ (carried out in $\md{M}$ which we assumed to satisfy $\ISigma_1$).
	The only non-trivial step is the one for $\exists$. Assume $\phi' = \exists v \phi$. Then $\psi' = \exists v \psi$. Take $\compl{\phi'}(=\compl{\psi'})$, which, by definition, is of the form $\exists v \eta$. In $\eta$ replace all the occurrences of maximal terms in $\eta$ (i.e. the ones which do not occur within a term) which contain only free variables (in $\eta$) with fresh variables, without using the same variable twice. Then rename the free variables of the resulting formula according to the procedure adopted in condition $4.$ of the definition of the term trivialization. In this way we obtain the term trivialization of both $\psi$ and $\phi$.
\end{proof}

\subsection{FACT} \label{sub_fact}
In Subsection \ref{sub_speed_up_via_FACT}, an alternative proof has been provided of Theorem \ref{lem::main} which says that if $\PA$ proves reflection over fragments of another theory $\Th$, then $\Th$ is feasibly reducible to $\PA$. The second proof of that theorem used the fact that arithmetized completeness theorem can proved with a proof of size polynomial in the size of the formula defining $\Th$. This was stated as Lemma \ref{lem::ACT}. In this subsection, we prove this result.

\begin{lemat}[Feasible Arithmetized Completeness Theorem, FACT]
	There exists a polynomial $p(n)$  and a P-time computable function $f$ such that for every $n$, $f(n)$ is a $\PA$ proof of the sentence
	\begin{center}
		"If an $n$-theory $\Th$ in a language $\Lang$ is consistent then it has a $p(n)$ full model"
	\end{center}
	 Moreover, there exists a P-time computable function $f$ such that for any $l,k$, $f(l, k)$ is a $\PA$ proof of the sentence
	\begin{center}
		"If $\M$ is a $l$-full model of a $k$-theory $\Th$, then $\Th$ is consistent."
	\end{center}
\end{lemat}
\begin{proof}

To prove the second part we show by induction on the lengths of proofs that any $l$-full model (which, recall, is the same as a complete consistent Henkinized theory) is closed under reasoning in first order logic. This argument is carried out uniformly with the only difference that we use different feasible satisfaction predicates depending on the complexity of the model.
	
	To prove the first part we follow the "leftmost branch" strategy. The proof is routine but we present it to be on the safe side.  Assume that an $\Lang$-theory $\theta(x)$ of depth $n$ is consistent. (Note that $\theta(x)$ and $\Lang$ might contain arbitrary parameters.) Note that we need not care about the rise in $\Sigma_k$-complexity of the formula defining a model for $\theta(x)$ as long as the construction of the relevant formula is uniformly feasible in $n$. Let $\form^H_{\Lang}$ be the set of formulae of $\Lang$ enriched with Henkin constants (we denote the Henkin constant for the formula $\phi$ with $c_{\phi}$ and assume that the function $\phi\mapsto c_{\phi}$ is $\Delta_1$).

	\textit{Step $1.$: finding a complete, consistent Henkin extension} Let $\theta'(x)$ be defined as
	\[\theta(x) \vee \exists \phi, v \left(v \in \vrbl \wedge \phi \in \form^H_{\Lang} \wedge  x = (\exists v \phi \rightarrow \phi[c_{\phi}/v])\right).\]
	Here $\theta'(x)$ is of depth polynomial in $n$ (we may assume that the length, hence also the depth, of the formula defining $\Lang$ is polynomial in $n$) and of length polynomial in the length of $\theta$. We check that $\Con_{\theta'(x)}$ holds:   
    each proof of $\exists x (x\neq x)$ from the axioms of $\theta'(x)$ can be transformed into $\theta(x)$ proof of 
	\[\neg \bigwedge_{j\leq a} (\exists v_{i_j} \phi_j \rightarrow \phi_j[c_{\phi}/v_{i_j}]).\] 
	Then we check that the above is equivalent to 
	\[\neg\bigwedge_{j\leq a} (\exists v_{i_j} \phi_j\wedge \forall v_{i_j}\neg\phi_{i_j}).\]
	which contradicts the consistency of $\theta$. Note that the above argument is uniform in $\theta$.
	
	Let $\sigma(x)= y$ be any enumeration of $\form_{\Lang}^H$. For any binary sequence $\tau$ of length $y$ let $\enum(\sigma,\tau,y)$ be the theory:
	\[x\in \enum(\sigma,\tau,y) := \exists i< y \left(\left(\tau(i) = 0\wedge x =\sigma(i)\right) \vee \left(\tau(i) = 1\wedge x = \neg\sigma(i)\right)\right).\] 
	$\enum(\sigma,\tau,y)$ enumerates first $y$ elements of $\sigma$, adding negation at the front of the $i$-th element if $\tau(i)= 1$.
	Let $\theta^H(x)$ be the sentence saying:
	
	\noindent "There exist the unique $y,\tau$ such that
	\begin{enumerate}
	\item $\tau\textnormal{ is a binary sequence of length } y $ and
	\item $\forall i< y  \left(\tau(i) = 0\iff \Con_{\theta' + \enum(\sigma,\tau,i)+\sigma(i)}\right)$ and
	\item $\left(\Con_{\theta' + \enum(\sigma,\tau,y)+\sigma(y)}\wedge x=\sigma(y)\right)\vee \left(\neg\Con_{\theta' + \enum(\sigma,\tau,y)+\sigma(y)}\wedge x=\neg\sigma(y)\right).$"
	\end{enumerate}
	Once again $\theta^H(x)$ is of length polynomial in the length of $\theta$ and this polynomial does not depend on the initial choice of $\theta$. We show that $\theta^H$ is a complete and consistent theory with Henkin sentences (which, according to our definitions is the same as a full model). The whole argument was carried out uniformly in $\theta$ and can be produced by a P-time function $f$.

\end{proof}

\subsection{A glossary of technical notions}
\label{subsection_glossary}

This paper contains a fairly large number of technical definitions; here we enclose a glossary of such terms in order to assist the reader.

\begin{itemize}

\item $x \in \Ass(y)$ means that  $x$ is an assignment for a formula or a term $y$ (or for a set of terms or formulae $y$), i.e. $x$ is a function whose domain includes the free variables of $y$ (or whose domain includes free variables of all elements of $y$). See Definition \ref{konw_technical definitions} and Convention \ref{konwencja_slopp_syntax}.
 \item $x \in \ClTerm_{\Lang}$ means that $x$ is a closed term of a language $\Lang$, see Definition \ref{konw_technical definitions} and Convention \ref{konwencja_slopp_syntax}.
 \item $x \in \ClTermSeq_{\Lang}$ means that $x$ is a sequence of closed terms of a language $\Lang$, see Definition \ref{konw_technical definitions} and Convention \ref{konwencja_slopp_syntax}.
 
\item $\Con_{\Th}$ is an arithmetized consistency statement for $\Th$. See Corollary \ref{thm::PolyProv}.
    \item $\CT^-$ is the compositional theory of truth over $\PA$, see Definition \ref{defi_ctminus}.
    \item $\CT^-[\B]$ is the compositional theory of truth over a theory $\B$ extending $\IDelta_0 + \Exp$, see Definition \ref{defi_ctminus}.
    \item $\dpt(\phi)$ is the syntactic depth of a formula $\phi$, see Definition \ref{defi_syntactic_depth}. 
    \item $\ElDiag(\md{M})$ (elementary diagram of $\md{M}$) is the same as a full model $\md{M}$, this notation is used when $\md{M}$ is viewed as a complete Henkinized theory, rather than a structure.
\item $x \in \form_{\Lang}$ means that $x$ is a formula of the language $\Lang$, see Definition \ref{konw_technical definitions} and Convention \ref{konwencja_slopp_syntax}.
\item $x \in \form^{\leq 1}_{\Lang}$ means that $x$ is a formula of the language $\Lang$ with at most one free variable, see Definition \ref{konw_technical definitions} and Convention \ref{konwencja_slopp_syntax}. 
\item $x \in \form^{1}_{\Lang}(x)$  means that $x$ is a formula of the language $\Lang$ with exactly one free variable, see Definition \ref{konw_technical definitions} and Convention \ref{konwencja_slopp_syntax}.
    \item $\FS^-$ is the Friedman--Sheard self-referential theory of truth over $\PA$ without induction, see Definition \ref{defi::FS}.
    \item $\FS^-[\B]$ is  the Friedman--Sheard self-referential theory of truth over a theory $\B$ extending $\IDelta_0 + \Exp$ without induction, see Definition \ref{defi::FS}.
    
\item $x \in \FV(y)$ means that $x$ is a free variable of $y$, see Definition \ref{konw_technical definitions} and Convention \ref{konwencja_slopp_syntax}.
\item $x \in \FVSeq(y)$ means that $y$ is a coded sequence whose elements are (some) free variables of $x$, see Definition \ref{konw_technical definitions} and Convention \ref{konwencja_slopp_syntax}.
    \item $\KF^-$ is the Kripke--Feferman self-referential theory of truth over $\PA$ without induction, see Definition \ref{defi::KF}.
    \item $\KF^-[\B]$ is the Kripke--Feferman self-referential theory of truth over a theory $\B$ extending $\IDelta_0 + \Exp$ without induction, see Definition \ref{defi::KF}.
    \item $\len(s)$ is the length of a sequence $s$, see Definition \ref{konw_technical definitions} and Convention \ref{konwencja_slopp_syntax}.
    \item $n$-model (full model) is a  (full) model defined with a formula of depth $n$.
    \item $\Prov_{\Th}(y)$ means that there exists a proof of $y$ in the theory $\Th$. See Corollary \ref{thm::PolyProv}.
    \item  $\Proof_{\Th}(m,n)$ means that $m$ is a proof of $n$ from the theory $\Th$. See Corollary \ref{thm::PolyProv}.
    \item $x \in \Sent_{\Lang}$ means that $x$ is a sentence of the language $\Lang$, see Definition \ref{konw_technical definitions} and Convention \ref{konwencja_slopp_syntax}.
    \item $x \in \term_{\Lang}$ means that $x$ is a term of the language $\Lang$, see Definition \ref{konw_technical definitions} and Convention \ref{konwencja_slopp_syntax}. 
   
\item $x \in \TermSeq_{\Lang}$ means that $x$ is a sequence of terms of the language $\Lang$, see Definition \ref{konw_technical definitions} and Convention \ref{konwencja_slopp_syntax}. 
\item $n$-theory is a theory defined with a formula of depth $n$.
\item $\vrbl(x)$ means that $x$ is (an arithmetized) variable, see \ref{konw_technical definitions}.

\item $\beta \sim_v \alpha$ means $\alpha$ and $\beta$ are functions, $v$ is a variable and $\alpha(w) = \beta(w)$ for all variables $w$, possibly except for $v$ which also possibly belongs only to the domain of $\beta$, see Definition \ref{konw_technical definitions}.
\item $\phi[\alpha]$ is a formula $\phi$ with the numeral $\num{\alpha(v)}$ substitued for every every occurrence of $v$ for every free variable $v$ of $\phi$, see Definition \ref{defi_substitutions}.
\item $\phi[s/v]$ denotes the formula $\phi$ with the term $s$ substituted for the variable $v$, see Definition \ref{defi_substitutions}.
\item $\norm{\phi}{\Th}$ is the length of the shortest proof of $\phi$ in $\Th$, see Definition \ref{def_lengths_of_proofs}.
\item $\phi \imsubf \psi$ means that $\phi$ is an immediate subformula of $\psi$, see the proof of Lemma \ref{lem::AEV}.
\item $\compl{\phi}$ is the term trivialization of $\phi$, see remarks preceding Lemma \ref{lem_congruence}.
\item $\md{M} \preceq_{\Lang} \md{N}$ means that $\md{M}' \preceq \md{N}'$, where $\md{M}', \md{N}'$ are reducts of $\md{M}, \md{N}$ to the language $\Lang$, see Definition \ref{def_L_elementary_extension} or  \ref{def_elementary_extension} and the subsequent remarks.
\item $t^{\alpha}$ is the value of term $t$ in which every free variable $v$ has been evaluated to $\alpha(v)$, see Definition \ref{defi_substitutions}.

\item $\Th \vdash^{n} \phi$ means that the length of the shortest proof of $\phi$ in $\Th$ is not greater than $n$, see Definition \ref{def_lengths_of_proofs}.
\item $\Th \restr n$ for a theory $\Th$ means the set of axioms of $\Th$ of length at most $n$, see Definition \ref{def_restriction_of_a_theory}.
\item $\num{x}$ is a numeral denoting $x$, see Definition \ref{defin:arithm} and Convention \ref{konwencja_slopp_syntax}.
\item $\val{x}=y$ means that $y$ is the value of the term $x$, see Definition \ref{konw_technical definitions} and Convention \ref{konwencja_slopp_syntax}.

\end{itemize}

\bibliographystyle{plain}
\bibliography{Speed_up} 

\end{document}